\documentclass[11pt]{amsart}
\usepackage{amsmath, amssymb, amscd, mathrsfs, url, pinlabel,hyperref, verbatim}
\usepackage[margin=1.25in,marginparwidth=1in,centering,letterpaper,dvips]{geometry}
\usepackage{color,dcpic,latexsym,graphicx,epstopdf,comment}
\usepackage[all]{xy}
\usepackage{tikz,pgfplots}
\usepackage{enumitem,upquote,tabularx,textcomp}
\usepackage[all]{hypcap}

\title{Framed instanton homology and concordance}

\author{John A. Baldwin}
\email{john.baldwin@bc.edu}
\address{Department of Mathematics\\Boston College}

\author{Steven Sivek}
\email{s.sivek@imperial.ac.uk}
\address{Department of Mathematics\\Imperial College London}

\thanks{JAB was supported by NSF CAREER Grant DMS-1454865.}

\makeatletter
\newtheorem*{rep@theorem}{\rep@title}
\newcommand{\newreptheorem}[2]{%
\newenvironment{rep#1}[1]{%
 \def\rep@title{#2 \ref{##1}}%
 \begin{rep@theorem}}%
 {\end{rep@theorem}}}
\makeatother

\newtheorem {theorem}{Theorem}
\newreptheorem{theorem}{Theorem}
\newtheorem {lemma}[theorem]{Lemma}
\newtheorem {proposition}[theorem]{Proposition}
\newtheorem {corollary}[theorem]{Corollary}
\newtheorem {conjecture}[theorem]{Conjecture}

\newtheorem {question}[theorem]{Question}

\numberwithin{equation}{section}
\numberwithin{theorem}{section}

\theoremstyle{definition}
\newtheorem{definition}[theorem]{Definition}

\newtheorem{remark}[theorem]{Remark}
\newtheorem*{remark*}{Remark}

\newtheorem{example}[theorem]{Example}

\newlist{pcases}{enumerate}{1}
\setlist[pcases]{
  label=\bf{Case~\arabic*:}\protect\thiscase.~,
  ref=\arabic*,
  align=left,
  labelsep=0pt,
  leftmargin=0pt,
  labelwidth=0pt,
  parsep=0pt
}
\newcommand{\case}[1][]{%
  \if\relax\detokenize{#1}\relax
    \def\thiscase{}%
  \else
    \def\thiscase{~#1}%
  \fi
  \item
}

\newcommand{\Z}{\mathbb{Z}}
\newcommand{\R}{\mathbb{R}}
\newcommand{\C}{\mathbb{C}}

\newcommand{\F}{\mathbb{F}}
\newcommand{\Q}{\mathbb{Q}}
\newcommand{\spc}{\operatorname{Spin}^c}

\newcommand{\ltb}{\mathit{tb}}
\newcommand{\lsl}{\mathit{sl}}
\newcommand{\maxtb}{\overline{\ltb}}
\newcommand{\maxsl}{\overline{\lsl}}
\newcommand{\rank}{\operatorname{rank}}

\newcommand\cC{\mathcal{C}}

\newcommand{\cT}{\mathcal{T}}

\newcommand\hfk{\mathit{HFK}}
\newcommand\hfkhat{\widehat{\hfk}}

\newcommand\SHItfun{\mathit{SHI}}

\newcommand\SHItfunn{\textbf{\textup{\underline{SHI}}}} 

\newcommand\iinvt{\Theta} 

\DeclareFontFamily{U}{mathx}{\hyphenchar\font45}
\DeclareFontShape{U}{mathx}{m}{n}{
      <5> <6> <7> <8> <9> <10>
      <10.95> <12> <14.4> <17.28> <20.74> <24.88>
      mathx10
      }{}
\DeclareSymbolFont{mathx}{U}{mathx}{m}{n}
\DeclareFontSubstitution{U}{mathx}{m}{n}
\DeclareMathAccent{\widecheck}{0}{mathx}{"71}

\newcommand{\inr}{\operatorname{int}}

\newcommand{\hmtilde}{\widetilde{\mathit{HM}}}
\newcommand{\hfhat}{\widehat{\mathit{HF}}}

\newcommand{\pt}{\mathrm{pt}}
\newcommand{\PD}{\mathit{PD}}

\newcommand{\mirror}[1]{\overline{#1}}
\newcommand{\cinvt}{\nu^\sharp}
\newcommand{\chominvt}{\tau^\sharp}

\newcommand{\Kh}{\mathrm{Kh}}
\newcommand{\Khr}{\overline{\Kh}}
\newcommand{\Khodd}{\Kh'}
\newcommand{\Khoddr}{\overline{\Khodd}}
\newcommand{\dcover}{\Sigma_2}

\usetikzlibrary{calc,intersections}
\tikzset{every picture/.style=thick}
\tikzset{link/.style = { white, double = black, line width = 1.75pt, double distance = 1.25pt, looseness=1.75 }}
\tikzset{crossing/.style = {draw, circle, dotted, minimum size=0.5cm, inner sep=0, outer sep=0}}
\pgfplotsset{compat=1.12}

\begin{document}

\begin{abstract}
We define two concordance invariants of knots using framed instanton homology. These invariants $\cinvt$ and $\chominvt$ provide bounds on slice genus and maximum self-linking number, and the latter is a concordance homomorphism which agrees in all known cases with the $\tau$ invariant in Heegaard Floer homology. We use $\cinvt$ and $\chominvt$ to compute the framed instanton homology of all nonzero rational Dehn surgeries on: $20$ of the $35$ nontrivial prime knots through $8$ crossings, infinite families of twist and pretzel knots, and instanton L-space knots; and of 19 of the first 20 closed hyperbolic manifolds in the Hodgson--Weeks census. In another application, we determine when the cable of a knot is an instanton L-space knot. Finally, we discuss applications to the spectral sequence from odd Khovanov homology to the framed instanton homology of branched double covers, and to the behaviors of $\chominvt$ and $\tau$ under genus-2 mutation.
\end{abstract}

\maketitle

\section{Introduction} \label{sec:intro}
Of the four main Floer-theoretic invariants of $3$-manifolds, it is now known that Heegaard Floer homology, monopole Floer homology, and embedded contact homology are isomorphic, while the connections between these three and instanton Floer homology remain elusive.

We will consider a version of the latter called \emph{framed instanton homology}. This invariant, introduced by Kronheimer and Mrowka in \cite{km-yaft}, assigns to any closed, oriented $3$-manifold $Y$ an abelian group $I^\#(Y)$, defined as the instanton Floer homology of an admissible bundle on $Y\# T^3$. Kronheimer and Mrowka conjectured in \cite[Conjecture 7.24]{km-excision} that this theory is isomorphic to the \emph{hat} version of Heegaard Floer homology, which is  in turn isomorphic to the \emph{tilde} and \emph{hat} versions of monopole Floer  and embedded contact homology, respectively:
\begin{equation}\label{eqn:conj-iso}I^\#(Y)\cong_{\textrm{conj.}}\hfhat(Y)\cong\hmtilde(Y)\cong\widehat{\mathit{ECH}}(-Y).\end{equation}
This conjecture has only been verified for a small subset of 3-manifolds,  due to the difficulty of computing $I^\#$. Indeed, framed instanton homology has only been computed  for: connected sums of $S^1\times S^2$ \cite{scaduto}, branched double covers of two-fold quasi-alternating links \cite{scaduto,scaduto-stoffregen}, some Brieskorn  spheres including $\Sigma(2,3,6k\pm 1)$ \cite{scaduto}, a surface times a circle \cite{chen-scaduto}, and some Dehn surgeries on instanton L-space knots \cite{lpcs,bs-lspace}.
One of the main goals of our paper is to expand this list. 

We define  two concordance invariants of knots using framed instanton homology, and use these to compute the framed instanton homology (over $\C$) of:
\begin{itemize}
\item all nonzero rational Dehn surgeries on: $20$ of the $35$ nontrivial prime knots through $8$ crossings,  infinite families of twist and pretzel knots, and instanton L-space knots;  
\item 19 of the first 20 closed hyperbolic manifolds in the Hodgson--Weeks census.
\end{itemize} 
We also verify a version of the conjectured isomorphism \eqref{eqn:conj-iso} in these cases. 

We remark that the   computations above  are in some sense  a proof-of-concept, in that we expect our methods to enable computation of framed instanton homology for surgeries on many more infinite families of knots.

 In addition to these results, we determine when the cable of a knot is an instanton L-space knot, and we discuss applications  to (1) the spectral sequence from odd Khovanov homology to the framed instanton homology of branched double covers, and (2) the behaviors of our concordance invariants and the Heegaard Floer tau invariant  under genus-2 mutation. We outline our main constructions and results in detail below.

\subsection{Concordance and framed instanton homology}
We introduce two concordance invariants $\cinvt$ and $\chominvt$ using framed instanton homology with $\C$ coefficients. The first assigns to a knot $K$ the integer $\cinvt(K)$ defined by \[\cinvt(K):=N(K)-N(\mirror{K}),\] where $N(K)$ is the smallest nonnegative integer $n$ for which the  cobordism map \begin{align*}
I^\#(S^3;\C)&\to I^\#(S^3_n(K);\C)
\end{align*} induced by the trace of $n$-surgery vanishes. 
We prove in Theorem~\ref{thm:conc-invt} that $\cinvt(K)$ depends only on the smooth concordance class of $K$, and  satisfies the smooth slice genus bound \[|\cinvt(K)| \leq \max(2g_s(K)-1,0).\] 
It will be unsurprising from the definition  that  $\cinvt$ is relevant to computing framed instanton homology of Dehn surgeries. We prove the following result in this vein, via extensive analysis of surgery exact triangles (see Theorem~\ref{thm:rational-surgeries} for a slightly stronger statement):

\begin{theorem}
\label{thm:main-surgery}
For every knot $K\subset S^3$,  there is an integer $r_0(K) \geq |\cinvt(K)|$ such that
\[ \dim_\C I^\#(S^3_{p/q}(K);\C) = q\cdot r_0(K) + |p - q\cinvt(K)| \]
for all nonzero rational  $p/q$ with $p$ and $q$ relatively prime and $q \geq 1$.  If $\cinvt(K) \neq 0$ then the same is true for $p/q=0/1$.
\end{theorem}

\begin{remark}
The integer $r_0(K)$ is defined in most cases (e.g., when $\cinvt(K)\neq 0$) by  \[r_0(K):=\dim_\C I^\#(S^3_{\cinvt(K)}(K);\C).\] See Definition \ref{def:shape} for   details.
\end{remark}

\begin{remark}
\label{rmk:nufromsurgery}
Note from Theorem \ref{thm:main-surgery} that $\cinvt(K)$ and $r_0(K)$ are determined by the pair \[\dim_\C I^\#(S^3_{N}(K);\C) \quad\textrm{and}\quad\dim_\C I^\#(S^3_{-N}(K);\C)\] for sufficiently large integers $N$. Similarly, \[\dim_\C I^\#(S^3_{p/q}(K);\C)\] is determined for \emph{all} nonzero rationals $p/q$ by $\cinvt(K)$ and the value of \[\dim_\C I^\#(S^3_{r/s}(K);\C)\] for \emph{any} nonzero rational number $r/s$.
\end{remark}

\begin{remark}
\label{rmk:euler}
Framed instanton  homology comes with  an absolute $\Z/2\Z$-grading. We note that  $\dim_\C I^\#(Y;\C)$ determines this grading since the associated Euler characteristic is \[\chi (I^\#(Y;\C))=\begin{cases}|H_1(Y;\Z)|,& b_1(Y)=0\\
0,& b_1(Y)>0
\end{cases}\] as proven by Scaduto \cite{scaduto}.  In particular, the formula in Theorem \ref{thm:main-surgery}  completely determines the graded vector space $I^\#(S^3_{p/q}(K);\C)$.
\end{remark}

\begin{remark}
One should compare Theorem \ref{thm:main-surgery} with \cite[Proposition 13]{hanselman-cosmetic}, proved using the immersed curves formulation of bordered Floer homology; and \cite[Proposition~9.6]{osz-rational}, proved earlier using the rational surgery formula in knot Floer homology.
\end{remark}

Given Theorem \ref{thm:main-surgery} and our interest in computing framed instanton homology of surgeries, we would like tools for computing $\cinvt$. This partially motivates the definition of our second concordance invariant $\chominvt$. We  prove in Theorem~\ref{thm:nu-quasi-hom} that $\cinvt$ defines a quasi-morphism from the smooth concordance group $\cC$ to $\Z$, and  define  \[\chominvt(K) := \frac{1}{2}\lim_{n\to\infty} \frac{\cinvt(\#^n K)}{n} \] to be one-half its homogenization. This  concordance invariant then defines a homomorphism \[\chominvt:\cC\to \R\] which satisfies the slice genus bound (Proposition \ref{prop:tau-invariant}) \begin{equation}\label{eqn:tauslice}|\chominvt(K)|\leq g_s(K).\end{equation} Moreover, we use the contact invariant we defined in \cite{bs-instanton} to prove that this concordance homomorphism gives rise to a \emph{slice-torus invariant}, as defined by Lewark \cite{lewark} following Livingston \cite{livingston-tau}; more precisely:

\begin{theorem}\label{thm:slice-torus}
 The concordance homomorphism $2\chominvt$ is a slice-torus invariant.
 \end{theorem} 

We  use this  to prove in Theorem \ref{thm:sl-bound} that $\cinvt$ and $\chominvt$ bound maximum self-linking number, \begin{equation} \label{eqn:sl} \maxsl(K) \leq 2\chominvt(K)-1 \leq \cinvt(K). \end{equation} The inequalities \eqref{eqn:tauslice} and \eqref{eqn:sl}, together with the fact that $\maxsl(K) = 2g_s(K)-1$ for quasipositive knots $K$, then immediately imply:

\begin{corollary}
\label{cor:chom-qp}
If $K$ is a quasipositive knot then $\chominvt(K) = g_s(K)$.
\end{corollary}

Lewark proves in \cite{lewark} that the values of any two slice-torus invariants agree on homogeneous knots. Since $2\tau$ is also a slice-torus invariant, where $\tau(K)\in\Z$ is the tau invariant in Heegaard Floer homology, it follows that:

\begin{corollary}
\label{cor:homogeneous}
If $K$ is a homogeneous knot then $\chominvt(K)=\tau(K)$.\end{corollary}

\begin{remark}
\label{rmk:tau-tau2}
Note that $\chominvt = \tau$ for quasipositive knots as well, since both equal $g_s$ in this case. It follows that $\chominvt = \tau$ for all prime knots through 9 crossings, except possibly for  $9_{42}$, $9_{44}$, $9_{48}$; see Corollary \ref{cor:homogeneous2}.
\end{remark}

Lewark also shows that slice-torus invariants are equal to minus  signature for alternating knots, so we have:

\begin{corollary}
\label{cor:alternating}
If $K$ is an alternating knot then $\chominvt(K) = -\sigma(K)/2$.\footnote{In our convention, the signature of the right-handed trefoil is $-2$.}\end{corollary} 

\begin{remark}
Ghosh, Li, and Wong \cite{glw} have defined a tau-like concordance invariant using instanton Floer homology, by very different means, and proved it equal to $\chominvt$.  Since their invariant is an integer by construction, this gives the only proof we know of that $\chominvt(K) \in \Z$ for all knots $K$.
\end{remark}

These corollaries enable us to compute $\chominvt$ for all prime knots through 8 crossings.
We use these computations of $\chominvt$ to determine $\cinvt$ for all such knots except $7_7$ and $8_{13}$ (\S\ref{sec:computations-dim}, Table~\ref{table:small-values}).
We then use these values of $\cinvt$, together with Theorem~\ref{thm:main-surgery} and lots of Kirby calculus, to determine $r_0(K)$ and thus the framed instanton homology (over $\C$) for all nonzero rational  surgeries on 20 of the 35 nontrivial prime knots through 8 crossings:

\begin{theorem}
\label{thm:8-crossings}
For any of the knots $K$ listed in Table~\ref{table:main}, the values $r_0(K)$ and $\cinvt(K)$ are as shown in the table. 
\end{theorem} 
\begin{table}
\centering
\begin{tabular}{ccc}
$K$ & $\cinvt(K)$ & $r_0(K)$ \\
\hline
$3_1$ & $-1$ & $1$ \\
$4_1$ & $0$ & $2$ \\
$5_1$ & $-3$ & $3$ \\
$5_2$ & $-1$ & $3$ \\
$6_1$ & $0$ & $4$ \\
$6_2$ & $-1$ & $5$ \\
$6_3$ & $0$ & $6$ \\
$7_1$ & $-5$ & $5$ \\
$7_2$ & $-1$ & $5$ \\
$7_3$ & $3$ & $7$ \\
$7_4$ & $1$ & $7$ \\
\hline
\end{tabular}
\hspace{0.5cm}
\begin{tabular}{ccc}
$K$ & $\cinvt(K)$ & $r_0(K)$ \\
\hline
$8_1$ & $0$ & $6$ \\
$8_2$ & $-3$ & $9$ \\
$8_3$ & $0$ & $8$ \\
$8_4$ & $-1$ & $9$ \\
$8_5$ &  $3$ & $11$ \\
$8_6$ & $-1$ & $11$ \\
$8_8$ & $0$ & $12$ \\
$8_{19}$ & $5$ & $5$ \\
$8_{20}$ & $0$ & $4$ \\
\\
\\
\hline
\end{tabular}
\caption{Values which determine $\dim_\C I^\#(S^3_r(K);\C)$ for  rational $r$, with the restriction that if $\cinvt(K)=0$ then $r\neq 0$, by Theorem~\ref{thm:main-surgery}.} \label{table:main}
\end{table}

We use similar strategies to compute the framed instanton homology  (over $\C$) of surgeries on infinite families of twist knots and pretzel knots, as in the theorems below.

\begin{theorem}
\label{thm:main-twist}
Let $K_n$ be the twist knot with a positive clasp and $n \geq 1$ positive half-twists shown in Figure~\ref{fig:main-twist}.  Then
\[ r_0(K_n)=n \quad\textrm{ and } \quad\cinvt(K_n) = \begin{cases} 0, & n\textrm{ even} \\ -1, & n\textrm{ odd}. \end{cases} \] Suppose $p$ and $q$ are relatively prime integers with $q \geq 1$. Then  \[ \dim_\C I^\#(S^3_{p/q}(K_n);\C) = \begin{cases}
qn+|p|, &n\textrm{ even and }p\neq 0,\\
qn+|p+q|, &n\textrm{ odd}.
\end{cases}\]

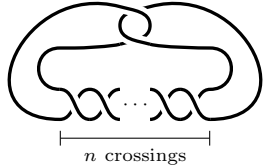
\begin{figure}
\begin{tikzpicture}
\draw[link] (0.2,0) to[out=0,in=180,looseness=1] ++(0.4,0.4) to[out=0,in=180,looseness=1] ++(0.4,-0.4) to [out=0,in=180] (1,0) to[out=0,in=15,looseness=2] (0.2,1.5) to[out=195,in=90,looseness=1] (-0.2,1.25) to[out=270,in=165,looseness=1] (0.2,1) to[out=345,in=180,looseness=0.5] (1,0.95) to[out=0,in=0] (1,0.4) to [out=180,in=0,looseness=1] ++(-0.4,-0.4) to[out=180,in=0,looseness=1] ++(-0.4,0.4) (-0.2,0.4) to[out=180,in=0,looseness=1] (-0.6,0) to[out=180,in=0,looseness=1] (-1,0.4) to[out=180,in=180] (-1,0.95) to[out=0,in=180] (-0.8,0.95) to[out=0,in=195,looseness=0.5] (-0.2,1);
\draw[link] (-0.2,1) to[out=15,in=270,looseness=1] (0.2,1.25) to[out=90,in=345,looseness=1] (-0.2,1.5) to[out=165,in=180,looseness=2] (-1,0) to[out=0,in=180,looseness=1] (-0.6,0.4) to[out=0,in=180,looseness=1] (-0.2,0);
\node at (0.03,0.2) {\tiny $\cdots$};
\begin{scope} \clip (-0.2,0.8) rectangle (0.2,1.2);
\draw[link] (-0.2,1.25) to[out=270,in=165,looseness=1] (0.2,1);
\end{scope}
\begin{scope} \clip (-1,-0.1) rectangle (-0.6,0.5);
\draw[link] (-0.6,0) to[out=180,in=0,looseness=1] (-1,0.4);
\end{scope}
\begin{scope} \clip (0,0) rectangle (1,0.4);
\draw[link] (0.2,0.4) to[out=0,in=180,looseness=1] ++(0.4,-0.4) ++(0,0.4) to[out=0,in=180,looseness=1] ++(0.4,-0.4);
\end{scope}
\draw[thin,|-|] (-1,-0.25) to node[midway,below] {\tiny $n \mathrm{\ crossings}$} ++(2,0);
\end{tikzpicture}
\caption{The twist knot $K_n$.}
\label{fig:main-twist}
\end{figure}
\end{theorem}

\begin{remark}
Prior to this theorem, the framed instanton homology \[I^\#(S^3_{p/q}(K_n);\C)\] was known only for  $p/q=\pm 1$, by Scaduto's computation of $I^\#$ for the Brieskorn spheres $\Sigma(2,3,6n\pm 1)$ \cite{scaduto}.
\end{remark}

\begin{theorem}
\label{thm:main-pretzels} 
Let $P(a,b,c)$ denote the $(a,b,c)$ pretzel knot. 
Then  \[ \cinvt(P(n,3,-3)) = 0 \quad\textrm{and}\quad r_0(P(n,3,-3)) = 4 \] for all $n\in\Z$.
Suppose $p$ and $q$ are relatively prime integers with $q \geq 1$ and $p\neq 0$. Then \[ \dim_\C I^\#(S^3_{p/q}(P(n,3,-3));\C) = 4q + |p|. \] 
\end{theorem}

\begin{remark}
Since each $P(n,3,-3)$ is smoothly slice, we also know that
\[ \dim_\C I^\#(S^3_0(P(n,3,-3));\C) = 4q+2. \]
Indeed, Theorem~\ref{thm:main-surgery} does not apply a priori to 0-surgery on $K$ when $\cinvt(K)=0$, but if $K$ is smoothly slice then we have
\[ \dim_\C I^\#(S^3_0(K);\C) = \dim_\C I^\#(S^3_1(K);\C) + 1 \]
because Theorem~\ref{thm:conc-invt} guarantees that $K$ is ``W-shaped'', as discussed in \S\ref{sec:nu-sharp}.  By contrast, the twist knots $K_{2m}$ from Theorem~\ref{thm:main-twist} have $\cinvt(K_{2m})=0$ but are not generally slice, so we do not make any claims about $\dim_\C I^\#(S^3_0(K_{2m});\C)$.
\end{remark}

\begin{theorem}
\label{thm:main-pretzels-2} 
Let $P(a,b,c)$ denote the $(a,b,c)$ pretzel knot. 
Then \[ \cinvt(P(2n-1,3,2)) = 2n-1 \quad\textrm{and}\quad r_0(P(2n-1,3,2)) = 6n-1 \] for all $n \geq 1$.
Suppose $p$ and $q$ are relatively prime integers with $q \geq 1$. Then \[ \dim_\C I^\#(S^3_{p/q}(P(2n-1,3,2));\C) = (6n-1)q + |p-(2n-1)q|. \] 
\end{theorem}

We further compute framed instanton Floer homology (over $\C$) for all Dehn surgeries on instanton L-space knots. Recall that a nontrivial knot $K$ is an \emph{instanton L-space knot} if \[\dim_\C I^\#(S^3_{n}(K);\C) = n\] for some  integer $n>0$. These include  all nontrivial knots  with positive lens space surgeries, like positive torus knots; or, more generally, knots with positive surgeries that are branched double covers of quasi-alternating links \cite[Corollary~1.2]{scaduto}. In previous work \cite{bs-lspace}, we implicitly computed $\cinvt$ for such knots (using some results from \cite{lpcs}); in combination with Theorem \ref{thm:main-surgery}, this gives:

\begin{theorem}
\label{thm:nu-l-space}
Let $K$ be an instanton L-space knot. Then \[\cinvt(K) = r_0(K) = 2g(K)-1.\] Suppose $p$ and $q$ are relatively prime integers with $q\geq 1$. Then 
\[\dim_\C I^\#(S^3_{p/q}(K);\C)= \begin{cases}
p, & p/q \geq 2g(K)-1 \\
2q(2g(K)-1) - p, & p/q < 2g(K)-1.
\end{cases}\]
\end{theorem}

\begin{remark} Theorem \ref{thm:nu-l-space} extends one of the main results of Lidman--Pinz{\'o}n-Caicedo--Scaduto \cite{lpcs}, which gave the   framed instanton homology (over $\C$) for \emph{integer} surgeries on instanton L-space knots.
\end{remark}

In a related vein, we use Theorem \ref{thm:main-surgery} to determine exactly when the cable of a knot is an instanton L-space knot, in analogy with the results of Hedden and Hom  in Heegaard Floer homology \cite{hedden-cabling-2,hom-cabling}: 

\begin{theorem}
\label{thm:l-space-cable}
The cable $K_{p,q}$ is an instanton L-space knot if and only if $K$ is an instanton L-space knot and $p/q > 2g(K)-1$.
\end{theorem}

Finally, we  use the techniques  above to determine the framed instanton homology (over $\C$) of many of the smallest volume  hyperbolic 3-manifolds:

\begin{theorem}
\label{thm:hyperbolic} The framed instanton homology of the first 20 manifolds in the Hodgson--Weeks census of closed hyperbolic manifolds is as shown in Table~\ref{table:hw-census}.
\end{theorem}

\begin{table}
\centering
\begin{tabular}{clcc}
\# & Name & $|H_1|$ & $\dim I^\#$ \\
\hline
$0$ & \texttt{m003(-3,1)} & $25$ & $25$ \\
$1$ & \texttt{m003(-2,3)} & $5$ & $7$ \\
$2$ & \texttt{m007(3,1)} & $18$ & $18$ \\
$3$ & \texttt{m003(-4,3)} & $25$ & $25$ \\
$4$ & \texttt{m004(6,1)} & $6$ & $8$ \\
$5$ & \texttt{m004(1,2)} & $1$ & $5$ \\
$6$ & \texttt{m009(4,1)} & $6$ & $8$ \\
$7$ & \texttt{m003(-3,4)} & $10$ & $10$ or $12$ \\
$8$ & \texttt{m003(-4,1)} & $35$ & $35$ \\
$9$ & \texttt{m004(3,2)} & $3$ & $7$ \\
\hline
\end{tabular}
\hspace{0.5cm}
\begin{tabular}{clcc}
\# & Name & $|H_1|$ & $\dim I^\#$ \\
\hline
$10$ & \texttt{m004(7,1)} & $7$ & $9$ \\
$11$ & \texttt{m004(5,2)} & $5$ & $9$ \\
$12$ & \texttt{m003(-5,3)} & $35$ & $35$ \\
$13$ & \texttt{m007(1,2)} & $21$ & $21$ \\
$14$ & \texttt{m007(4,1)} & $21$ & $21$ \\
$15$ & \texttt{m007(3,2)} & $27$ & $27$ \\
$16$ & \texttt{m006(3,1)} & $30$ & $30$ \\
$17$ & \texttt{m003(-5,4)} & $30$ & $30$ \\
$18$ & \texttt{m006(-3,2)} & $15$ & $15$ \\
$19$ & \texttt{m015(5,1)} & $7$ & $9$ \\
\hline
\end{tabular}
\caption{Framed instanton homology of the first 20 closed manifolds in the Hodgson--Weeks census (we could not  pin down $I^\#$ for census manifold 7).} \label{table:hw-census}
\end{table}

\subsection{The spectral sequence from odd Khovanov homology} Let $L$ be any link in $S^3$. In \cite{scaduto}, Scaduto constructed a spectral sequence \[\Khoddr(L) \Rightarrow I^\#(-\dcover(L)) \] with $E^2$ page  the reduced odd Khovanov homology of $L$,  converging to the framed instanton homology of the branched double cover of $S^3$ along $L$ (with its orientation reversed). This spectral sequence collapses whenever the bigraded group \[\Khoddr(L)=\bigoplus_{i,j} \Khoddr_{\!i,j}(L)\] is \emph{thin}, meaning that there is some $\delta\in \Z$ such that \[ \Khoddr_{\!i,j}(L) = 0 \textrm{ unless } j-2i = \delta, \] as is the case for quasi-alternating links. Interestingly, it appears to be   an open  question whether the converse is true (though we could not find this question or its analogue for the spectral sequence in Heegaard Floer homology \cite{osz-branched} posed anywhere):

\begin{question}
Is there a link $L\subset S^3$ for which $\Khoddr(L)$ is not thin but the spectral sequence \[\Khoddr(L) \Rightarrow I^\#(-\dcover(L))\] collapses at the $E^2$ page?
\end{question}

In this vein, we prove the following:

\begin{theorem}

\label{thm:main-ss}
With coefficients in $\C$, the spectral sequence \[\Khoddr(K;\C) \Rightarrow I^\#(-\dcover(K);\C)\] does not collapse  at the $E^2$ page for any of the  prime knots with at most 10 crossings whose reduced odd Khovanov homology over $\C$ is not thin, except possibly for $10_{152}$.
\end{theorem}

We prove this theorem by identifying $\dcover(K)$ for each such  knot $K$ as surgery on another knot in $S^3$, and using our   results to bound the dimension of the framed instanton homology (over $\C$) of these surgeries.

\subsection{Comparison with Heegaard Floer homology} Let $\F:=\Z/2\Z$ from this point on. We have  noted that $\chominvt = \tau$ for homogeneous and quasipositive knots. In fact,  we will  prove:

\begin{proposition}
\label{prop:tau-tau} If the equality \[\dim_\C I^\#(Y;\C) = \dim_\F \hfhat(Y;\F)\] holds for all $Y$ obtained via integer surgery on knots in $S^3$ then $\chominvt = \tau$ for all knots in $S^3$.
\end{proposition}

Note that this result provides a potential strategy for \emph{disproving} Kronheimer and Mrowka's conjectured isomorphism \eqref{eqn:conj-iso}: one need only find a knot $K$ for which $\tau^\#(K)\neq \tau(K)$. One source of possible such $K$ are twisted Whitehead doubles of $(2,2n+1)$ torus knots, as these are knots for which the $\tau$ invariant differs from Rasmussen's $s$ invariant, by work of Hedden and Ording \cite{hedden-ording}. In particular, a negative answer to the following  would disprove \eqref{eqn:conj-iso}:

\begin{question}
Is $\tau^\#(K) =0$ for $K$ the $2$-twisted positive Whitehead double of $T_{2,3}$?
\end{question}

Many of our computations of $\cinvt(K)$ and $r_0(K)$ use basic topology together with somewhat formal properties of framed instanton homology that are common to other Floer homology theories including Heegaard Floer homology. Accordingly, we are able to prove the following, providing some evidence for the conjectured isomorphism \eqref{eqn:conj-iso}:

\begin{theorem}
\label{thm:main-comparison}
Let $K$ be any of the  knots in Table~\ref{table:main}, any of the twist or pretzel knots in Theorems \ref{thm:main-twist}, \ref{thm:main-pretzels}, or \ref{thm:main-pretzels-2}, or any knot which is both an instanton and Heegaard Floer L-space knot. Then
\[\dim_\C I^\#(S^3_{p/q}(K);\C) = \dim_\F \hfhat(S^3_{p/q}(K);\F)\] for all nonzero rational numbers $p/q$. Let $Y$ be any of the  3-manifolds in Table~\ref{table:hw-census}, except possibly for census manifold 7. Then \[\dim_\C I^\#(Y;\C) = \dim_\F \hfhat(Y;\F).\]
\end{theorem}

\subsection{Mutation and concordance invariants} We conclude with some new observations regarding the behaviors of the concordance invariants $\chominvt$ and $\tau$ under  mutation. 

Given a knot $K\subset S^3$ and a  genus-2 surface $\Sigma$ in the complement of $K$, \emph{genus-2 mutation} along $\Sigma$ is process of  cutting $S^3$ open along $\Sigma$, and then regluing the pieces according to the hyperelliptic involution of $\Sigma$. The resulting manifold is still $S^3$ but the knot $K$ is taken to a potentially different knot, its \emph{genus-2 mutant}. 
Readers might be more familiar with \emph{Conway mutation}, which is an operation on knots defined diagrammatically. It turns out that any Conway mutation can be achieved by a sequence of genus-2 mutations \cite{ruberman,tillmann}.

There has been a lot of  work devoted to understanding whether various knot invariants are preserved by mutation. Very recently, for instance, Kotelskiy--Watson--Zibrowius \cite{KWZ} proved that Rasmussen's $s$ invariant (over a field) is preserved by Conway mutation. This is interesting because Conway mutation need not preserve  slice genus. Another result along these lines follows from Zibrowius's  work \cite{zibrowius} showing that $\delta$-graded knot Floer homology is preserved by Conway mutation. This implies that if $\hfkhat(K)$ is \emph{thin}, as  is the case for quasi-alternating knots, then $\tau(K) = \tau(K')$ for any Conway mutant $K'$ of $K$.
It remains open whether $\tau$ is preserved by Conway mutation in general, and is even less  clear  \emph{a priori} whether $\tau$ should  be preserved by arbitrary genus-2 mutation (which need not   preserve $\delta$-graded knot Floer homology). The same questions apply to $\chominvt$. We conjecture the following:

\begin{conjecture}
\label{conj:tau-mutation}
 The values of $\chominvt$ and $\tau$ are preserved by genus-2 mutation.
\end{conjecture}

One can define genus-2 mutation more generally. Given a separating genus-2 surface $\Sigma$ in a  $3$-manifold $Y$,  the corresponding \emph{genus-2 mutant} of $Y$ is the $3$-manifold $Y'$ obtained by cutting $Y$ open along $\Sigma$ and then regluing according to the hyperelliptic involution. Genus-2 mutation preserves classical homology, \[H_1(Y;\Z)\cong H_1(Y';\Z).\] Clarkson proved in \cite{clarkson} that it does not preserve  $\spc$-graded Heegaard Floer homology, but there is some evidence suggesting that \[\rank\hfhat(Y) = \rank\hfhat(Y'),\] for any closed $3$-manifold $Y$ and any genus-2 mutant $Y'$.  We conjecture that this holds for both Heegaard Floer homology and framed instanton homology:

\begin{conjecture}
\label{conj:rank-mutation}
 The ranks of  $I^\#$ and $\hfhat$ are preserved by genus-2 mutation.
\end{conjecture}

\begin{remark}
The analogue of Conjecture \ref{conj:rank-mutation}  was proven by Ruberman \cite{ruberman-mutation-floer} for the instanton Floer homology of homology $3$-spheres over $\F$.\footnote{The original proof claimed mutation invariance over $\Z$, but it contains an error.}
\end{remark}

The new observation we wish to highlight is:

\begin{proposition}
\label{prop:mutation}
Conjecture \ref{conj:rank-mutation} implies  Conjecture \ref{conj:tau-mutation}.
\end{proposition}

\begin{proof}
Let us suppose  the rank of $I^\#$ is preserved by genus-2 mutation and show that $\chominvt$ is as well.  Let $K$ be a knot in $S^3$ and $K'$ a genus-2 mutant of $K$.
Then $\#^nK$ and $\#^nK'$ are   related by a sequence of genus-2 mutations for every positive integer $n$. The $3$-manifolds \[S^3_{\pm N}(\#^nK) \quad \textrm{and}\quad S^3_{\pm N}(\#^nK')\] are also  then related by a sequence of genus-2 mutations for every positive integer $N$ (the fact that the mutation preserves the surgery framing $N$ follows from the fact that it preserves the order of first homology). Therefore, \[\dim_\C I^\#(S^3_{\pm N}(\#^nK);\C) = \dim_\C I^\#(S^3_{\pm N}(\#^nK');\C)\] for all positive integers $n$ and $N$. By Remark \ref{rmk:nufromsurgery}, this implies that \[\cinvt(\#^nK) = \cinvt(\#^nK')\] for all positive integers $n$, and therefore that $\chominvt(K) = \chominvt(K')$ by definition. 

The same  argument shows that the mutation invariance of $\tau$ follows from that of $\rank\hfhat$, since $\tau$ can also be expressed as the homogenization of a concordance invariant $\hat\nu$ which is determined by the ranks of $\hfhat$ of  surgeries, as shown in \S\ref{sec:comparison}. \end{proof}

\subsection{Organization}

We review framed instanton homology in \S\ref{sec:background}.  In \S\ref{sec:nu-sharp} we define $\cinvt$ and prove that it is a concordance invariant. In \S\ref{sec:rational-surgeries} we  prove Theorems~\ref{thm:main-surgery},   \ref{thm:nu-l-space}, and \ref{thm:l-space-cable}. In \S\ref{sec:quasi-morphism} and \S\ref{sec:tools} we establish some properties of $\cinvt$ and $\chominvt$, including tools for computing them, and prove Theorem \ref{thm:slice-torus}.  In \S\ref{sec:computations-dim} we compute  $\cinvt(K)$ and $\chominvt(K)$ for various $K$ and use these values to determine $r_0(K)$ in many cases, proving Theorems~\ref{thm:8-crossings}, \ref{thm:main-twist}, \ref{thm:main-pretzels}, and \ref{thm:main-pretzels-2}.  We prove Theorem~\ref{thm:main-ss} in \S\ref{sec:ss}, and compute the framed instanton homology of several small volume closed hyperbolic manifolds in \S\ref{sec:small-hyperbolic}, proving Theorem \ref{thm:hyperbolic}. We prove Theorem \ref{thm:main-comparison} in \S\ref{sec:comparison}.

\subsection{Conventions}
\label{ssec:conventions}

In the later sections of this paper we use data from various knot tables, including KnotInfo \cite{knotinfo}, the Rolfsen table \cite{rolfsen}, and the Knot Atlas \cite{knotatlas}.  KnotInfo does not always agree with the others about the chiralities of various knots, so when distinguishing between $K$ and $\mirror{K}$, we always use \cite{rolfsen,knotatlas} to decide which is which.  For example, $3_1$ is the left-handed trefoil. We  adopt the convention that the signature of the right-handed trefoil is $-2$. Beginning in \S\ref{sec:nu-sharp}, we will use $\C$ coefficients for framed instanton homology unless  stated otherwise. 
\subsection{Acknowledgments}

We are indebted to the creators of SnapPy \cite{snappy} and Regina \cite{regina}, which were essential tools in discovering every result in this paper about Kirby calculus or branched double covers.  We thank Ken Baker for providing the proof of Proposition~\ref{prop:pretzel-homeo}.  We also thank Jen Hom for explaining how our $\cinvt$ relates to the Heegaard Floer $\nu$ invariant, as discussed in \S\ref{sec:comparison}, and the anonymous referee for helpful feedback.

\section{Background on framed instanton homology} \label{sec:background}
In this section, we provide    background on framed instanton  homology and establish some notational conventions. All manifolds in this paper are smooth, oriented, and compact, and all submanifolds are smoothly and properly embedded. 

\subsection{Framed instanton homology}
 
Let $Y$ be a closed 3-manifold, and  let $\lambda$ be  a  (possibly empty) multicurve in $Y$.  To define the framed instanton homology of $(Y,\lambda)$,   introduced in \cite{km-yaft} and   developed further in \cite{scaduto}, we  first equip the connected sum $Y\#T^3$ with a Hermitian line bundle $w$ such that $c_1(w)$ is Poincar\'e dual to \[[\lambda \cup S^1] \in H_1(Y\#T^3;\Z),\] where the $S^1 $ above is a factor \[S^1\times \{\pt\}\textrm{ of }T^3 = S^1\times T^2.\] We then fix a $U(2)$-bundle $E\to Y\#T^3$ and an isomorphism $\det(E)\xrightarrow{\sim} w$. From this data one defines the instanton Floer homology group
\[ I_*(Y\# T^3)_w, \]
via a Chern--Simons functional on the space of projectively flat connections on $E$, modulo determinant-1 gauge transformations; see \cite{donaldson-book}. 
This is a relatively $\Z/8\Z$-graded, absolutely  $\Z/2\Z$-graded abelian group which admits a degree-4 involution $\frac{1}{2}\mu(\pt)$, and we define the \emph{framed instanton homology}
\[ I^\#(Y,\lambda) \]
to be the fixed set of this involution.  It is relatively $\Z/4\Z$-graded, and retains the absolute $\Z/2\Z$-grading; its Euler characteristic with respect to the latter is
\begin{equation} \label{eq:framed-chi}
\chi(I^\#(Y,\lambda)) :=\rank I_{\textrm{even}}^\#(Y,\lambda)-\rank I_{\textrm{odd}}^\#(Y,\lambda)= \begin{cases} |H_1(Y;\Z)|, & b_1(Y)=0 \\ 0, & b_1(Y)>0 \end{cases}
\end{equation}
as mentioned in Remark \ref{rmk:euler} \cite[Corollary~1.4]{scaduto}.  

\begin{remark}
The group $I^\#(Y,\lambda)$  depends implicitly on the basepoint $y\in Y$ at which we take the connected sum with $T^3$, but we omit this from the notation.
\end{remark}

\begin{remark}
\label{rmk:isomorphism-homology}
Up to isomorphism, $I^\#(Y,\lambda)$  depends on $\lambda$ only through its mod 2 homology class \[[\lambda]\in H_1(Y;\Z/2\Z).\]  We will frequently write $I^\#(Y)$ for $I^\#(Y,\lambda)$ when this class is trivial, and will often conflate $\lambda$ with its mod 2 homology class.
\end{remark}

A smooth cobordism $(X,\nu): (Y_0,\lambda_0) \to (Y_1,\lambda_1)$ induces a  homomorphism
\[ I^\#(X,\nu): I^\#(Y_0,\lambda_0) \to I^\#(Y_1,\lambda_1), \]
as  in \cite[\S7.2]{scaduto}.  This map depends implicitly on the basepoints $y_0 \in Y_0$ and $y_1 \in Y_1$ at which we take the connected sums with $T^3$, an  arc $\gamma \subset X$ from $y_0$ to $y_1$,  as well as  framings of these basepoints and a compatible framing of $\gamma$. Given these choices, we construct a new cobordism \begin{equation}\label{eq:Xsharp}X^\# := X \bowtie (T^3 \times [0,1]): Y_0\#T^3 \to Y_1\#T^3\end{equation} by removing tubular neighborhoods of the arcs\[\gamma \subset X \quad \textrm{and}\quad \{\pt\}\times [0,1]\subset T^3\times[0,1],\] and gluing what remains along the resulting copies of $S^2\times[0,1]$ according to the framings (using the product framing on the latter arc). The map $I^\#(X,\nu)$ is then defined to be the cobordism map on instanton Floer homology associated to the  line bundle over $X^\#$ whose first Chern class is Poincar{\'e} dual to  \[[\nu \cup (S^1\times [0,1])]\in H_2(X^\#,\partial X^\#;\Z),\] restricted to the fixed point set of the involution $\frac{1}{2}\mu(\pt).$ 
This map is well-defined up to sign, and  depends on $\nu$  only through the  mod 2 homology class \[[\nu]\in H_2(X,\partial X;\Z/2\Z).\] (The  sign can be pinned down by choosing a homology orientation on $X$, and the resulting map depends on $\nu$ only through its  \emph{integral} homology class.)

\begin{remark}In this paper we  only consider cobordisms built by attaching  $2$-handles to a product $Y\times [0,1]$. We will assume  the basepoints $y_0$ and $y_1$ are disjoint from the attaching regions of these 2-handles and  that $\gamma$ is a product arc with product framing, and  therefore omit  this extra data from the notation.
\end{remark}

\subsection{The surgery exact triangle}

Framed instanton homology comes equipped with a surgery exact triangle \cite{floer-surgery,braam-donaldson}.  We describe it here as presented by Scaduto in \cite{scaduto}.

\begin{theorem}[{\cite[\S7.5]{scaduto}}] \label{thm:exact-triangle}
Let $K$ be a framed knot in a closed $3$-manifold $Y$, and let $\mu \subset Y\setminus N(K)$ be a meridian of $K$. There is an exact triangle
\[ \dots \to I^\#(Y,\lambda) \to I^\#(Y_0(K),\lambda \cup \mu) \to I^\#(Y_1(K),\lambda) \to I^\#(Y,\lambda) \to \dots \]
for any multicurve $\lambda \subset Y \setminus N(K)$.  Moreover, each map in this triangle is a cobordism map induced by the corresponding 2-handle cobordism.
\end{theorem}

Given a knot $K \subset S^3$, with framing $n$ relative to the Seifert framing, we can choose
\[ \lambda = \begin{cases} 0, & n\mathrm{\ odd} \\ \mu, & n\mathrm{\ even} \end{cases} \]
to make the various $\lambda$ and $\lambda \cup \mu$ appearing in Theorem~\ref{thm:exact-triangle} nullhomologous mod $2$.  Thus, we have an exact triangle
\begin{equation} \label{eq:triangle-untwisted}
\dots \to I^\#(S^3) \xrightarrow{I^\#(X_n,\nu_n)} I^\#(S^3_n(K)) \to I^\#(S^3_{n+1}(K)) \to \dots,
\end{equation}
in which \[X_n=X_n(K):S^3\to S^3_n(K)\] is  the trace of the $n$-surgery on $K$ and \[\nu_n=\nu_{K,n}\] is some properly embedded surface in $X_n$.  For all $n \geq 0$, the map \[F_n=F_{K,n}=I^\#(X_n,\nu_n)\] has odd degree and the other two maps in the triangle have even degree \cite[\S7.3]{scaduto}.  One can also verify from \cite[\S3.3]{scaduto} that the  surface $\nu_n\subset X_n$ satisfies
\begin{equation} \label{eq:nu-parity}
[\nu_n] \cdot [\Sigma_n] \equiv n \pmod{2},
\end{equation}
where $\Sigma_n$ is the union of a Seifert surface for $K$ and a core of the $n$-framed $2$-handle.  (We only care about the intersection number mod 2 because $I^\#(X_n,\nu_n)$ is determined up to sign by the class $[\nu_n] \in H_2(X_n,\partial X_n;\Z/2\Z)$, as discussed above.)

More generally, for any knot $K\subset Y$, we can \emph{always} choose $\lambda$ so that the various $\lambda$ and $\lambda \cup \mu$  in Theorem~\ref{thm:exact-triangle} are nullhomologous mod $2$. Indeed, if $K$ is nullhomologous mod 2 in $Y$ then $\mu$ is nullhomologous mod 2 in exactly one of $Y_0(K)$ or $Y_1(K)$, and we set $\lambda = 0$ or $\mu$, respectively. If $K$ is not nullhomologous mod 2 in $Y$ then there is a closed surface $\Sigma\subset Y$ which intersects $K$ in $n$ points for  some odd integer $n$.  Removing small neighborhoods of these intersection points from $\Sigma$, we obtain a surface  bounded mod 2 by $n$ copies of $\mu$. This shows (since $n$ is odd) that \[[\mu] = 0\in H_1(Y\setminus N(K);\Z/2\Z),\] which  implies  that $[\mu]=0\in H_1(Y_0(K);\Z/2\Z)$ as well, so we can just take $\lambda = 0$. Thus, we can always arrange that the exact triangle in Theorem~\ref{thm:exact-triangle} takes the form \[\dots \to I^\#(Y) \to I^\#(Y_0(K)) \to I^\#(Y_1(K)) \to\dots\] We use this implicitly in \S\ref{sec:rational-surgeries}, where the surgery exact triangles appear without any multicurves.

In \S\ref{sec:nu-sharp} and elsewhere, we will need to show that certain cobordism maps  in these exact triangles vanish when working with $\C$ coefficients. We will do so using the following proposition (cf.\ \cite{km-embedded2}):

\begin{proposition}[{\cite[Proposition~6.7]{bs-lspace}}] \label{prop:adjunction-inequality}
Let $(X,\nu): (Y_0,\lambda_0) \to (Y_1,\lambda_1)$ be a smooth cobordism with $b_1(X)=0$.  Suppose that $S \subset X$ is a closed, smoothly embedded surface such that either:
\begin{itemize}
\item $[S] \cdot [S] \geq \max(2g(S)-1,1)$; or
\item $S$ is a sphere of self-intersection zero, and $[S]\cdot[F] \neq 0$ for some other closed surface $F \subset X$.
\end{itemize}
Then the induced map \[I^\#(X,\nu): I^\#(Y_0,\lambda_0;\C) \to I^\#(Y_1,\lambda_1;\C)\] on framed instanton homology with $\C$ coefficients is zero.
\end{proposition}

\subsection{The spectral sequence from odd Khovanov homology}

Odd Khovanov homology \cite{orsz} associates to any link $L \subset S^3$ a bigraded abelian group \[\Khodd(L)= \bigoplus_{i,j}\Khodd_{i,j}(L).\]  There is a smaller invariant, the \emph{reduced} odd Khovanov homology \[\Khoddr(L)= \bigoplus_{i,j}\Khoddr_{i,j}(L),\] which is related to the unreduced version by 
\[ \Khodd_{i,j}(L) \cong \Khoddr_{\!i,j-1}(L) \oplus \Khoddr_{\!i,j+1}(L). \]
Its graded Euler characteristic is the Jones polynomial of $L$, meaning that
\[ V_L(q^2) = \sum_{i,j\in\Z} (-1)^i \rank(\Khoddr_{\!i,j}(L)) \cdot q^j. \]

\begin{theorem}[{\cite[Theorem~1.1]{scaduto}}] \label{thm:kh-to-i-ss}
For any link $L\subset S^3$, there is a spectral sequence
\[ \Khoddr(L) \Rightarrow I^\#(-\dcover(L)), \]
whose $E^2$ page is the reduced odd Khovanov homology of $L$.
\end{theorem}

We say that $\Khoddr(L)$ is \emph{thin} if there is some integer $\delta\in\Z$ such that
\[ \Khoddr_{\!i,j}(L) = 0 \textrm{ unless } j-2i = \delta. \]
If $\Khoddr(L)$ is thin, then the fact that $V_L(-1) = \pm \det(L)$ implies that
\[ \rank \Khoddr(L) = \det(L). \]
In particular, quasi-alternating knots have thin odd Khovanov homology \cite{mo-qa}.
 
\begin{proposition} \label{prop:ss-thin}
Let $K \subset S^3$ be a knot satisfying \[\dim_\C \Khoddr(K;\C) = \det(K),\] such as any knot for which $\Khoddr(K)$ is thin.  Then the spectral sequence of Theorem~\ref{thm:kh-to-i-ss},  with coefficients in $\C$, collapses at the $E^2$ page, which then implies that \[\dim_\C I^\#(-\dcover(L);\C)=\det(K).\]
\end{proposition}

\begin{proof}
Using \eqref{eq:framed-chi}, but with coefficients in $\C$, it is always true that
\[ \dim I^\#(-\dcover(K);\C) \geq |\chi(I^\#(-\dcover(K);\C))| = |H_1(-\dcover(K);\Z)| = \det(K). \]
If we know that $\dim_\C \Khoddr(K;\C) = \det(K)$, then the spectral sequence gives us an inequality
\[ \det(K) = \dim \Khoddr(K;\C) \geq \dim I^\#(-\dcover(K);\C), \]
so all of these terms must be equal.
\end{proof}

\begin{remark}
As noted in \S\ref{ssec:conventions}, we will henceforth work with coefficients in $\C$, and use $I^\#(Y,\lambda)$ to mean $I^\#(Y,\lambda;\C)$, unless stated otherwise.
\end{remark}

\section{A concordance invariant from integer surgeries} \label{sec:nu-sharp}

In this section, we  study the cobordism maps arising in the exact triangles
\begin{equation} \label{eq:triangles}
\dots \to I^\#(S^3) \xrightarrow{F_n} I^\#(S^3_n(K)) \to I^\#(S^3_{n+1}(K)) \xrightarrow{G_{n+1}} \dots
\end{equation}
of \eqref{eq:triangle-untwisted}, with \[F_n = I^\#(X_n,\nu_n),\] where $X_n=X_n(K)$ is the trace of $n$-surgery on $K$. In particular, we  make   use of Proposition~\ref{prop:adjunction-inequality} to determine when the  maps $F_n$ and $G_n$ vanish (recall that we are henceforth working with coefficients in $\C$), and  use that to understand how the dimension of $I^\#(S^3_n(K))$ varies with $n$. This will lead to the definition of our first concordance invariant $\cinvt$ and set the stage for the proof of Theorem \ref{thm:main-surgery}.

We begin with the following, which is a direct analogue of \cite[Proposition~7.2]{kmos}.

\begin{lemma} \label{lem:cancel-2-handles}
Let $K \subset S^3$ be a knot, and let $F_n$ and $G_n$ be the cobordism maps  in \eqref{eq:triangles}.  If $n \neq 0$ then $F_n \circ G_n = 0$ as a map $I^\#(S^3_n(K)) \to I^\#(S^3_n(K))$.
\end{lemma}

\begin{proof}
This is essentially Lemma~4.13 and Remark~4.14 of \cite{bs-stein}, which in turn follows the proof of \cite[Proposition~6.5]{km-embedded2}.   The composition is induced by a cobordism
\[ W: S^3_n(K) \to S^3 \to S^3_n(K)\, \]
in which we attach a $0$-framed $2$-handle $H_\mu$ to $S^3_n(K) \times [0,1]$ along a meridian $\mu$ of $K\times\{1\}$ and then attach an $n$-framed 2-handle $H_K$ to $K$ in the resulting $S^3$.  The cobordism $W$ contains a smoothly embedded $2$-sphere $S$ of self-intersection zero, given by the union of a cocore of $H_\mu$ and a core of $H_K$.

We wish to apply Proposition~\ref{prop:adjunction-inequality} to $S$, so we must construct a surface $F$ with $[S]\cdot[F] \neq 0$.  Note that the disjoint union of $|n|$ parallel cores of $H_\mu$ is bounded by a nullhomologous link in $S^3_n(K) \times \{1\}$. Letting $F$ be the union of the $|n|$ parallel cores of $H_\mu$ with a Seifert surface for this link,  we have that $[S] \cdot [F] = |n| \neq 0$ as desired.
\end{proof}

We now use Lemma~\ref{lem:cancel-2-handles} to determine how the dimension of $I^\#(S^3_n(K))$ varies with $n$.

\begin{proposition} \label{prop:unimodal-rank-part-1}
Let $K \subset S^3$ be a knot.  There is an integer $N = N(K) \geq 0$ depending on $K$ such that $N \leq \max(2g_s(K)-1,1)$ and
\begin{equation} \label{eq:rank-n-geq-0}
\dim I^\#(S^3_n(K)) = \dim I^\#(S^3_N(K)) + |N-n|
\end{equation}
for all $n \geq 0$.
\end{proposition}

\begin{proof}
Consider the surgery exact triangles \eqref{eq:triangles}.    Since $I^\#(S^3) \cong \C$, the map $F_n$ is either zero or injective.  Note that $F_n = I^\#(X_n,\nu_n)$ satisfies
\[ F_n = 0 \quad\mathrm{for\ all}\quad n \geq \max(2g_s(K)-1, 1) \]
by Proposition~\ref{prop:adjunction-inequality}, because $X_n$ contains a closed surface of genus $g_s(K)$ and self-intersection $n$, built by taking a smooth genus-$g_s(K)$ surface in $S^3\times [0,1]$ with boundary $K\times\{1\}$ and attaching a core of the 2-handle.  Moreover, if $F_n=0$ then
\[ \dim I^\#(S^3_{n+1}(K)) = \dim I^\#(S^3_n(K)) + 1, \]
by exactness.  We thus let $N \leq \max(2g_s(K)-1,1)$ be the smallest nonnegative integer such that $F_n = 0$ for all $n \geq N$, and then the proposition follows for all $n \geq N$ by induction.

If $N = 0$ then we are done, so we suppose now that $N \geq 1$; then $F_{N-1}$ is injective.  In general, if $F_n$ is injective for some $n$ then $G_{n+1}=0$ by the exactness of \eqref{eq:triangles}, hence
\[ \dim I^\#(S^3_n(K)) = \dim I^\#(S^3_{n+1}(K)) + 1. \]
If in addition $n > 0$, then $F_n \circ G_n = 0$ by Lemma \ref{lem:cancel-2-handles}, which implies that $G_n = 0$, and so again by exactness we conclude that $F_{n-1}$ is also injective.  Thus $F_{N-1}, F_{N-2}, \dots, F_0$ are injective and the cases $N-1 \geq n \geq 0$ of the proposition follow by induction as well.
\end{proof}

Proposition~\ref{prop:unimodal-rank-part-1} cannot always be extended to all $n \in \Z$, because Lemma~\ref{lem:cancel-2-handles} does not tell us that $F_0 \circ G_0 = 0$.  However, in most cases we can strengthen it substantially.

\begin{proposition} \label{prop:unimodal-rank}
Suppose that $N(K)$ and $N(\mirror{K})$ are not both 1.  Then at least one of $N(K)$ and $N(\mirror{K})$ is zero, and if we let
\[ N = \begin{cases} N(K), & N(K) > 0 \\ -N(\mirror{K}), & N(K)=0 \end{cases} \]
then
\begin{equation} \label{eq:unimodal-for-all-n}
\dim I^\#(S^3_n(K)) = \dim I^\#(S^3_N(K)) + |N-n| \quad\mathrm{for\ all\ }n \in \Z.
\end{equation}
If $N(K)$ and $N(\mirror{K})$ are also not both zero, then \[\dim I^\#(S^3_0(K),\mu) = \dim I^\#(S^3_0(K),0),\] where $\mu\subset S^3\setminus N(K)$ is a meridian of $K$.
\end{proposition}

\begin{proof}
We first observe that in addition to the exact triangles \eqref{eq:triangles} involving $F_n$ and $G_n$, we have a pair of exact triangles
\[ \xymatrix@R1ex{
\dots \ar[r] & I^\#(S^3) \ar[r]^-{\tilde{F}_{-1}} & I^\#(S^3_{-1}(K)) \ar[r] & I^\#(S^3_0(K),\mu) \ar[r]^-{\tilde{G}_0} & \dots \\
\dots \ar[r] & I^\#(S^3) \ar[r]^-{\tilde{F}_0} & I^\#(S^3_0(K),\mu) \ar[r] & I^\#(S^3_1(K)) \ar[r]^-{\tilde{G}_1} & \dots \\
} \]
as in Theorem~\ref{thm:exact-triangle}.  Lemma~\ref{lem:cancel-2-handles} does not tell us whether $F_0\circ G_0$ or $\tilde{F}_0 \circ \tilde{G}_0$ are zero, but we claim that
\[ F_0 \circ \tilde{G}_0 = \tilde{F}_0 \circ G_0 = 0. \]
To prove that the map
\[ \tilde{F}_0\circ G_0: I^\#(S^3_0(K),0) \to I^\#(S^3) \to I^\#(S^3_0(K),\mu) \]
vanishes (the other case is identical), we observe that the bundles on either copy of $S^3_0(K)$ uniquely determine the $SO(3)$ bundles over the cobordisms defining $G_0$ and $\tilde{F}_0$, since in either case the restriction map from $H^2$ of the cobordism to that of $S^3_0(K)$ is an isomorphism.  There is now a 2-sphere of self-intersection 0 in the composite cobordism, formed as the union of a cocore of the $G_0$-handle and a core of the $\tilde{F}_0$-handle, on which the restriction of the $SO(3)$-bundle is nontrivial, and this forces the composite map to be zero as claimed (by a neck-stretching argument, a nonzero map leads to flat connections over the boundary $S^1\times S^2$ of a neighborhood of the 2-sphere, but these cannot exist over any $\{\pt\} \times S^2$ since the bundle is nontrivial there).

We can apply the above fruitfully whenever $N(K) \geq 1$.  In this case, we saw in the proof of Proposition~\ref{prop:unimodal-rank-part-1} that $F_0$ is injective, and since $F_0 \circ \tilde{G}_0 = 0$ we also have $\tilde{G}_0 = 0$, hence
\begin{equation} \label{eq:-1-vs-0mu}
\dim I^\#(S^3_{-1}(K)) = \dim I^\#(S^3_0(K),\mu) + 1.
\end{equation}
If in fact $N(K) \geq 2$ then we similarly know that $F_1$ is injective, so Lemma~\ref{lem:cancel-2-handles} says that $G_1$ and $\tilde{G}_1$ are both zero (the proof did not depend on the bundles involved).  In this case we have 
\begin{equation} \label{eq:zero-surgery-0-mu}
\dim I^\#(S^3_0(K),\mu) = \dim I^\#(S^3_0(K),0) = \dim I^\#(S^3_1(K)) + 1,
\end{equation}
and combining this with \eqref{eq:-1-vs-0mu} gives
\begin{align*}
\dim I^\#(S^3_1(\mirror{K})) &= \dim I^\#(S^3_{-1}(K)) \\
&= \dim I^\#(S^3_0(K),\mu)+1 \\
&= \dim I^\#(S^3_0(K),0)+1 \\
&= \dim I^\#(S^3_0(\mirror{K}),0) + 1.
\end{align*}
By Proposition~\ref{prop:unimodal-rank-part-1}, this is only possible if $N(\mirror{K}) = 0$, in which case
\begin{equation}\label{eq:mirror-n} \dim I^\#(S^3_n(\mirror{K})) = \dim I^\#(S^3_0(\mirror{K})) + n \end{equation}
for all $n \geq 0$.  Reversing orientation again, this extends \eqref{eq:rank-n-geq-0} from $n\geq 0$ to all integers $n$, completing the proof of the proposition in the case where $N(K) \geq 2$.

The case where $N(\mirror{K}) \geq 2$ is identical, so we may now assume that $N(K)$ and $N(\mirror{K})$ are both $0$ or $1$; by assumption at least one of these is zero, so \eqref{eq:unimodal-for-all-n} is automatically satisfied by the reasoning following \eqref{eq:mirror-n}. It remains to be shown that
\[ \dim I^\#(S^3_0(K),0) = \dim I^\#(S^3_0(K),\mu) \]
as long as $N(K)$ and $N(\mirror{K})$ are not both zero.  Supposing that $N(K) = 1$ and $N(\mirror{K}) = 0$, we have
\begin{align*}
\dim I^\#(S^3_0(K),0) &=\dim I^\#(S^3_0(\mirror{K}),0) \\
&= \dim I^\#(S^3_1(\mirror{K})) - 1 \\
&= \dim I^\#(S^3_{-1}(K)) - 1 \\
&= \dim I^\#(S^3_0(K),\mu),
\end{align*}
where the last equality is \eqref{eq:-1-vs-0mu}.
The case $(N(K),N(\mirror{K})) = (0,1)$ follows by exchanging the roles of $K$ and $\mirror{K}$.
\end{proof}

\begin{remark} \label{rem:N-equals-1-1}
In the case $N(K)=N(\mirror{K})=1$, the equality \[\dim I^\#(S^3_{-1}(K)) = \dim I^\#(S^3_0(K),\mu) + 1\] in \eqref{eq:-1-vs-0mu} holds since $N(K)=1$, while $N(\mirror{K})=1$ implies that \[\dim I^\#(S^3_1(\mirror{K})) = \dim I^\#(S^3_0(\mirror{K}),0)-1,\] by Proposition \ref{prop:unimodal-rank-part-1}.  Together these give
\[ \dim I^\#(S^3_0(K),\mu) = \dim I^\#(S^3_0(K),0)-2, \]
and then it follows easily that
\[ \dim I^\#(S^3_n(K)) = \dim I^\#(S^3_0(K),\mu) + |n| \]
for all integers $n \neq 0$.
\end{remark}

Proposition~\ref{prop:unimodal-rank} nearly says that the sequence of integers
\[ \dim I^\#(S^3_n(K)) \quad (n \in \Z) \]
is unimodal.  Indeed, in every case except  when $N(K) = N(\mirror{K}) = 1$, it achieves a unique minimum at $n = N(K) - N(\mirror{K})$.  In the remaining case, where $N(K)=N(\mirror{K})=1$, there are two minima, at $n=\pm1$, by the discussion in Remark \ref{rem:N-equals-1-1}; this happens, for instance, when $K$ is the unknot.  See Figure~\ref{fig:dim-plots} for some examples.
\begin{figure}
\begin{tikzpicture}
\begin{axis}[height=5cm, axis equal image,
    title = {$\dim I^\#(S^3_n(U))$},
    axis y line=middle, axis x line=bottom, x axis line style={latex-latex},
    xmin=-5.5, xmax=5.5, xtick = {-4,-2,...,4}, minor xtick = {-5,-3,...,5},
    ymin=0, ymax=7.5, ytick = {2,4,6}, minor ytick = {1,3,5,7},
]
\addplot [mark=*, samples=11] {min(abs(x+1),abs(x-1))+1};
\end{axis}

\begin{axis}[xshift=5cm, height=5cm, axis equal image,
    title = {$\dim I^\#(S^3_n(T_{2,3}))$},
    axis y line=middle, axis x line=bottom, x axis line style={latex-latex},
    xmin=-5.5, xmax=5.5, xtick = {-4,-2,...,4}, minor xtick = {-5,-3,...,5},
    ymin=0, ymax=7.5, ytick = {2,4,6}, minor ytick = {1,3,5,7},
]
\addplot [mark=*, samples=11] {abs(x-1)+1};
\end{axis}

\begin{axis}[xshift=10cm, height=5cm, axis equal image,
    title = {$\dim I^\#(S^3_n(4_1))$},
    axis y line=middle, axis x line=bottom, x axis line style={latex-latex},
    xmin=-5.5, xmax=5.5, xtick = {-4,-2,...,4}, minor xtick = {-5,-3,...,5},
    ymin=0, ymax=7.5, ytick = {2,4,6}, minor ytick = {1,3,5,7},
]
\addplot [mark=*, samples=5, domain=-5:-1] {abs(x)+2};
\addplot [mark=*, samples=5, domain=1:5] {abs(x)+2};
\end{axis}
\end{tikzpicture}
\caption{Plots of $\dim I^\#(S^3_n(K))$ for $K$ the unknot (left), right-handed trefoil (middle), and figure eight (right, shown for $n\neq 0$) respectively.  The unknot and trefoil are W-shaped and V-shaped, respectively, and it is not clear which shape describes the figure eight, though it is true that $\dim I^\#(S^3_0(4_1),\mu) = 2$.} \label{fig:dim-plots}
\end{figure}
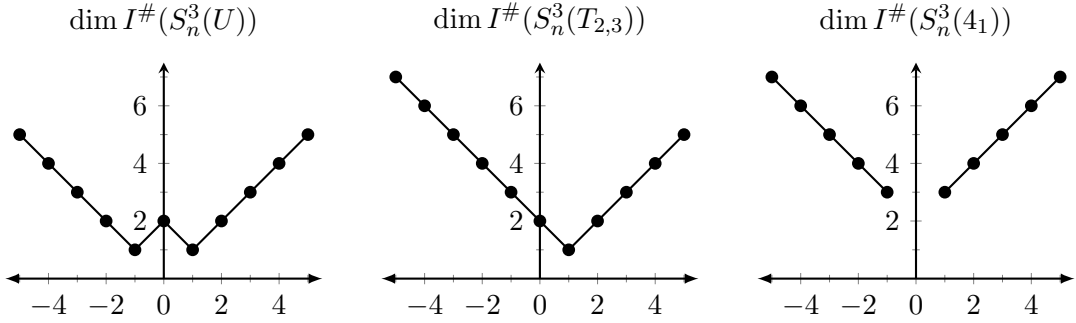
This motivates the following definitions.

\begin{definition} \label{def:conc-invt}
For any knot $K \subset S^3$, we define $\cinvt(K) \in \Z$ by the equation
\[ \cinvt(K) = N(K) - N(\mirror{K}). \]
\end{definition}

\begin{definition} \label{def:shape}
We say that $K\subset S^3$ is \emph{V-shaped} if either $\cinvt(K) \neq 0$ or $\cinvt(K) = N(K) = 0$.  Otherwise Proposition~\ref{prop:unimodal-rank} says that $\cinvt(K) = 0$ and $N(K) = 1$, and in this case we say that $K$ is \emph{W-shaped}.  (See Figure~\ref{fig:dim-plots}.)  We then define
\begin{equation} \label{eq:r_0}
r_0(K) = \begin{cases} \dim I^\#(S^3_{\cinvt(K)}(K)), & K \textrm{ is V-shaped} \\ \dim I^\#(S^3_0(K),\mu), & K \textrm{ is W-shaped}. \end{cases}
\end{equation}
We note that $r_0(K) = r_0(\mirror{K})$, since Definition~\ref{def:conc-invt} implies that $\cinvt(\mirror{K})=-\cinvt(K)$ and hence \[S^3_{\cinvt(K)}(K) \cong -S^3_{\cinvt(\mirror{K})}(\mirror{K}).\]  Moreover, it is clear from the definition and from \eqref{eq:framed-chi} that $r_0(K) \geq |\cinvt(K)|$ and that $r_0(K)$ has the same parity as $\cinvt(K)$.
\end{definition}

\begin{theorem} \label{thm:conc-invt}
The invariant $\cinvt(K)$ is a smooth concordance invariant.  It satisfies
\[ \dim I^\#(S^3_n(K)) = r_0(K) + |n - \cinvt(K)| \]
for all integers $n$, except when $K$ is W-shaped and $n=0$; in particular, all smoothly slice knots are W-shaped.  We also have $\cinvt(\mirror{K}) = -\cinvt(K)$ and $|\cinvt(K)| \leq \max(2g_s(K)-1,0)$.
\end{theorem}

\begin{proof}
The claims about the dimension of $I^\#(S^3_n(K))$ are immediate from Proposition~\ref{prop:unimodal-rank} and Remark~\ref{rem:N-equals-1-1}.  We will see below that $N(K)$ is itself a concordance invariant; thus if $K$ is slice then $N(K) = N(U) = 1$ and $N(\mirror{K}) = N(\mirror{U}) = 1$, and so $K$ is W-shaped.

The claim that $\cinvt(\mirror{K}) = -\cinvt(K)$ is immediate from the definition.  Since $N(\mirror{K}) \geq 0$, we have
\[ \cinvt(K) \leq N(K) \leq \max(2g_s(K)-1,1), \]
where the second inequality is by Proposition~\ref{prop:unimodal-rank-part-1}.  The same applies to $\cinvt(\mirror{K})$, and since $g_s(K) = g_s(\mirror{K})$, we then have
\[ |\cinvt(K)| = \max(\cinvt(K),\cinvt(\mirror{K})) \leq \max(2g_s(K)-1,1). \]
The right side is equal to $2g_s(K)-1$ in all cases except $g_s(K)=0$.  But we have $\cinvt(U) = N(U)-N(U) = 0$, so if $K$ is smoothly slice then the concordance invariance of $\cinvt$ will imply that $\cinvt(K)=0$, proving the desired inequality.

We will now show that $N(K)$ is a concordance invariant; an identical argument applies to $N(\mirror{K})$, and then this implies the same for their difference $\cinvt(K)$.  In the proof of Proposition~\ref{prop:unimodal-rank-part-1}, we defined $N(K)$ as the least nonnegative integer such that the map
\[ F_n: I^\#(S^3) \to I^\#(S^3_n(K)) \]
is zero for all $n \geq N$, so it suffices to prove that the rank of $F_n$ (which is either $0$ or $1$, since $I^\#(S^3) \cong \C$) is itself a smooth concordance invariant.

Suppose that $K_0$ is smoothly concordant to $K_1$, and let $C \subset S^3 \times [0,1]$ be a smoothly embedded cylinder with boundary $-K_0 \times \{0\} \sqcup K_1 \times \{1\}$; we can arrange for simplicity that it restricts to a product cobordism inside $S^3 \times ([0,\frac{1}{3}] \cup [\frac{2}{3},1])$.  We build the $n$-framed $2$-handle cobordism $X_n(K_1)$ by attaching an $n$-framed 2-handle to $K_1 \times \{1\} \subset S^3 \times [0,1]$; let $D$ denote the core of this handle, with boundary $K_1 \times \{1\}$.  Inside $X_n(K_1)$, the union of $S^3 \times [0,\frac{1}{3}]$ and a neighborhood of $C \cup D$ is then diffeomorphic to the 2-handle cobordism $X_n(K_0)$, so that we can write $X_n(K_1)$ as a composition
\[ S^3 \xrightarrow{X_n(K_0)} S^3_n(K_0) \xrightarrow{V} S^3_n(K_1) \]
for some smooth cobordism $V$; see Figure~\ref{fig:concordance-invariance}.
\begin{figure}
\begin{tikzpicture}
\draw (-3,-2) rectangle (0,2);
\node[above,inner sep=1pt] at (-1.5,-2) {\small $S^3\times[0,1]$};
\draw (0,-0.5) arc (-90:90:0.5);
\draw (0,-1.5) arc (-90:90:1.5);
\coordinate (K0d) at (-3,-0.5);
\coordinate (K0u) at (-3,0.5);
\coordinate (K1d) at (0,-1);
\coordinate (K1u) at (0,1);
\draw[fill=black] (K0u) circle (0.05) node[anchor=south west] {\small $K_0$};
\draw[fill=black] (K0d) circle (0.05) node[anchor=north west] {\small $K_0$};
\draw[fill=black] (K1u) circle (0.05) node[anchor=south east] {\small $K_1$};
\draw[fill=black] (K1d) circle (0.05) node[anchor=north east] {\small $K_1$};
\draw (K0u) to +(1,0) to[out=0,in=180] ($(K1u)+(-1,0)$) node[below] {\small $C$} to (K1u);
\draw (K0d) to +(1,0) to[out=0,in=180] ($(K1d)+(-1,0)$) to (K1d);
\draw (K1d) arc (-90:90:1) node[pos=0.75,right] {\small $D$};

\node at (2.25,0) {$\longleftrightarrow$};

\begin{scope}[xshift=6cm]
\draw[fill=black!10] (0,2) -- (-3,2) -- (-3,-2) -- (0,-2) -- (0,-1.5) arc(-90:90:1.5) -- cycle;
\coordinate (K0d) at (-3,-0.5);
\coordinate (K0u) at (-3,0.5);
\coordinate (K1d) at (0,-1);
\coordinate (K1u) at (0,1);
\draw [fill=white] (-3,2) -- (-2,2) -- (-2,0.75) to[out=0,in=180] ($(K1u)+(-1,0.25)$) -- +(1,0) arc(90:-90:1.25) -- +(-1,0) to[out=180,in=0] (-2,-0.75) -- (-2,-2) -- (-3,-2) -- cycle;
\draw[fill=black!10] ($(K0u)+(1,-0.25)$) to[out=0,in=180] ($(K1u)+(-1,-0.25)$) -- +(1,0) arc(90:-90:0.75) -- +(-1,0) to[out=180,in=0] ($(K0d)+(1,0.25)$) -- cycle;
\draw[fill=white] (0,0.5) -- (0,-0.5) arc(-90:90:0.5) -- cycle;
\draw[thin,dashed] (K0u) to +(1,0) to[out=0,in=180] ($(K1u)+(-1,0)$) to (K1u) arc (90:-90:1) -- +(-1,0) to[out=180,in=0] ($(K0d)+(1,0)$) -- (K0d);
\end{scope}
\end{tikzpicture}
\caption{Decomposing $X_n(K_1)$ as the union $X_n(K_0) \cup V$, where $V$ is the shaded region at right.}
\label{fig:concordance-invariance}
\end{figure}
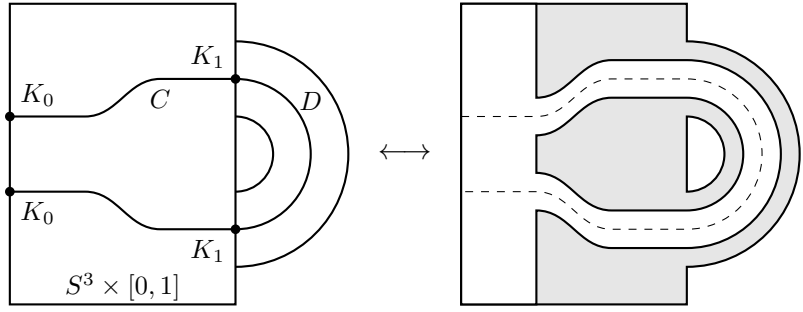

The restriction to $X_n(K_0)$ of the surface $\nu_{K_1,n} \subset X_n(K_1)$ defining the map \[F_{K_1,n} = I^\#(X_n(K_1),\nu_{K_1,n})\] still has pairing $n \pmod{2}$ with a capped-off Seifert surface  for $K_0$, just as it did for $K_1$, so it follows that $F_{K_1,n}$ factors through
\[ F_{K_0,n} = I^\#(X_n(K_0),\nu_{K_0,n}): I^\#(S^3) \to I^\#(S^3_n(K_0)). \]
This implies that $\rank F_{K_1,n} \leq \rank F_{K_0,n}$, and the opposite inequality holds as well since concordance is symmetric.  Thus, the rank of $F_{K,n}$ depends on $K$ only up to smooth concordance, as claimed.
\end{proof}

We can now justify the data plotted in Figure~\ref{fig:dim-plots}.  We note that according to Theorem~\ref{thm:conc-invt}, we have $\cinvt(K)=0$ if $K$ is amphichiral, such as the unknot or figure eight.

\begin{example} \label{ex:unknot-plot}
For the unknot we have $\cinvt(U) = 0$, and $\dim I^\#(S^3_{\pm1}(U)) = \dim I^\#(S^3) = 1$ implies that $r_0(U) = 0$ as well.  We also know that $\dim I^\#(S^3_0(K)) = \dim I^\#(S^1\times S^2) = 2$ \cite[\S7.6]{scaduto}, so that $N(U)=N(\mirror{U})=1$ and hence $U$ is W-shaped, with $\dim I^\#(S^3_n(U))$ having the form shown in Figure~\ref{fig:dim-plots}.
\end{example}

\begin{example} \label{ex:trefoil-plot}
For the right-handed trefoil we have
\begin{align*}
\dim I^\#(S^3_1(T_{2,3})) = \dim I^\#(-\Sigma(2,3,5)) &= 1, \\
\dim I^\#(S^3_{-1}(T_{2,3})) = \dim I^\#(\Sigma(2,3,7)) &= 3,
\end{align*}
as computed in \cite[Corollary~1.7]{scaduto} and \cite[Corollary~1.6]{scaduto}, respectively.  It follows from this and the fact that $\dim I^\#(S^3_n(T_{2,3})) \geq n$ for all $n \geq 1$ that $\cinvt(T_{2,3})=1$, and that $\dim I^\#(S^3_n(T_{2,3})) = |n-1|+1$ for all $n \in \Z$.
\end{example}

\begin{example} \label{ex:figure-eight-plot}
We have $\cinvt(4_1)=0$ since $4_1$ is amphichiral.  We know once again that
\[ \dim I^\#(S^3_{\pm1}(4_1)) = \dim I^\#(\pm \Sigma(2,3,7)) = 3, \]
from which everything follows except the case of $0$-surgery.  In fact, we can show that $\dim I^\#(S^3_0(4_1),\mu) = 2$ because this $0$-surgery is a torus bundle over the circle, but we do not know whether $\dim I^\#(S^3_0(4_1))$ is $2$ or $4$.  The answer depends on whether $4_1$ is V-shaped or W-shaped.
\end{example}

\begin{remark}
The computations of Examples~\ref{ex:trefoil-plot} and \ref{ex:figure-eight-plot} rely on knowing $I^\#$ for some Brieskorn spheres.  Later on, we will redo these computations (in Lemmas~\ref{lem:t2q-surgery} and \ref{lem:figure-eight}) in a way which only relies on the facts that $\cinvt(U)=r_0(U)=0$, which in Example~\ref{ex:unknot-plot} only required us to know that $I^\#(S^3) \cong \C$.
\end{remark}

Recall from the introduction that a nontrivial knot $K$ is an \emph{instanton L-space knot} if \[\dim_\C I^\#(S^3_{n}(K);\C) = n\] for some integer $n>0$. We implicitly computed $\cinvt$ for such knots in \cite{bs-lspace}:

\begin{proposition} \label{prop:nu-l-space}
If $K$ is an instanton L-space knot then $\cinvt(K)=r_0(K)=2g(K)-1$.
\end{proposition}

\begin{proof}
Suppose $K$ is an instanton L-space knot. In \cite{bs-lspace}, we proved that $\dim I^\#(S^3_n(K))=n$ if and only if $n \geq 2g(K)-1$.  Thus, if $g(K) \geq 2$ then we must have $N(K)=2g(K)-1 \geq 3$, which implies that $N(\mirror{K})=0$ by Proposition~\ref{prop:unimodal-rank}, so $\cinvt(K) = 2g(K)-1$ in this case. When $g(K)=1$, we proved that $K$ must be the right-handed trefoil, and then \[\cinvt(K) = 2g(K)-1 =1,\]  by Example~\ref{ex:trefoil-plot}. Thus, $r_0(K) = \dim I^\#(S^3_{2g(K)-1}(K))=2g(K)-1$ in both cases.
\end{proof}

\section{Rational surgeries} \label{sec:rational-surgeries}
We proved Theorem \ref{thm:main-surgery} for integer surgeries in the previous section, as part of Theorem \ref{thm:conc-invt}. The main goal of this section is to prove Theorem \ref{thm:main-surgery} for  rational surgeries, restated in a slightly stronger form as Theorem \ref{thm:rational-surgeries} below. We then use Theorem \ref{thm:rational-surgeries} to prove Theorem \ref{thm:nu-l-space}, regarding surgeries on instanton L-space knots, and Theorem \ref{thm:l-space-cable}, which determines when cables of a knot are instanton L-space knots.

For notational simplicity in what follows, let us introduce the  quantity $\delta(K)$ below, which is well-defined according to Theorem~\ref{thm:conc-invt}:

\begin{definition} \label{def:delta-k}
Let $K \subset S^3$ be a knot, and write $\nu = \cinvt(K)$.  We define
\[ \delta(K) \in 2\Z_{\geq 0} \]
to be the unique integer such that
\[ \dim I^\#(S^3_n(K)) = (\delta(K) + \cinvt(K)) + |n-\cinvt(K)| \]
for all nonzero integers $n$.
\end{definition}

Indeed, Theorem~\ref{thm:conc-invt} immediately implies that
\begin{equation} \label{eq:delta-expression}
\delta(K) = r_0(K) - \cinvt(K) := \dim I^\#(S^3_{\cinvt(K)}(K),\lambda)-\cinvt(K),\end{equation} where $\lambda$ is either $0$ or $\mu$, depending on whether $K$ is V-shaped or W-shaped.
Note that the difference on the right-hand side, and hence $\delta(K)$, is both nonnegative and even by \eqref{eq:framed-chi}, as claimed.  One should view $\delta(K)$ as measuring how far $K$ is from having positive integer L-space surgeries.

We will show below that the framed instanton homologies of nearly all rational surgeries on $K$ are determined completely by the slope and the values of $\cinvt(K)$ and $\delta(K)$.  Following \cite[Proposition~7.3]{kmos}, which carries out a similar approach in monopole Floer homology, we break up the computation of $I^\#(S^3_{p/q}(K))$ into a sequence of surgery exact triangles and see that they all split.  This will reduce the computation to an  argument by induction.

Throughout this section, we will write continued fractions in the notation
\[ [a_0,a_1,\dots,a_n] = a_0 - \frac{1}{a_1 - \frac{1}{\ddots - \frac{1}{a_n}}}. \]
Every $\frac{p}{q} \in \Q$ has a unique continued fraction of this form where the $a_i$ are all integers and $a_i \geq 2$ for all $i \geq 1$, though $a_0$ may be any integer.  The following lemma is standard (see e.g. \cite{hardy-wright}) but we include a proof because it is usually stated for continued fractions defined with a different sign convention.

\begin{lemma} \label{lem:convergents}
Let $a_0,a_1,\dots$ be a sequence of nonzero real numbers, possibly finite, and define a sequence $(p_n,q_n)$ for all $n \geq -1$ by
\begin{equation} \label{eq:convergent-recurrence}
\begin{aligned}
(p_{-1},q_{-1}) &= (1,0), \\
(p_0,q_0) &= (a_0,1), \\
(p_n,q_n) &= (a_np_{n-1}-p_{n-2}, a_nq_{n-1}-q_{n-2}) \qquad\mathrm{for\ all\ }n\geq 1.
\end{aligned}
\end{equation}
Then $[a_0,\dots,a_n] = \frac{p_n}{q_n}$ for all $n \geq 0$.  We also have $q_np_{n-1}-p_nq_{n-1} = 1$ for all $n \geq 0$, and if $a_n \geq 2$ for all $n \geq 1$ then the sequence $(q_n)$ is positive and strictly increasing for $n \geq 0$.
\end{lemma}

\begin{proof}
We induct on $n$; the lemma clearly holds for $n \leq 1$.  For $n = k \geq 2$ we observe that
\[ [a_0,\dots,a_{k-2},a_{k-1},a_k] = [a_0,\dots,a_{k-2},b]  \]
where $b = (a_{k-1}a_k-1)/a_k$, and by the $n=k-1$ case we know that the right side is equal to
\begin{align*} \frac{bp_{k-2}-p_{k-3}}{bq_{k-2}-q_{k-3}} &= \frac{(a_{k-1}a_k-1)p_{k-2} - a_kp_{k-3}}{(a_{k-1}a_k-1)q_{k-2}-a_kq_{k-3}} \\
&= \frac{ a_k(a_{k-1}p_{k-2}-p_{k-3}) - p_{k-2} }{ a_k(a_{k-1}q_{k-2}-q_{k-3}) - q_{k-2} } \\
&= \frac{ a_k p_{k-1} - p_{k-2} }{ a_k q_{k-1} - q_{k-2} } = \frac{p_k}{q_k},
\end{align*}
so the left side is also equal to $\frac{p_k}{q_k}$, as claimed.

For the remaining claims, we first compute for all $n \geq 1$ that
\begin{align*}
q_n p_{n-1} - p_n q_{n-1} &= (a_nq_{n-1}-q_{n-2})p_{n-1} - (a_np_{n-1}-p_{n-2})q_{n-1} \\
&= q_{n-1}p_{n-2} - p_{n-1}q_{n-2},
\end{align*}
so by induction this is equal to $q_0p_{-1} - p_0q_{-1} = 1$.  The assertion that the $q_n$ are positive and increasing if $a_n \geq 2$ for all $n \geq 1$ also follows by induction: we have $q_0=1$, and then if $a_n \geq 2$ and $q_{n-1} > 0$ then 
\[ q_n = a_nq_{n-1} - q_{n-2} \geq 2q_{n-1}-q_{n-2} = q_{n-1} + (q_{n-1}-q_{n-2}) > q_{n-1}. \qedhere \]
\end{proof}

\begin{proposition} \label{prop:two-triangles}
Let $\frac{p}{q} \not\in \Z$ be a rational number, where $p$ and $q$ are relatively prime and $q \geq 2$.  Then there is a pair of surgery exact triangles
\begin{align*}
\dots \to I^\#(S^3_{a/b}(K)) &\to I^\#(S^3_{p/q}(K)) \to I^\#(S^3_{c/d}(K)) \xrightarrow{F} \dots \\
\dots \to I^\#(S^3_{c/d}(K)) &\to I^\#(S^3_{e/f}(K)) \to I^\#(S^3_{a/b}(K)) \xrightarrow{G} \dots,
\end{align*}
where $a,b,c,d,e,f$ are integers satisfying the following:
\begin{itemize}
\item the pairs $(a,b)$, $(c,d)$, and $(e,f)$ are each relatively prime;
\item we have $b,d > 0$ and $f \geq 0$, with equality only if $e=1$;
\item we have $(p,q) = (a+c,b+d)$, and either $(a,b)=(c+e,d+f)$ or $(c,d) = (a+e,b+f)$;
\item both $\frac{a}{b}$ and $\frac{c}{d}$ lie between $\lfloor\frac{p}{q}\rfloor$ and $\lceil\frac{p}{q}\rceil$, inclusive; and so does $\frac{e}{f}$ unless $b=d=1$ and $\frac{e}{f}=\frac{1}{0}$.
\end{itemize}
We have $F\circ G = 0$ if $\frac{p}{q} \neq \frac{1}{q}$, and likewise if $\frac{p}{q} \neq -\frac{1}{q}$ then $G \circ F = 0$.
\end{proposition}

\begin{proof}
We write $\frac{p}{q}$ in terms of its continued fraction
\[ \frac{p}{q} = [a_0,a_1,\dots,a_n], \qquad a_k \geq 2 \mathrm{\ for\ }1 \leq k \leq n, \]
and define the sequence $(p_k,q_k)$ for $k \leq n$ as in \eqref{eq:convergent-recurrence}.  We note that $n \geq 1$ since $\frac{p}{q}$ is not an integer.  We can also write
\[ \frac{p}{q} = [a_0,\dots,a_{n-1},a_n-1,-1], \]
and so we define the rational slopes $\frac{a}{b}$ and $\frac{c}{d}$ by
\begin{align} \label{eq:triangle-ab-cd}
\frac{a}{b} &= [a_0,\dots,a_{n-1},a_n-1], & \frac{c}{d} &= [a_0,\dots,a_{n-1}].
\end{align}
We thus form a 3-manifold $Y$ by performing $a_0$-surgery on $K$ and then surgeries of slopes $a_1,\dots,a_{n-1},(a_n - 1)$ on a chain of meridians of $K$ of length $n$.  We then add one last meridian $m \subset Y$ to this chain and perform surgeries of slopes $\infty$, $-1$, or $0$ on it, as shown at the top of Figure~\ref{fig:slam-dunks}.
\begin{figure}
\begin{tikzpicture}
\begin{scope}
\draw[link] (0.5,0) circle (0.5);
\node[draw,fill=white,rectangle] at (0,0) {$K$};
\node[inner sep=0,below] at (0.5,-0.5) {$\strut a_0$};
\draw[link] (1.25,0) circle (0.5);
\draw[link] (0.5,0.5) arc (90:0:0.5);
\node[inner sep=0,below] at (1.25,-0.5) {\strut $a_1$};
\draw[link] (2,0) circle (0.5);
\draw[link] (1.25,0.5) arc (90:0:0.5);
\node[inner sep=0,below] at (2,-0.5) {\strut $a_2$};
\draw[link] (2.75,0.5) arc (90:270:0.5);
\draw[link] (2,0.5) arc (90:0:0.5);
\node at (3,0) {\small $\dots$};
\draw[link] (3.25,-0.5) arc (-90:90:0.5);
\draw[link] (4,0) circle (0.5);
\draw[link] (3.25,0.5) arc (90:0:0.5);
\node[inner sep=0,below] at (4,-0.5) {\strut $a_{n-1}$};
\draw[link] (4.75,0) circle (0.5);
\draw[link] (4,0.5) arc (90:0:0.5);
\node[inner sep=0,above] at (4.75,0.5) {\strut $a_n - 1$};
\draw[link] (5.5,0) circle (0.5);
\draw[link] (4.75,0.5) arc (90:0:0.5);
\node[inner sep=0,right] at (6.1,0) {\strut $\{\infty,-1,0\}$};
\node at (8.4,0) {\Large $\cong$};
\draw[link] (9.75,0) circle (0.5);
\node[draw,fill=white,rectangle] at (9.25,0) {$K$};
\node[inner sep=0,right] at (10.4,0) {$\strut \{\frac{a}{b},\frac{p}{q},\frac{c}{d}\}$};
\end{scope}

\begin{scope}[yshift=-2.5cm]
\draw[link] (0.5,0) circle (0.5);
\node[draw,fill=white,rectangle] at (0,0) {$K$};
\node[inner sep=0,below] at (0.5,-0.5) {$\strut a_0$};
\draw[link] (1.25,0) circle (0.5);
\draw[link] (0.5,0.5) arc (90:0:0.5);
\node[inner sep=0,below] at (1.25,-0.5) {\strut $a_1$};
\draw[link] (2,0) circle (0.5);
\draw[link] (1.25,0.5) arc (90:0:0.5);
\node[inner sep=0,below] at (2,-0.5) {\strut $a_2$};
\draw[link] (2.75,0.5) arc (90:270:0.5);
\draw[link] (2,0.5) arc (90:0:0.5);
\node at (3,0) {\small $\dots$};
\draw[link] (3.25,-0.5) arc (-90:90:0.5);
\draw[link] (4,0) circle (0.5);
\draw[link] (3.25,0.5) arc (90:0:0.5);
\node[inner sep=0,below] at (4,-0.5) {\strut $a_{n-1}$};
\draw[link] (4.75,0) circle (0.5);
\draw[link] (4,0.5) arc (90:0:0.5);
\node[inner sep=0,above] at (4.75,0.5) {\strut $a_n - 1$};
\draw[link] (5.5,0) circle (0.5);
\draw[link] (4.75,0.5) arc (90:0:0.5);
\node[inner sep=0,below] at (5.5,-0.5) {\strut $0$};
\draw[link] (6.25,0) circle (0.5);
\draw[link] (5.5,0.5) arc (90:0:0.5);
\node[inner sep=0,below] at (7.25,-0.5) {\strut $\{\infty,-1,0\}$};
\node at (8.4,0) {\Large $\cong$};
\draw[link] (9.75,0) circle (0.5);
\node[draw,fill=white,rectangle] at (9.25,0) {$K$};
\node[inner sep=0,right] at (10.4,0) {$\strut \{\frac{c}{d},\frac{e}{f},\frac{a}{b}\}$};
\end{scope}
\end{tikzpicture}
\caption{Realizing rational surgeries on $K$ as integral surgeries on the union of $K$ and a chain of meridians.}
\label{fig:slam-dunks}
\end{figure}
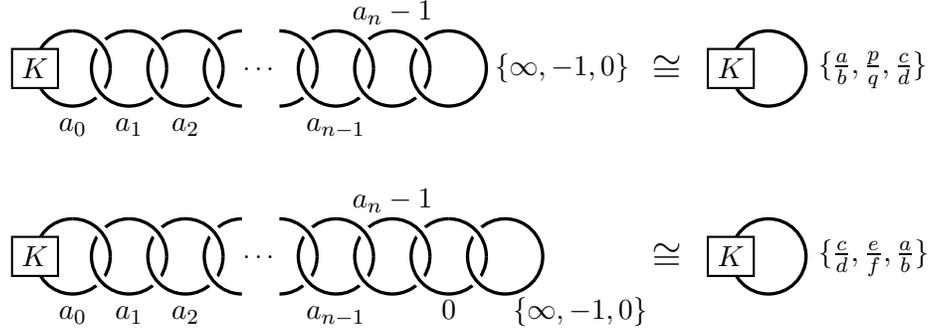
These surgeries on $m$ fit into a surgery exact triangle
\[ \dots \to I^\#(Y) \to I^\#(Y_{-1}(m)) \to I^\#(Y_0(m)) \to \dots \]
by Theorem~\ref{thm:exact-triangle} and \eqref{eq:triangle-untwisted}, and a series of slam dunks identify the results of these surgeries with $\frac{a}{b}$-, $\frac{p}{q}$-, and $\frac{c}{d}$-surgery on $K \subset S^3$, respectively.

To determine the values of $\frac{a}{b}$ and $\frac{c}{d}$, we note that $\frac{p}{q} = \frac{p_n}{q_n}$ by Lemma~\ref{lem:convergents}, and applying the same lemma to \eqref{eq:triangle-ab-cd} gives
\begin{align*}
\frac{a}{b} &= \frac{(a_n-1)p_{n-1}+p_{n-2}}{(a_n-1)q_{n-1}-q_{n-2}} = \frac{p_n - p_{n-1}}{q_n-q_{n-1}}, &
\frac{c}{d} &= \frac{p_{n-1}}{q_{n-1}}.
\end{align*}
Lemma~\ref{lem:convergents} says that $0 < q_{n-1} < q_n$, from which we see that $b$ and $d$ are strictly positive and $(p,q)=(a+c,b+d)$; and that $qc - pd = 1$.  From the latter it is easy to compute that $pb-qa = p(q-d) - q(p-c) = qc - pd = 1$, and likewise $bc-ad = b(p-a)-a(q-b) = pb-qa = 1$.  These relations tell us that $(a,b)$ and $(c,d)$ are relatively prime, and the relation $bc-ad=1$ also says that
\begin{itemize}
\item $b\neq d$ unless $b=d=1$ and $c=a+1$;
\item $a\neq 0$ unless $b=c=1$, in which case $\frac{p}{q} = \frac{a+c}{b+d} = \frac{1}{d+1}$.
\end{itemize}

Next, we let $\frac{e}{f} = [a_0,\dots,a_{n-1},a_n-2]$; in the event that $a_n=2$, we interpret this as $[a_0,\dots,a_{n-2}]$ if $n \geq 2$ and as $\frac{1}{0}$ if $n=1$.  We then fix the framing $0$ on $m$, and consider doing further Dehn surgeries of slopes $\infty,-1,0$ on a meridian $m'$ of $m$, as shown at the bottom of Figure~\ref{fig:slam-dunks}.  These amount to rational surgeries on $K$ with slopes
\begin{align*}
[a_0,\dots,a_{n-1}] &= \frac{c}{d}, & [a_0,\dots,a_{n-1},a_n-2] &= \frac{e}{f}, & [a_0,\dots,a_{n-1},a_n-1] &= \frac{a}{b}
\end{align*}
respectively, hence yield a surgery exact triangle
\[ \dots \to I^\#(S^3_{c/d}(K)) \to I^\#(S^3_{e/f}(K)) \to I^\#(S^3_{a/b}(K)) \to \dots. \]
As above, we have
\[ \frac{e}{f} = \frac{(a_n-2)p_{n-1} - p_{n-2}}{(a_n-2)q_{n-1}-q_{n-2}} = \frac{a-p_{n-1}}{b-q_{n-1}} = \frac{a - c}{b - d}, \]
though if we insist that $f \geq 0$ then the precise signs of $e$ and $f$ depend on whether $a_n$ is equal to or greater than $2$, namely
\[ (e,f) = \begin{cases} (1,0), & b=d \\ (a-c,b-d), & b>d \\ (c-a,d-b), & d<b. \end{cases} \]
Thus either $(a,b)=(c+e,d+f)$ or $(c,d)=(a+e,b+f)$, as claimed; and we have
\[ fc-ed = \pm\left((b-d)c - (a-c)d\right) = \pm(bc-ad) = \pm1 \]
and similarly $fa-eb = \pm1$, so $e$ and $f$ are relatively prime as well.

To bound the values of $\frac{a}{b}$, $\frac{c}{d}$, and $\frac{e}{f}$, we write $N = \lfloor \frac{p}{q} \rfloor$.  Then we have
\[ N \leq \frac{p}{q}-\frac{1}{q} < \frac{p}{q}+\frac{1}{q} \leq N+1, \]
so the relations $qc-pd=1$ and $pb-qa=1$ tell us that
\[ N \leq \frac{p}{q}-\frac{1}{bq} = \frac{a}{b} < \frac{p}{q} < \frac{c}{d} = \frac{p}{q}+\frac{1}{dq} \leq N+1 \]
with equality at the left or right end only if $b=1$ or $d=1$, respectively.  If either $\frac{a}{b}$ or $\frac{c}{d}$ is not an integer then it must also lie strictly between $N$ and $N+1$, and the same argument with $fa-eb=\pm1$ or $fc-ed=\pm1$ respectively says that $\frac{e}{f}$ lies between $N$ and $N+1$ inclusive.  Thus if $\frac{e}{f} \not\in [N,N+1]$ then we must have $b=d=1$ and $\frac{e}{f}=\frac{1}{0}$.

Finally, we consider the composite
\[ I^\#(S^3_{a/b}(K)) \xrightarrow{G} I^\#(S^3_{c/d}(K)) \xrightarrow{F} I^\#(S^3_{a/b}(K)) \]
corresponding to the cobordism $X$ in which we attach a pair of $0$-framed $2$-handles to $S^3_{a/b}(K) \times [0,1]$ along $m\times\{1\}$ and $m'\times\{1\}$.  Just as in Lemma~\ref{lem:cancel-2-handles}, we find an embedded 2-sphere $S \subset X$ of self-intersection $0$ as the union of the cocore of the $m$ handle and the core of the $m'$ handle.  The core of the $m$ handle has boundary a knot which generates $H_1(S^3_{a/b}(K))$, so if $a\neq 0$ (which we have seen is true unless $\frac{p}{q} = \frac{1}{d+1}$) then we can build a surface $F \subset X$ with $S \cdot F > 0$ by taking $|a|$ parallel cores of the $m$ handle and gluing them to a Seifert surface for their boundary in $S^3_{a/b}(K)$.  In this case Proposition~\ref{prop:adjunction-inequality} says that $F \circ G = 0$, as desired.

An identical argument applies to the composition
\[ I^\#(S^3_{c/d}(K)) \xrightarrow{F} I^\#(S^3_{a/b}(K)) \xrightarrow{G} I^\#(S^3_{c/d}(K), \]
which can be realized by attaching 0-framed 2-handles to $S^3_{c/d}(K) \times [0,1]$ along $m' \times \{1\}$ and a meridian of $m' \times \{1\}$.  Here we can conclude that $G \circ F = 0$ unless $\frac{c}{d}=0$, in which case $\frac{c}{d} = [a_0,\dots,a_{n-1}]$ implies that $n=1$ and $a_0=0$, and so $\frac{p}{q} = [0,a_1] = -\frac{1}{a_1}$.
\end{proof}

We can now compute $I^\#(S^3_{p/q}(K))$ for most $\frac{p}{q}$ by checking that the exact triangles of Proposition~\ref{prop:two-triangles} are split. The main result of this section, Theorem \ref{thm:rational-surgeries}, follows from the two preliminary results below.

\begin{lemma} \label{lem:compute-1/n}
If $\cinvt(K) \leq 0$, then we have
\[ \dim I^\#(S^3_{p/q}(K)) = p + q \cdot \delta(K) \]
for any $\frac{p}{q}$ of the form $\frac{1}{n}$ with $n \geq 1$.  If $\cinvt(K) \leq -1$, then the same holds for $\frac{p}{q}=\frac{0}{1}$ and all $\frac{p}{q} = \frac{-1}{n}$ with $n \geq 1$ as well.
\end{lemma}

\begin{proof}
We take $\lambda \in H_1(S^3_0(K);\Z/2\Z)$ to be $0$ if $K$ is V-shaped and the class $[\mu]$ of a meridian if $K$ is W-shaped.  If $K$ is V-shaped then
\begin{align*}
\dim I^\#(S^3_0(K)) &= \dim I^\#(S^3_{\cinvt(K)}(K)) + |\cinvt(K)|\\
&= r_0(K) - \cinvt(K) = \delta(K)
\end{align*}
by \eqref{eq:r_0} and the assumption that $\cinvt(K) \leq 0$, whereas if $K$ is W-shaped then $\cinvt(K)=0$ and
\[ \delta(K) = r_0(K) - \cinvt(K) = \dim I^\#(S^3_0(K),\mu). \]
Thus $\delta(K) = \dim I^\#(S^3_0(K),\lambda)$ in both cases.  In particular, if $\cinvt(K) < 0$ then we have
\[ \dim I^\#(S^3_{0/1}(K)) = 0 + 1 \cdot \delta(K), \]
so from now on we can assume that $\frac{p}{q} \neq 0$.

For the case $\frac{p}{q} = \frac{1}{n}$ with $n \geq 1$, we first observe that $\frac{1}{n} = [1,2,\dots,2]$, as a continued fraction of total length $n$.  Thus if $\frac{p}{q}=\frac{1}{n}$ with $n \geq 2$, then Proposition~\ref{prop:two-triangles} gives us a pair of exact triangles
\begin{align*}
\dots \to I^\#(S^3_{0/1}(K),\lambda) &\to I^\#(S^3_{1/n}(K)) \to I^\#(S^3_{1/(n-1)}(K)) \xrightarrow{F_n} \dots \\
\dots \to I^\#(S^3_{1/(n-1)}(K)) &\to I^\#(S^3_{1/(n-2)}(K)) \to I^\#(S^3_{0/1}(K),\lambda) \xrightarrow{G_n} \dots
\end{align*}
with $G_n \circ F_n = 0$.  (The possibility that $\lambda \neq 0$ does not change anything, since the proof that $G_n \circ F_n = 0$ is insensitive to the choice of bundles over the cobordisms.)

We claim that $G_n$ is injective and $F_n = 0$ for $n \geq 2$; it suffices to prove that each $G_n$ is injective, since then $G_n \circ F_n = 0$ implies that $F_n = 0$.  When $n=2$, we have
\[ \dim I^\#(S^3_1(K)) = \dim I^\#(S^3_0(K),\lambda) + \dim I^\#(S^3_\infty(K)) = \delta(K) + 1 \] (since $\cinvt(K)\leq 0$),
so $G_2$ is injective.  Now if $G_n$ is injective for some $n \geq 2$ then we have $F_n=0$, hence
\[ \dim I^\#(S^3_{1/n}(K)) = \dim I^\#(S^3_{1/(n-1)}(K)) + \dim I^\#(S^3_0(K),\lambda), \]
and this guarantees by exactness that $G_{n+1}$ must be injective as well.  The claim follows by induction, and similarly this last equation implies by induction that
\begin{align*}
\dim I^\#(S^3_{1/n}(K)) &= \dim I^\#(S^3_1(K)) + (n-1)\cdot \dim I^\#(S^3_0(K),\lambda) \\
&= (\delta(K)+1) + (n-1) \cdot\delta(K) = 1 + n\cdot \delta(K)
\end{align*}
for all $n \geq 1$, proving the lemma for all $\frac{p}{q} = \frac{1}{n}$ with $n \geq 1$.

Assuming now that $\cinvt(K) \leq -1$ and hence that $\lambda=0$, the case $\frac{p}{q} = -\frac{1}{n}$ with $n \geq 2$ is similar.  Here we have $-\frac{1}{n} = [0,n]$, so the exact triangles from Proposition~\ref{prop:two-triangles} have the form
\begin{align*}
\dots \to I^\#(S^3_{-1/(n-1)}(K)) \to I^\#(S^3_{-1/n}(K)) &\to I^\#(S^3_{0/1}(K)) \xrightarrow{F_n} \dots \\
\dots \to I^\#(S^3_{0/1}(K),\mu) \to I^\#(S^3_{-1/(n-2)}(K)) &\to I^\#(S^3_{-1/(n-1)}(K)) \xrightarrow{G_n} \dots
\end{align*}
with $F_n \circ G_n = 0$; note that this time we have explicitly placed different bundles on $S^3_0(K)$ in each triangle, but that $\dim I^\#(S^3_0(K),\mu) = \dim I^\#(S^3_0(K),0)$ since $\cinvt(K) \neq 0$, by Proposition \ref{prop:unimodal-rank}.  Now the composition
\[ I^\#(S^3_{0/1}(K)) \xrightarrow{F_2} I^\#(S^3_{-1}(K)) \xrightarrow{G_2} I^\#(S^3_{0/1}(K),\mu) \]
is zero because it comes from a cobordism with an embedded 2-sphere of self-intersection zero on which the associated $SO(3)$-bundle is nontrivial, exactly as in the proof of Proposition~\ref{prop:unimodal-rank}.  From $\cinvt(K) \leq -1$ we deduce that
\[ \dim I^\#(S^3_0(K)) = \dim I^\#(S^3_{-1}(K)) + \dim I^\#(S^3_\infty(K)), \]
so $G_2$ is injective and thus $F_2$ must be zero.  Now if $F_n$ is zero then by exactness $G_{n+1}$ must be surjective, so we use the fact that $F_{n+1}\circ G_{n+1} = 0$ to deduce that $F_{n+1}=0$ as well.  We induct to conclude that $F_n = 0$ for all $n \geq 2$, and then that
\begin{align*}
\dim I^\#(S^3_{-1/n}(K)) &= (n-1)\dim I^\#(S^3_0(K)) + \dim I^\#(S^3_{-1}(K)) \\
&= (n-1)\delta(K) + (\delta(K) - 1),
\end{align*}
which is equal to $-1 + n\cdot\delta(K)$, as desired.
\end{proof}

\begin{proposition} \label{prop:rational-rank-above-nu}
Fix a rational number $\frac{p}{q} \geq \cinvt(K)$, with $p$ and $q$ relatively prime and $q \geq 1$.  Then
\[ \dim I^\#(S^3_{p/q}(K)) = p + q\cdot\delta(K), \]
except if $\frac{p}{q}=\frac{0}{1}$ and $K$ is W-shaped, in which case we have $\dim I^\#(S^3_0(K),\mu) = \delta(K)$ instead.
\end{proposition}

\begin{proof}
Throughout this proof we will always interpret ``$I^\#(S^3_0(K))$'' as $I^\#(S^3_0(K),\mu)$ if $K$ is W-shaped, and $I^\#(S^3_0(K),0)$ otherwise.

In the case $q=1$, we recall from the definition of $\delta(K)$ that
\[ \dim I^\#(S^3_p(K)) = (\delta(K) + \cinvt(K)) + |p - \cinvt(K)| = p + \delta(K), \]
since by assumption $p \geq \cinvt(K)$.  We also note that the desired formula holds for $\frac{p}{q}=\frac{1}{0}$, since $\dim I^\#(S^3) = 1$; and if $p=\pm 1$ then the formula follows from Lemma~\ref{lem:compute-1/n}.

Now suppose that $q \geq 2$ and $|p| \geq 2$, and that we have proved the proposition for all $\frac{r}{s}$ with $1 \leq s < q$.  We consider the pair of exact triangles
\begin{align*}
\dots \to I^\#(S^3_{a/b}(K)) &\to I^\#(S^3_{p/q}(K)) \to I^\#(S^3_{c/d}(K)) \xrightarrow{F} \dots \\
\dots \to I^\#(S^3_{c/d}(K)) &\to I^\#(S^3_{e/f}(K)) \to I^\#(S^3_{a/b}(K)) \xrightarrow{G} \dots
\end{align*}
provided by Proposition~\ref{prop:two-triangles}; since $\frac{p}{q} \neq \pm\frac{1}{q}$ we have $F\circ G = G\circ F = 0$.  We note that $\frac{a}{b}$ and $\frac{c}{d}$ are both at least $\lfloor\frac{p}{q}\rfloor \geq \cinvt(K)$, and that $\frac{e}{f}$ satisfies the same inequality unless $\frac{e}{f} = \frac{1}{0}$.

We first claim that the map $G$ is either injective or surjective.  Since $\frac{a}{b}$, $\frac{c}{d}$, and $\frac{e}{f}$ are all at least $\cinvt(K)$ or equal to $\frac{1}{0}$, and since $0 \leq b,d,f < q$, we already know by hypothesis that
\begin{align*}
\dim I^\#(S^3_{c/d}(K)) &= c+d\cdot\delta(K), \\
\dim I^\#(S^3_{e/f}(K)) &= e+f\cdot\delta(K), \\
\dim I^\#(S^3_{a/b}(K)) &= a+b\cdot\delta(K).
\end{align*}
If $(a,b) = (c+e,d+f)$ then it follows that
\[ \dim I^\#(S^3_{a/b}(K)) = \dim I^\#(S^3_{c/d}(K)) + \dim I^\#(S^3_{e/f}(K)) \]
and so the map $G$ must be surjective.  Otherwise we have $(c,d) = (a+e,b+f)$, and then $G$ is injective.

If $G$ is surjective, then $F \circ G = 0$ implies that $F = 0$.  Otherwise we use $G \circ F = 0$ and the injectivity of $G$ to conclude that $F = 0$.  In any case, the first exact triangle above splits and we have
\begin{align*}
\dim I^\#(S^3_{p/q}(K)) &= \dim I^\#(S^3_{a/b}(K)) + \dim I^\#(S^3_{c/d}(K)) \\
&= (a+b\cdot\delta(K)) + (c+d\cdot\delta(K)) \\
&= p + q\cdot\delta(K),
\end{align*}
the last row making use of the facts that $p=a+c$ and $q=b+d$.  Thus the proposition holds for $\frac{p}{q}$, and so it follows in general by induction on $q$.
\end{proof}

We now prove the main theorem of this section, which implies Theorem \ref{thm:main-surgery}:

\begin{theorem} \label{thm:rational-surgeries}
Let $p$ and $q$ be relatively prime integers with $q \geq 1$.  If $K$ is W-shaped, then suppose in addition that $p\neq 0$.  Then
\[ \dim I^\#(S^3_{p/q}(K)) = q\cdot r_0(K) + |p-q\cinvt(K)|, \]
where $r_0(K)$ is the integer defined in \eqref{eq:r_0}.
\end{theorem}

\begin{proof}
Equation~\eqref{eq:delta-expression} says that $r_0(K) = \cinvt(K) + \delta(K)$.  Thus if $\frac{p}{q} \geq \cinvt(K)$ then we wish to show that
\[ \dim I^\#(S^3_{p/q}(K)) = p+q\cdot\delta(K), \]
which is already the content of Proposition~\ref{prop:rational-rank-above-nu}.  We may therefore assume from now on that $\frac{p}{q} < \cinvt(K)$, in which case we must prove that
\[ \dim I^\#(S^3_{p/q}(K)) = (-p + q\cdot\delta(K)) + 2q\cdot\cinvt(K). \]

Since $-\frac{p}{q} > -\cinvt(K) = \cinvt(\mirror{K})$, we apply Proposition~\ref{prop:rational-rank-above-nu} to $-\frac{p}{q}$-surgery on $\mirror{K}$ to get
\[ \dim I^\#(S^3_{p/q}(K)) = \dim I^\#(S^3_{-p/q}(\mirror{K})) = -p + q\cdot\delta(\mirror{K}). \]
Using the definition of $\delta$, we compute that
\begin{align*}
\delta(K) - \delta(\mirror{K}) &= \left(\dim I^\#(S^3_1(K)) - \cinvt(K) - |1-\cinvt(K)|\right) \\&\qquad{} - \left(\dim I^\#(S^3_{-1}(\mirror{K})) - \cinvt(\mirror{K}) - |{-}1-\cinvt(\mirror{K})|\right) \\
&= \left(\cinvt(\mirror{K}) - \cinvt(K)\right) + \left(|{-1}-\cinvt(\mirror{K})| - |1-\cinvt(K)|\right) \\
&= -2\cinvt(K),
\end{align*}
so we conclude that
\[ \dim I^\#(S^3_{p/q}(K)) = -p + q\left(\delta(K) + 2\cinvt(K)\right). \]
This simplifies to the desired expression.
\end{proof}

In Example~\ref{ex:unknot-plot} we only needed the fact that $\dim I^\#(S^3) = 1$ to show that the unknot satisfies $\cinvt(U) = r_0(U) = 0$.  In conjunction with Theorem~\ref{thm:rational-surgeries}, this implies the following.

\begin{corollary} \label{cor:lens-spaces}
For any relatively prime integers $p > q \geq 1$ we have
\[ \dim I^\#(L(p,q)) = \dim I^\#(S^3_{p/q}(U)) = p. \]
In other words, all lens spaces are instanton L-spaces.
\end{corollary}

We can now complete the proof of Theorem \ref{thm:nu-l-space}, effectively  begun in Proposition \ref{prop:nu-l-space}.

\begin{proof}[Proof of Theorem \ref{thm:nu-l-space}]
Suppose $K\subset S^3$ is an instanton L-space knot. Proposition \ref{prop:nu-l-space} says that \[\cinvt(K)=r_0(K)=2g(K)-1.\] Since $K$ is nontrivial by definition, $\cinvt(K)>0$, which means that $K$ is V-shaped. Therefore, Theorem \ref{thm:rational-surgeries} says that \[ \dim I^\#(S^3_{p/q}(K)) = q(2g(K)-1) + |p-q(2g(K)-1)| \] for all relatively prime integers $p$ and $q$ with $q\geq 1$, which immediately yields the expression claimed in Theorem \ref{thm:nu-l-space}. \end{proof}

We now turn to the proof of Theorem \ref{thm:l-space-cable}, regarding when cables are instanton L-space knots.
Let $K_{p,q}$ denote the $(p,q)$-cable of $K$, where $q \geq 2$ and $\gcd(p,q)=1$; this is the knot represented by a curve in $\partial N(K)$ in the homotopy class $\mu^p\lambda^q$.  We have
\[ g(K_{p,q}) = \frac{(|p|-1)(q-1)}{2} + q\cdot g(K) \]
\cite[\S21, Satz 1]{schubert}, or equivalently
\begin{equation} \label{eq:cable-genus}
2g(K_{p,q})-1 = |p|q + q\left(2g(K)-1 - \frac{|p|}{q}\right).
\end{equation}

\begin{lemma} \label{lem:cinvt-cable}
If $\cinvt(K) < \frac{p}{q}$ then $\cinvt(K_{p,q}) \leq pq-1$.
\end{lemma}

\begin{proof}
We use the identities
\begin{align*}
S^3_{(pq-1)/q^2}(K) &\cong S^3_{pq-1}(K_{p,q}), &
S^3_{(pq+1)/q^2}(K) &\cong S^3_{pq+1}(K_{p,q})
\end{align*}
from \cite[Corollary~7.3]{gordon}.  If $\cinvt(K) < \frac{p}{q}$ then
\[ \frac{pq \pm 1}{q^2} > \frac{p-1}{q} \geq \cinvt(K), \]
so that $pq \pm 1 \geq q^2\cinvt(K)$, and hence by Theorem~\ref{thm:rational-surgeries} we have
\begin{align*}
\dim I^\#(S^3_{(pq-1)/q^2}(K)) &= q^2r_0(K) + \big((pq-1) - q^2\cinvt(K)\big) \\
\dim I^\#(S^3_{(pq+1)/q^2}(K)) &= q^2r_0(K) + \big((pq+1) - q^2\cinvt(K)\big).
\end{align*}
Since $\dim I^\#(S^3_{pq+1}(K_{p,q})) = \dim I^\#(S^3_{pq-1}(K_{p,q})) + 2$, we must have $\cinvt(K_{p,q}) \leq pq-1$.
%
\end{proof}

\begin{proposition} \label{prop:lspace-cabling-slope}
If $K_{p,q}$ is an instanton L-space knot, then $\frac{p}{q} > 2g(K)-1$.
\end{proposition}

\begin{proof}
Supposing otherwise, we have $\frac{p}{q} < 2g(K)-1$; equality cannot hold since we have assumed $q \geq 2$.  Since $K_{p,q}$ is an instanton L-space knot, we know that $\cinvt(K_{p,q}) = r_0(K_{p,q}) = 2g(K_{p,q})-1$.  Now $pq < \cinvt(K_{p,q})$ holds automatically if $p<0$, and it is implied by \eqref{eq:cable-genus} if $p>0$, so $pq-1 < \cinvt(K_{p,q})$ as well and we have
\begin{align*}
\dim I^\#(S^3_{(pq-1)/q^2}(K)) &= \dim I^\#(S^3_{pq-1}(K_{p,q})) \\
&= (2g(K_{p,q})-1) + \big((2g(K_{p,q})-1) - (pq-1)\big) \\
&= 2|p|q - pq + 2q(2g(K)-1) - 2|p| + 1
\end{align*}
by Theorem~\ref{thm:rational-surgeries}.  On the other hand, if $\frac{p}{q} > \cinvt(K)$ then by Lemma~\ref{lem:cinvt-cable} we would have $\cinvt(K_{p,q}) < pq$, and we have seen that this is false, so in fact $\frac{pq-1}{q^2} < \frac{p}{q} < \cinvt(K)$; and now by Theorem~\ref{thm:rational-surgeries} we have
\begin{align*}
\dim I^\#(S^3_{(pq-1)/q^2}(K)) &= q^2r_0(K) + \big(q^2\cinvt(K) - (pq-1)\big) \\
&= q^2(r_0(K) + \cinvt(K)) - pq + 1.
\end{align*}
Combining these two expressions for $\dim I^\#(S^3_{(pq-1)/q^2}(K))$, we have
\[ q^2(r_0(K) + \cinvt(K)) - pq + 1 = 2|p|q - pq + 2q(2g(K)-1) - 2|p| + 1, \]
which after some rearranging becomes
\[ q^2(r_0(K) + \cinvt(K)) = 2|p|(q-1) + 2q(2g(K)-1). \]
Reducing modulo $q$ gives $2|p| \equiv 0 \pmod{q}$, and since $\gcd(p,q)=1$ and $q\geq 2$ this implies that $q=2$; then we have $4(r_0(K)+\cinvt(K)) = 2|p| + 4(2g(K)-1)$, and we reduce this modulo $4$ to see that $p$ is even as well, contradicting $\gcd(p,q)=1$.  We conclude that $\frac{p}{q} > 2g(K)-1$, as claimed.
\end{proof}

We may now prove Theorem \ref{thm:l-space-cable}:

\begin{proof}[Proof of Theorem \ref{thm:l-space-cable}]
If $K_{p,q}$ is an instanton L-space knot, then Proposition~\ref{prop:lspace-cabling-slope} says that $\frac{p}{q} > 2g(K)-1$ and so
\[ pq > 2g(K_{p,q})-1 = \cinvt(K_{p,q}) \]
by \eqref{eq:cable-genus}.  It follows that $S^3_{pq+1}(K_{p,q})$ is an instanton L-space, since $pq+1 \geq \cinvt(K_{p,q})$ as well, and then
\[ S^3_{(pq+1)/q^2}(K) \cong S^3_{pq+1}(K_{p,q}) \]
is an instanton L-space surgery of positive rational slope on $K$.  Thus
\[ pq+1 = q^2 r_0(K) + |(pq+1) - q^2\cinvt(K)| \]
by Theorem~\ref{thm:rational-surgeries}.  If $pq+1 \leq q^2\cinvt(K)$ then this becomes $2(pq+1) = q^2(r_0(K)+\cinvt(K))$, and since $\gcd(pq+1,q^2)=1$ we must have $q^2 \mid 2$, which is impossible; thus $pq+1 > q^2\cinvt(K)$ and this equation simplifies to $r_0(K) = \cinvt(K)$, hence $K$ has positive integral instanton L-space surgeries.

In the other direction, if $K$ is an instanton L-space knot and $\frac{p}{q} > 2g(K)-1$ then
\[ \frac{pq+1}{q^2} > \frac{p}{q} > \cinvt(K) \]
and so $S^3_{(pq+1)/q^2}(K)$ is an instanton L-space.  This is homeomorphic to $S^3_{pq+1}(K_{p,q})$, which is therefore a positive integral instanton L-space surgery on $K_{p,q}$.
\end{proof}

\section{The concordance invariant $\cinvt$ is a quasi-morphism} \label{sec:quasi-morphism}

The goal of this section is to prove the following theorem and derive some consequences, one of which will be our definition of the concordance homomorphism $\chominvt$.

\begin{theorem} \label{thm:nu-quasi-hom}
The smooth concordance invariant $\cinvt$ satisfies
\[ |\cinvt(K\#L) - \cinvt(K) - \cinvt(L)| \leq 1 \]
for all knots $K,L \subset S^3$.
\end{theorem}

Theorem~\ref{thm:conc-invt} tells us that if $K$ is not W-shaped then
\begin{equation} \label{eq:cinvt-rank}
\dim I^\#(S^3_n(K)) = \dim I^\#(S^3_{\cinvt(K)}(K)) + |\cinvt(K) - n|
\end{equation}
for all $n\in\Z$.  In these cases the 2-handle cobordism map
\[ F_{K,n} = I^\#(X_n(K),\nu_n): I^\#(S^3) \to I^\#(S^3_n(K)) \]
is injective for all $n < \cinvt(K)$ and zero for all $n \geq \cinvt(K)$.

If instead $K$ is W-shaped, then we have $\cinvt(K)=0$, and \eqref{eq:cinvt-rank} fails because
\[ \dim I^\#(S^3_{\pm1}(K)) = \dim I^\#(S^3_0(K)) - 1. \]
In this case $F_{K,n}$ is injective for all $n \leq -2$ and $n=0$, and zero otherwise.

\begin{lemma} \label{lem:connected-sum-surgery}
Let $K,L \subset S^3$ be knots, and fix $a,b\in\Z$.  If $F_{K,a}$ and $F_{L,b}$ are both injective, then so is $F_{K\#L,a+b}$.  Equivalently, if $F_{K\#L,a+b} = 0$, then either $F_{K,a}=0$ or $F_{L,b} = 0$.
\end{lemma}

\begin{proof}
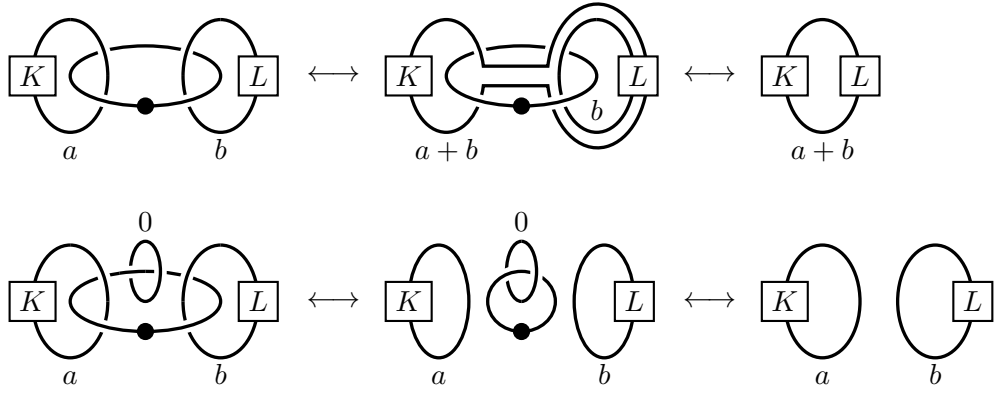
\begin{figure}
\begin{tikzpicture}
\begin{scope}
\coordinate (K1) at (0,0);
\coordinate (L1) at (3,0);
\draw[link] (K1) arc (180:-180:0.5 and 0.75) node[midway, inner sep=0,outer sep=0] (K1m) {} node[black,pos=0.75,inner sep=0,below] {$\strut a$};
\node[draw,fill=white,rectangle] at (K1) {$K$};
\draw[link] (L1) arc (0:360:0.5 and 0.75) node[midway, inner sep=0,outer sep=0] (L1m) {} node[black,pos=0.75,inner sep=0, below] {$\strut b$};
\node[draw,fill=white,rectangle] at (L1) {$L$};
\draw[link] let \p{KL} = ($(L1m)-(K1m)$) in ($(K1m)!0.5!(L1m)$) ellipse ({\x{KL}} and 0.4cm);
\node[draw,fill=black,circle,minimum size=0.2cm, inner sep=0] at ($(K1m)!0.5!(L1m) + (0,-0.4)$) {};
\begin{scope}
\clip (K1) rectangle ($(L1)+(0,1)$);
\draw[link] (K1m) arc (0:90: 0.5 and 0.75);
\draw[link] (L1m) arc (180:90: 0.5 and 0.75);
\end{scope}
\end{scope}
\node at (4,0) {$\longleftrightarrow$};
\begin{scope}[xshift=5cm]
\coordinate (K1) at (0,0);
\coordinate (L1) at (3,0);
\path[link] (K1) arc (180:-180:0.5 and 0.75) node[midway, inner sep=0,outer sep=0] (K1m) {} node[black,pos=0.75,inner sep=0,below] {$\strut a+b$};
\draw[link] (L1) arc (0:360:0.5 and 0.75) node[midway, inner sep=0,outer sep=0] (L1m) {} node[black,pos=0.75,inner sep=0, above] {$\strut b$};
\draw[link] (K1) arc (180:350:0.5 and 0.75) coordinate (K1b) -- ($(K1b -| L1m)+(-0.15,0)$) arc(190:360:0.67 and 1);
\draw[link] let \p{KL} = ($(L1m)-(K1m)$) in ($(K1m)!0.5!(L1m)$) ellipse ({\x{KL}} and 0.4cm);
\node[draw,fill=black,circle,minimum size=0.2cm, inner sep=0] at ($(K1m)!0.5!(L1m) + (0,-0.4)$) {};
\draw[link] (K1) arc (180:10:0.5 and 0.75) coordinate (K1t) -- ($(K1t -| L1m)+(-0.15,0)$) arc(170:0:0.67 and 1);
\node[draw,fill=white,rectangle] at (K1) {$K$};
\node[draw,fill=white,rectangle] at ($(L1)+(0.05,0)$) {$L$};
\begin{scope}
\clip (K1) rectangle ($(L1)+(0,1)$);
\draw[link] (L1m) arc (180:90: 0.5 and 0.75);
\end{scope}
\end{scope}
\node at (9,0) {$\longleftrightarrow$};
\begin{scope}[xshift=10cm]
\coordinate (K1) at (0,0);
\draw[link] (K1) arc (180:-180:0.5 and 0.75) node[midway] (L1) {} node[black,pos=0.75,inner sep=0,below] {$\strut a+b$};
\node[draw,fill=white,rectangle] at (K1) {$K$};
\node[draw,fill=white,rectangle] at (L1) {$L$};
\end{scope}

\begin{scope}[yshift=-3cm]
\coordinate (K1) at (0,0);
\coordinate (L1) at (3,0);
\draw[link] (K1) arc (180:-180:0.5 and 0.75) node[midway, inner sep=0,outer sep=0] (K1m) {} node[black,pos=0.75,inner sep=0,below] {$\strut a$};
\node[draw,fill=white,rectangle] at (K1) {$K$};
\draw[link] (L1) arc (0:360:0.5 and 0.75) node[midway, inner sep=0,outer sep=0] (L1m) {} node[black,pos=0.75,inner sep=0, below] {$\strut b$};
\node[draw,fill=white,rectangle] at (L1) {$L$};
\draw[link] let \p{KL} = ($(L1m)-(K1m)$) in ($(K1m)!0.5!(L1m)$) ellipse ({\x{KL}} and 0.4cm);
\node[draw,fill=black,circle,minimum size=0.2cm, inner sep=0] at ($(K1m)!0.5!(L1m) + (0,-0.4)$) {};
\coordinate (M) at ($(K1m)!0.5!(L1m) + (0,0.4)$);
\draw[link] (M) ellipse (0.2 and 0.4);
\node[above] at ($(M)+(0,0.4)$) {$0$};
\draw[link] let \p{KL} = ($(L1m)-(K1m)$) in (M) arc (90:110:{\x{KL}} and 0.4cm);
\begin{scope}
\clip (K1) rectangle ($(L1)+(0,1)$);
\draw[link] (K1m) arc (0:90: 0.5 and 0.75);
\draw[link] (L1m) arc (180:90: 0.5 and 0.75);
\end{scope}
\end{scope}
\node at (4,-3) {$\longleftrightarrow$};
\begin{scope}[xshift=5cm,yshift=-3cm]
\coordinate (K1) at (0,0);
\coordinate (L1) at (3,0);
\draw[link] (K1) arc (180:-180:0.4 and 0.75) node[black, pos=0.75, inner sep=0, below] {$\strut a$};
\node[draw,fill=white,rectangle] at (K1) {$K$};
\draw[link] (L1) arc (0:360:0.4 and 0.75) node[black, pos=0.75, inner sep=0, below] {$\strut b$};
\node[draw,fill=white,rectangle] at (L1) {$L$};
\coordinate (M) at ($(K1)!0.5!(L1) + (0,0.4)$);
\draw[link] (M) ellipse (0.2 and 0.4);
\node[above] at ($(M)+(0,0.4)$) {$0$};
\draw[link] let \p{KL} = ($(L1)-(K1)$) in ($(K1)!0.5!(L1)$) ellipse ({\x{KL}*0.15} and 0.4cm);
\node[draw,fill=black,circle,minimum size=0.2cm, inner sep=0] at ($(K1)!0.5!(L1) + (0,-0.4)$) {};
\draw[link] ($(M)-(0,0.4)$) arc (-90:90:0.2 and 0.4);
\end{scope}
\node at (9,-3) {$\longleftrightarrow$};
\begin{scope}[xshift=10cm,yshift=-3cm]
\coordinate (K1) at (0,0);
\coordinate (L1) at (2.5,0);
\draw[link] (K1) arc (180:-180:0.5 and 0.75) node[midway, inner sep=0,outer sep=0] (K1m) {} node[black,pos=0.75,inner sep=0,below] {$\strut a$};
\node[draw,fill=white,rectangle] at (K1) {$K$};
\draw[link] (L1) arc (0:360:0.5 and 0.75) node[midway, inner sep=0,outer sep=0] (L1m) {} node[black,pos=0.75,inner sep=0, below] {$\strut b$};
\node[draw,fill=white,rectangle] at (L1) {$L$};
\end{scope}
\end{tikzpicture}
\caption{Top, a handleslide and a cancellation produces $X_{a+b}(K\#L)$.  Bottom, sliding the $K$ and $L$ handles over the $0$-framed meridian and then cancelling produces $X_a(K) \natural X_b(L)$.  The dotted circles indicate 1-handles, as in \cite[\S5.4]{gompf-stipsicz}.}
\label{fig:cinvt-additive}
\end{figure}
Let $X_n(K): S^3 \to S^3_n(K)$ denote the cobordism built by attaching an $n$-framed $2$-handle to $S^3\times [0,1]$ along $K\times\{1\}$.  Letting ``$\natural$'' denote the boundary connected sum of two compact 4-manifolds with boundary, Figure~\ref{fig:cinvt-additive} gives an embedding
\[ X_{a+b}(K\#L) \hookrightarrow X_a(K) \natural X_b(L) \]
for any integers $a,b$ and knots $K,L \subset S^3$, in which we attach a single 2-handle to the former to obtain the latter.  The usual surface \[\nu_{K\#L,a+b} \subset X_{a+b}(K\#L),\] viewed as a subdomain of $X_a(K) \natural X_b(L)$, is also homologous mod 2 to $\nu_{K,a} \sqcup \nu_{L,b}$, as can be seen by computing its intersection with cocores of the handles along $K$ and $L$ and applying \eqref{eq:nu-parity}.  Thus the map
\[ I^\#(X_a(K) \natural X_b(L), \nu_{K,a} \sqcup \nu_{L,b}) \]
factors through $F_{K\#L,a+b}$.

Since we are working over a field, we have a K{\"u}nneth isomorphism
\[ I^\#(S^3_a(K) \# S^3_b(L)) \xrightarrow{\sim} I^\#(S^3_a(K)) \otimes I^\#(S^3_b(L)) \]
which is natural with respect to split cobordisms \cite[\S 7.7]{scaduto}, in the sense that the diagram
\[ \xymatrix{
I^\#(S^3 \# S^3) \ar[rrrr]^-{I^\#(X_a(K) \natural X_b(L),\nu_{K,a}\sqcup\nu_{K,b})} \ar[d]_{\cong} &&&& I^\#(S^3_a(K) \# S^3_b(L)) \ar[d]^{\cong} \\
I^\#(S^3) \otimes I^\#(S^3) \ar[rrrr]^-{F_{K,a}\otimes F_{L,b}} &&&& I^\#(S^3_a(K)) \otimes I^\#(S^3_b(L))
} \]
commutes.  Since the top arrow factors through $F_{K\#L,a+b}$, it follows that if $F_{K\#L,a+b} = 0$ then either $F_{K,a}=0$ or $F_{L,b}=0$.  Since each of these maps is injective if and only if it is nonzero, we have proved, equivalently, that if $F_{K,a}$ and $F_{L,b}$ are both injective, then so is $F_{K\#L,a+b}$.
\end{proof}

\begin{lemma} \label{lem:cinvt-plus-w}
Let $K$ and $L$ be knots in $S^3$.  If $L$ is W-shaped, then $\cinvt(K\#L) = \cinvt(K)$.
\end{lemma}

\begin{proof}
We will use two observations repeatedly: first, if a knot is W-shaped then so is its mirror image; and second, that
\[ \cinvt(K\#L\#\mirror{L}) = \cinvt(K) \]
since $L\#\mirror{L}$ is smoothly slice and $\cinvt$ is a smooth concordance invariant.

Suppose first that $K$ is W-shaped.  Then $F_{K,0}$ and $F_{L,0}$ are both injective, hence so is $F_{K\#L,0}$ by Lemma~\ref{lem:connected-sum-surgery}, and this implies that $\cinvt(K\#L) \geq 0$ whether or not $K\#L$ is also W-shaped.  On the other hand, both $\mirror{K}$ and $\mirror{L}$ are W-shaped since $K$ and $L$ are, and so the same argument says that
\[ -\cinvt(K\#L) = \cinvt(\mirror{K\#L}) = \cinvt(\mirror{K} \# \mirror{L}) \geq 0, \]
hence we must have $\cinvt(K\#L) = 0$.  This is equal to $\cinvt(K)$ by the assumption that $K$ is W-shaped.

Now we suppose instead that $K$ is not W-shaped but that $K\#L$ is.  Then $\mirror{L}$ is also W-shaped, and so the case we have already proved gives
\[ \cinvt\big((K\#L) \# \mirror{L}\big) = \cinvt(K\#L). \]
We conclude by the second observation above that $\cinvt(K) = \cinvt(K\#L)$.

Finally, we may assume that neither $K$ nor $K\#L$ is W-shaped.  Then $F_{K,a}$ is injective for all $a < \cinvt(K)$, and $F_{L,0}$ is also injective, so $F_{K\#L,a}$ is injective for all $a < \cinvt(K)$ by Lemma~\ref{lem:connected-sum-surgery}.  This implies that $\cinvt(K\#L) \geq \cinvt(K)$.  Now $\mirror{L}$ is again W-shaped but neither $K\#L$ nor $K\#L\#\mirror{L}$ is, the latter because $K\#L\#\mirror{L}$ is concordant to the V-shaped $K$, so the same argument says that
\[ \cinvt(K) = \cinvt\big((K\#L)\#\mirror{L}\big) \geq \cinvt(K\#L) \]
and hence the left and right sides are equal.
\end{proof}

\begin{proof}[Proof of Theorem~\ref{thm:nu-quasi-hom}]
The case where at least one of $K$ and $L$ is W-shaped is Lemma~\ref{lem:cinvt-plus-w}.  Similarly, suppose that neither $K$ nor $L$ is W-shaped, but that $K\#L$ is.  Then we have
\[ \cinvt(K) = \cinvt(\mirror{L}\#(K\#L)) = \cinvt(\mirror{L}) = -\cinvt(L) \]
by Theorem~\ref{thm:conc-invt} and the fact that $K$ is concordant to $K\#(L\#\mirror{L})$, Lemma~\ref{lem:cinvt-plus-w}, and Theorem~\ref{thm:conc-invt} again.  This implies that
\[ \cinvt(K) + \cinvt(L) = 0 = \cinvt(K\#L) \]
by the fact that W-shaped knots have $\cinvt = 0$.  Thus $\cinvt$ is additive as long as at least one of the summands or the connected sum itself is W-shaped.

Finally, we suppose that none of $K$, $L$, and $K\#L$ are W-shaped.  Then $F_{K,\cinvt(K)-1}$ and $F_{L,\cinvt(L)-1}$ are both injective, hence
\[ F_{K\#L, \cinvt(K)+\cinvt(L)-2} \]
is injective as well by Lemma~\ref{lem:connected-sum-surgery}, and this says that
\begin{equation} \label{eq:quasi-bound-1}
\cinvt(K\#L) \geq \cinvt(K) + \cinvt(L) - 1.
\end{equation}
(This last assertion requires $K\#L$ to be V-shaped, so that $F_{K\#L,n}$ is injective precisely when $n < \cinvt(K\#L)$.)  Moreover, none of $\mirror{K}$, $\mirror{L}$, and $\mirror{K\#L} = \mirror{K}\#\mirror{L}$ are W-shaped either, so the same argument gives
\[ \cinvt(\mirror{K\#L}) \geq \cinvt(\mirror{K}) + \cinvt(\mirror{L}) - 1. \]
Using $\cinvt(\mirror{K}) = -\cinvt(K)$ and so on, we rearrange to get
\[ \cinvt(K\#L) \leq \cinvt(K) + \cinvt(L) + 1, \]
which together with \eqref{eq:quasi-bound-1} completes the proof.
\end{proof}

Theorem~\ref{thm:nu-quasi-hom} says that $\cinvt$ defines a \emph{quasi-morphism} from the smooth concordance group $\cC$ to $\Z$.  As with any quasi-morphism of groups, we can define its homogenization, which up to a factor of $\frac{1}{2}$ is
\begin{equation*}\label{eq:tau-def} \chominvt(K) := \frac{1}{2}\lim_{n\to\infty} \frac{\cinvt(nK)}{n}. \end{equation*}
To see that the limit exists, one first proves by induction (using Theorem~\ref{thm:nu-quasi-hom} and the triangle inequality) that
\begin{equation} \label{eq:cinvt-nK}
|\cinvt(nK) - n\cinvt(K)| \leq n-1
\end{equation}
for all integers $n \geq 1$, and then one combines the resulting inequalities
\begin{align*}
|\cinvt(mnK) - m\cinvt(nK)| &\leq m-1, & |n\cinvt(mK) - \cinvt(mnK)| &\leq n-1
\end{align*}
using the triangle inequality and divides the result by $mn$ to deduce that
\[ \left| \frac{\cinvt(mK)}{m} - \frac{\cinvt(nK)}{n} \right| \leq \frac{m+n-2}{mn}. \]
It follows that the sequence $\{\frac{1}{n}\cinvt(nK)\}_{n\geq 1}$ is Cauchy and hence convergent.

\begin{proposition} \label{prop:tau-invariant}
One-half the homogenization of $\cinvt$ defines a group homomorphism
\[ \chominvt: \cC \to \R. \]
It satisfies $|2\chominvt(K)-\cinvt(K)| \leq 1$ and $|\chominvt(K)| \leq g_s(K)$ for all knots $K \subset S^3$.
\end{proposition}

\begin{proof}
For any $n \geq 1$, Theorem~\ref{thm:nu-quasi-hom} and the fact that $\cC$ is abelian, and in particular that $n(K\#L) \cong nK \# nL$, combine to give
\[ \left| \frac{\cinvt(n(K\#L))}{2n} - \frac{\cinvt(nK)}{2n} - \frac{\cinvt(nL)}{2n} \right| \leq \frac{1}{2n}. \]
Taking limits as $n\to\infty$, it follows easily that $\chominvt(K\#L) = \chominvt(K) + \chominvt(L)$.

The claim that $|2\chominvt(K)-\cinvt(K)|\leq 1$ is proved by dividing \eqref{eq:cinvt-nK} by $n$ and letting $n\to\infty$.  This plus Theorem~\ref{thm:conc-invt} implies the slice genus bound when $K$ is not slice, since then
\begin{equation} \label{eq:nu-plus-tau}
|2\chominvt(K)| \leq |2\chominvt(K)-\cinvt(K)| + |\cinvt(K)| \leq 1 + (2g_s(K)-1) = 2g_s(K).
\end{equation}
But if $K$ is slice, then so is $nK$ for every $n \geq 1$, hence $\chominvt(K) = 0 = g_s(K)$ anyway.
\end{proof}

\begin{corollary} \label{cor:tau-maximal}
If $|\chominvt(K)| = g_s(K) > 0$, then $|\cinvt(K)| = 2g_s(K)-1$ and $\cinvt(K)$ has the same sign as $\chominvt(K)$.
\end{corollary}

\begin{proof}
By assumption equality must hold at all points in equation~\eqref{eq:nu-plus-tau}, so in particular $|\cinvt(K)|=2g_s(K)-1$ and $|2\chominvt(K)-\cinvt(K)|=1$.  Since $|2\chominvt(K)| \geq 2$, it follows that $\cinvt(K)$ must have the same sign as $2\chominvt(K)$, and hence as $\chominvt(K)$ itself.
\end{proof}

The reason for the factor of $\frac{1}{2}$ in the definition of $\chominvt$  is so that it behaves  like the $\tau$ invariant in Heegaard Floer homology. In fact, we will see in Section \ref{sec:tools} that these two concordance homomorphisms agree for homogeneous knots and quasipositive knots, as claimed in the introduction.  (See \cite[pp.~76--77]{lewark} for the definition of homogeneous knots.  A knot is quasipositive if for some $n$ it is the closure of a quasipositive $n$-braid, which is defined as a product of conjugates of the positive generators $\sigma_1,\dots,\sigma_{n-1}$ of the braid group $B_n$.)
We conjecture the following:

\begin{conjecture}
$\chominvt(K)=\tau(K)$ for every knot $K\subset S^3$.
\end{conjecture}

In Section \ref{sec:comparison}, we will show that this conjecture follows from a special case of Kronheimer--Mrowka's conjectured isomorphism relating framed instanton homology and Heegaard Floer homology, proving Proposition \ref{prop:tau-tau}.

\section{Two times $\chominvt$ is a slice-torus invariant} \label{sec:tools}
The main goal of this section is to prove Theorem \ref{thm:slice-torus}, which states that the concordance homomorphism $2\chominvt$ is a slice-torus invariant. Our proof relies on the results of the previous section, together with a $\maxtb$-bound proved using our contact invariant from \cite{bs-instanton}. 

\subsection{Slice-torus invariants} \label{ssec:slice-torus}

Lewark \cite{lewark} defined a \emph{slice-torus} knot invariant to be a smooth concordance homomorphism $\phi: \cC \to \R$ which satisfies:
\begin{itemize}
\item  $\phi(K) \leq 2g_s(K)$ for all knots $K \subset S^3$, and 
\item $\phi(T_{p,q}) = 2g_s(T_{p,q}) = (p-1)(q-1)$
for all positive torus knots $T_{p,q}$.
\end{itemize}  He proved that these properties uniquely determine $\phi(K)$ for a large class of knots $K$.  We  show that they are satisfied by $2\chominvt(K)$, per Theorem \ref{thm:slice-torus}, restated here:

\begin{reptheorem}{thm:slice-torus}
The concordance homomorphism $2\chominvt(K)$ is a slice-torus invariant.
\end{reptheorem}

We have already shown in Proposition~\ref{prop:tau-invariant} that $2\chominvt(K)$ satisfies the required slice genus bound. The sharpness for positive torus knots is stated in Proposition~\ref{prop:torus-tau-bound}. We will prove this in \S\ref{ssec:tb-bound}, thereby completing the proof of Theorem \ref{thm:slice-torus}, by establishing an inequality relating the Thurston--Bennequin number $\maxtb(K)$ to $\cinvt(K)$ for positive torus knots, and indeed for any knot with $\maxtb(K) = 2g(K)-1$ and sufficiently positive genus.

In the meantime, we deduce some of the consequences of Theorem~\ref{thm:slice-torus} described in the introduction.  The first of these is an immediate strengthening of the inequality alluded to above, replacing Thurston--Bennequin number with self-linking number:

\begin{theorem} \label{thm:sl-bound}
Let $\maxsl(K)$ denote the maximum self-linking number $\lsl(\cT)$ over all transverse knots $\cT \subset (S^3,\xi_{\mathrm{std}})$ which are smoothly isotopic to $K$.  Then
\[ \maxsl(K) \leq 2\chominvt(K)-1 \leq \cinvt(K). \]
\end{theorem}

\begin{proof}
Let $\beta \in B_n$ denote a braid on $n$ strands whose closure $\hat\beta$ is a transverse representative of $K$, and such that
\[ \lsl(\hat\beta) = \maxsl(K). \]
If $\beta$ has algebraic length $w$, then we know that $\lsl(\hat\beta) = w-n$.  We also define a graph $\Gamma^+(\hat\beta)$ whose vertices are the circles in a Seifert resolution of $\hat\beta$, and which has one edge per positive crossing of $\hat\beta$, connecting the corresponding pair of Seifert circles; let $O^+ \geq 1$ denote the number of connected components of $\Gamma^+(\hat\beta)$.

Lewark \cite[Theorem~5]{lewark} proves that since $2\chominvt$ is a slice-torus invariant, we have
\[ -1 + w - n + 2O^+ \leq 2\chominvt(K), \]
or equivalently
\[ \maxsl(K) = w-n \leq 2\chominvt(K) - (2O^+-1). \]
But $O^+ \geq 1$, so we have $\maxsl(K) \leq 2\chominvt(K) - 1$.  The remaining inequality $2\chominvt(K)-1 \leq \cinvt(K)$ is an immediate consequence of Proposition~\ref{prop:tau-invariant}.
\end{proof}

As described in the introduction, Corollary~\ref{cor:chom-qp}, which states that $\chominvt = g_s$ for quasipositive knots, and even this more general proposition both follow readily.

\begin{proposition} \label{prop:slice-bennequin-positive}
If $\maxsl(K) = 2g_s(K)-1$, then $\chominvt(K) = g_s(K)$.  In particular, if $K$ is a nontrivial positive knot, then $\chominvt(K)=g(K)$ and $\cinvt(K) = 2g(K)-1$.
\end{proposition}

\begin{proof}
Theorem~\ref{thm:sl-bound} and Proposition~\ref{prop:tau-invariant} combine to give
\[ 2g_s(K) - 1 = \maxsl(K) \leq 2\chominvt(K)-1 \leq 2g_s(K)-1 \]
and equality must hold throughout.  For positive knots $K$, Tanaka \cite{tanaka} proved that $\maxtb(K) = 2g(K)-1 = 2g_s(K)-1$; since $\maxtb(K) \leq \maxsl(K) \leq 2g(K)-1$, the above implies that $\chominvt(K)=g(K)$, and then Corollary~\ref{cor:tau-maximal} determines the value of $\cinvt(K)$.
\end{proof}

\begin{proof}[Proof of Corollary \ref{cor:chom-qp}]
Quasipositive knots satisfy $\maxsl(K) = 2g_s(K)-1$ \cite{rudolph}, so we apply Proposition~\ref{prop:slice-bennequin-positive}.  See also \cite[Proposition~5.6]{lewark}.
\end{proof}

Likewise, we may prove Corollary~\ref{cor:alternating}, which says that $\chominvt=-\sigma/2$ for alternating knots, with respect to the convention that $\sigma(T_{2,3}) = -2$:

\begin{proof}[Proof of Corollary~\ref{cor:alternating}]
 Since $2\chominvt$ is slice-torus, Lewark proves in \cite[Corollary~5.9]{lewark} that $2\chominvt(K) = -\sigma(K)$ for any alternating knot $K$.  (He states that slice-torus invariants are equal to  signature for alternating knots, but his convention is that $\sigma(T_{2,3}) = +2$.)
\end{proof}

As mentioned in the introduction, another implication of Theorem~\ref{thm:slice-torus} is the following version of Corollary~\ref{cor:homogeneous}, which also justifies Remark~\ref{rmk:tau-tau2}.

\begin{corollary} \label{cor:homogeneous2}
If $K$ is a homogeneous knot then $\chominvt(K)$ is equal to the tau invariant $\tau(K)$ in Heegaard Floer homology.  In particular, $\chominvt = \tau$ for all prime knots through 9 crossings, except possibly for $9_{42}$, $9_{44}$,  $9_{48}$.
\end{corollary}

\begin{proof}
Both $2\chominvt$ and $2\tau$ are slice-torus invariants, and \cite[Theorem~5]{lewark} implies (see \cite[p.\ 77]{lewark}) that the values of slice-torus invariants are uniquely determined for homogeneous knots.  Cromwell \cite{cromwell} found that $8_{20}$, $8_{21}$, $9_{42}$, $9_{44}$, $9_{45}$, $9_{46}$, $9_{48}$ are the only non-homogeneous prime knots through 9 crossings. But each of $8_{20}$, $8_{21}$,  $9_{45}$, $9_{46}$ (or their mirrors) is quasipositive, according to KnotInfo \cite{knotinfo}, so $2\chominvt = 2\tau = 2g_s$ for these.
\end{proof}

Finally, we deduce the  following crossing change inequality, which we will use for some computations in \S\ref{sec:ss}. It is stated as \cite[Proposition~5.3]{lewark} but is essentially due to Livingston \cite{livingston-tau}.

\begin{proposition} \label{prop:tau-change}
Suppose that we obtain the knot $K_-$ from a diagram of $K_+$ by changing a single positive crossing to a negative one.  Then
\[ 0 \leq \chominvt(K_+) - \chominvt(K_-) \leq 1. \]
\end{proposition}

\begin{proof}
Livingston \cite[Corollary~3]{livingston-tau} exhibits a pair of genus-one cobordisms from $K_+$ to $K_-$ and from $K_+ \# T_{-2,3}$ to $K_-$, so that
\begin{align*}
|\chominvt(K_+) - \chominvt(K_-)| &\leq g_s(K_+\#\mirror{K_-}) \leq 1, \\
|\chominvt(K_+) - \chominvt(T_{2,3}) - \chominvt(K_-)| &\leq g(K_+\#T_{-2,3}\#\mirror{K_-}) \leq 1
\end{align*}
by Proposition~\ref{prop:tau-invariant}.  Since $2\chominvt$ is a slice-torus invariant we have $\chominvt(T_{2,3}) = 1$, so the second inequality implies that $\chominvt(K_+) - \chominvt(K_-) \geq \chominvt(T_{2,3})-1 = 0$, completing the proof.
\end{proof}

\subsection{A weak Thurston--Bennequin bound} \label{ssec:tb-bound}

Let $\maxtb(K)$ denote the maximum Thurston--Bennequin invariant $\ltb(\Lambda)$ among all Legendrian knots $\Lambda \subset (S^3,\xi_{\mathrm{std}})$ which are smoothly isotopic to $K$.  This subsection is devoted to proving the following bound on $\cinvt$. As we shall see, this implies Proposition \ref{prop:slice-bennequin-positive}, which is then used to finish the proof of Theorem \ref{thm:slice-torus}.

\begin{proposition} \label{prop:tb-bound}
Let $K \subset S^3$ be a knot of genus $g(K) \geq 2$ such that $\maxtb(K) = 2g(K)-1$.  Then $\cinvt(K) = 2g(K)-1$.
\end{proposition}

We call this a \emph{weak} Thurston--Bennequin bound because it does not apply to most knots, and because in principle we should actually prove a more general inequality $\maxtb(K) \leq \cinvt(K)$.   However, once we use it below to prove that $2\chominvt$ is a slice-torus invariant, the inequality will immediately follow for all knots in $S^3$ with the left-hand side improved to $\maxsl(K)$, per Theorem~\ref{thm:sl-bound}.  Before proving Proposition~\ref{prop:tb-bound}, we show how it implies that $\chominvt = g_s$ for positive torus knots, which is the last step to proving that $2\chominvt$ is indeed a slice-torus invariant.

\begin{proposition} \label{prop:torus-tau-bound}
Let $K$ be a positive torus knot.  Then $\chominvt(K) = g(K)$.
\end{proposition}

\begin{proof}
A connected sum of $n \geq 2$ copies of $K$ satisfies $g(nK) \geq 2$ and $\maxtb(nK) = 2g(nK)-1$, so by Proposition~\ref{prop:tb-bound}, we have
\[ \cinvt(nK) = 2g(nK) - 1 = 2n\cdot g(K)-1 \]
for all integers $n \geq 2$.  Dividing by $2n$ gives us
\[ \frac{\cinvt(nK)}{2n} = g(K) - \frac{1}{2n}, \]
and taking limits as $n\to\infty$ gives us the desired $\chominvt(K) = g(K)$.
%
\end{proof}

We may now complete the proof of Theorem \ref{thm:slice-torus}:

\begin{proof}[Proof of Theorem \ref{thm:slice-torus}]
We have $2\chominvt(K)\leq 2g_s(K)$ by Proposition \ref{prop:tau-invariant}. The sharpness of this bound for positive torus knots follows from Proposition~\ref{prop:torus-tau-bound}.
\end{proof}

It remains to prove Proposition~\ref{prop:tb-bound}.  This proposition follows almost immediately from:

\begin{proposition} \label{prop:legendrian-injective}
If $n = \maxtb(K)-1 = 2g(K)-2$ is positive, then the cobordism map
\begin{equation} \label{eq:F_n-tb-bound}
F_n = I^\#(X_n(K),\nu_n): I^\#(S^3) \to I^\#(S^3_n(K))
\end{equation}
appearing in the exact triangle \eqref{eq:triangles} is injective.
\end{proposition}

The proof of Proposition~\ref{prop:legendrian-injective} relies on  the contact invariant we defined in \cite{bs-instanton}. Given a contact structure $\xi$ on a 3-manifold $M$ with nonempty convex boundary, and dividing set $\Gamma \subset \partial M$, this invariant consists of an element of sutured instanton homology \cite{km-excision, bs-naturality},
\[ \iinvt(\xi) \in \SHItfun(-M,-\Gamma).\footnote{We will view $\SHItfun(-M,-\Gamma)$ as a $\C$-module, though it is technically  a \emph{projectively transitive system of $\C$-modules}, and is denoted by $\SHItfunn(-M,-\Gamma)$ in  \cite{bs-instanton,bs-naturality}.} \]
This element  is zero if $\xi$ is overtwisted and nonzero if $(M,\xi)$ embeds into a Stein fillable contact manifold.  It furthermore behaves naturally with respect to Legendrian surgery \cite[Proposition~4.6]{bs-instanton}, in the sense that if $(M',\Gamma',\xi')$ is obtained from $(M,\Gamma,\xi)$ by Legendrian surgery along a Legendrian knot $\Lambda \subset M$, then there is a 2-handle cobordism map
\[ F_\Lambda: \SHItfun(-M',-\Gamma') \to \SHItfun(-M,-\Gamma) \]
sending $\iinvt(\xi')$ to $\iinvt(\xi)$.

We will not give a full definition of $\SHItfun$ here, but remark that if $Y$ is a closed, orientable 3-manifold, and $Y(1)$ is  the sutured manifold
\[ Y(1) = (Y \setminus \inr(B^3), S^1), \]
then
\begin{equation} \label{eq:defn-shi}
\SHItfun(Y(1)) \cong I_*(Y \# (R\times_h S^1) | R)_{\alpha\sqcup\eta}.
\end{equation}
Here, $R$ is a closed surface of genus at least 1; $R\times_h S^1$ is the mapping torus of an orientation-preserving diffeomorphism $h: R\to R$; the closed curves $\eta$ and $\alpha$ lie in the $R\times_h S^1$ summand, with $\eta$ nonseparating inside an $R$ fiber and $\alpha\cdot R=1$; and the $\alpha\sqcup\eta$ subscript means that we use a line bundle $w$ with $c_1(w)$ Poincar\'e dual to $\alpha+\eta$. Finally, the ``$|R\,$" in the notation indicates that the $\C$-module in \eqref{eq:defn-shi} is the simultaneous generalized $(2g(R)-2,2)$-eigenspace of the commuting operators \[\mu(R),\mu(\pt):I_*(Y \# (R\times_h S^1))_{\alpha\sqcup\eta}\to I_*(Y \# (R\times_h S^1))_{\alpha\sqcup\eta}.\] With this as background, we may now prove Proposition \ref{prop:legendrian-injective}.

\begin{proof}[Proof of Proposition~\ref{prop:legendrian-injective}]
We will first explain how to prove that the map $I^\#(X_n(K),0)$ is injective, and then modify this argument to prove the same for $F_n = I^\#(X_n(K),\nu_n)$.

For simplicity of notation, let \[Y:=S^3_n(K) \quad \textrm{and} \quad X:=X_n(K):S^3\to S^3_n(K)\] for the rest of this proof. Since $n \leq \maxtb(K)-1$, we can find a Legendrian representative $\Lambda$ of  $K$ in $(S^3,\xi_{\mathrm{std}})$, with $\ltb(\Lambda) = n+1$. Legendrian surgery on $\Lambda$ then gives a contact structure $\xi$ on $Y$.  By \cite[Proposition~4.6]{bs-instanton} we have a 2-handle cobordism map
\[ F_\Lambda: \SHItfun(-Y(1)) \to \SHItfun(-S^3(1)) \]
such that $F_\Lambda(\iinvt(\xi)) = \iinvt(\xi_{\mathrm{std}})$, which is nonzero.  The map $F_\Lambda$ comes from a smooth cobordism of the form
\[ X_{R,h}^{\bowtie} := X \bowtie ((R\times_h S^1)\times[0,1]): S^3 \# (R\times_h S^1) \to Y \# (R\times_h S^1) \]
for some $R$ and $h$. Here, we are thinking of  $X=X_n(K)$ as obtained from $S^3\times[0,1]$ by attaching an $n$-framed $2$-handle along $K\times \{1\}$, and $X_{R,h}^{\bowtie}$ as obtained by gluing the complement of an arc $\{\pt\}\times[0,1]\subset X$ to the complement of an arc $\{\pt\}\times[0,1]\subset (R\times_h S^1)\times[0,1]$, via the product framings. 
With respect to the identification \eqref{eq:defn-shi}, the map $F_\Lambda$ is then the dual of the cobordism map on instanton Floer homology,
\begin{equation}\label{eqn:i-map-nonzero} I_*(X_{R,h}^{\bowtie})_{(\alpha\sqcup\eta)\times[0,1]}: I_*(S^3 \# (R \times_h S^1)|R)_{\alpha\sqcup\eta} \to I_*(Y \#(R \times_h S^1)|R)_{\alpha\sqcup\eta} \end{equation} for some curves $\alpha$ and $\eta$ as described above, so this map must be nonzero as $F_\Lambda$ is.

Meanwhile, recall  that the  map $I^\#(X_n(K),0)$ is induced by the 2-handle cobordism 
\[ X^\# = X_{T^2,\mathrm{Id}}^{\bowtie}=X \bowtie (T^3\times [0,1]): S^3 \# T^3 \to Y \# T^3 \]  defined in \eqref{eq:Xsharp},
in which the line bundle $w$ over $T^3 \times [0,1]$ is represented by a curve $\alpha$ dual to a $T^2$ fiber, with no $\eta$ present.   By a standard  excision  argument \cite[Theorem~7.7]{km-excision}, we have an isomorphism
\[ I_*(Z\#(R\times_h S^1)|R)_{\alpha\sqcup\eta} \xrightarrow{\sim} I_*(Z\#(R\times_{\mathrm{Id}} S^1)|R)_{\alpha'\sqcup\eta'} \]
for any $Z$, where $\alpha' \subset R\times S^1$ has the form $\{\pt\}\times S^1$.  This is induced by a cobordism\[ (Z\times [0,1]) \bowtie W: Z \# (R\times_h S^1) \to Z \# (R\times S^1) \] for some fixed cobordism (independent of $Z$) \[W: R\times_h S^1 \to R\times S^1\] and some framed arc in $W$ connecting the boundary components. The diagram
\[ \xymatrix{
I_*(S^3 \#(R\times_h S^1)|R)_{\alpha\sqcup\eta} \ar[rrr]^-{I_*(X_{R,h}^{\bowtie})_{(\alpha\sqcup\eta)\times[0,1]}} \ar[d]_{\cong} &&& I_*(Y \#(R\times_h S^1)|R)_{\alpha\sqcup\eta} \ar[d]^{\cong} \\
I_*(S^3 \#(R\times S^1)|R)_{\alpha'+\eta'} \ar[rrr]^-{I_*(X_{R,\mathrm{Id}}^{\bowtie})_{(\alpha'\sqcup\eta')\times[0,1]}} &&& I_*(Y \#(R\times S^1)|R)_{\alpha'+\eta'}
} \]
then commutes, since the vertical cobordisms defined as above commute with the horizontal ones obtained by attaching $n$-framed 2-handles along $K$.  The top map is nonzero, as argued above, so we may conclude that the bottom map is as well.

We apply a similar excision argument to replace $R$ with $T^2$, and hence $X_{R,h}^{\bowtie}$ with the cobordism $X^\#=X_{T^2,\mathrm{Id}}^{\bowtie}$ above, and then another to replace $\alpha\sqcup\eta$ with $\alpha$, exactly as in \cite[\S7.4]{km-excision}, where the invariance of $\SHItfun$ is proved up to isomorphism.  The ultimate conclusion is that the map
\[ I_*(X^\#)_{\alpha\times[0,1]}: I_*(S^3\# T^3|T^2)_{\alpha} \to I_*(Y\#T^3|T^2)_{\alpha}, \]
which, by definition, is precisely the map \[I^\#(X_n(K),0): I^\#(S^3) \to I^\#(S^3_n(K)),\] is also nonzero.  Since $I^\#(S^3) \cong \C$, this map must then be injective.

We have shown that $I^\#(X_n(K),0)$ is injective, but we recall that we are actually interested in the map
\[ F_n = I^\#(X_n(K),\nu_n), \]
which a priori may be different because the product bundle on $X_n(K)$ may not be the one which appears in the surgery exact triangle \eqref{eq:triangles}.  However, we can write the maps in question as finite sums
\[ I^\#(X_n(K),0) = \sum_{s:H_2(X_n(K);\Z)\to\Z} I^\#(X_n(K),0; s) \]
and
\[ I^\#(X_n(K),\nu_n) = \sum_{s:H_2(X_n(K);\Z)\to\Z} I^\#(X_n(K),\nu_n; s) \]
using the decomposition of \cite[Theorem~1.16]{bs-lspace}.  Letting $\Sigma_n \subset X_n(K)$ be the union of a genus-minimizing Seifert surface for $K$ and a core of the $n$-framed 2-handle, we have $g(\Sigma_n) = g(K) > 0$ and
\[ \Sigma_n\cdot\Sigma_n = n = 2g(K)-2 > 0, \]
so the adjunction inequality of \cite[Theorem~1.16]{bs-lspace} says that for any closed surface $\alpha \subset X_n(K)$,
\begin{align*}
I^\#(X_n(K),\alpha;s) \neq 0 \ &\Longrightarrow\ |s(\Sigma_n)| + \Sigma_n\cdot\Sigma_n \leq 2g(\Sigma_n)-2 \\
&\Longrightarrow\ |s(\Sigma_n)| \leq 0.
\end{align*}
Thus each $I^\#(X_n(K),\alpha;s)$ is zero, except possibly when $s$ is the unique homomorphism $s_0: H_2(X_n(K);\Z) \to \Z$ with $s_0(\Sigma_n) = 0$.  In other words, we have shown that
\[ I^\#(X_n(K),\alpha) = I^\#(X_n(K),\alpha; s_0), \]
where $s_0$ is identically zero since $\Sigma_n$ generates $H_2(X_n(K);\Z)$.  But \cite[Theorem~1.16]{bs-lspace} also comes with a sign relation which implies that
\begin{align*}
I^\#(X_n(K),\nu_n;s_0) &= (-1)^{\frac{1}{2}(s_0(\nu_n)+\nu_n\cdot\nu_n)+0\cdot\nu_n} I^\#(X_n(K),0;s_0) \\
&= (-1)^{\frac{1}{2}(\nu_n\cdot\nu_n)} I^\#(X_n(K),0;s_0),
\end{align*}
so we conclude that $F_n = I^\#(X_n(K),\nu_n)$ is equal to $\pm I^\#(X_n(K),0)$.  Thus the injectivity of $F_n$ follows from that of $I^\#(X_n(K),0)$.
\end{proof}

\begin{proof}[Proof of Proposition~\ref{prop:tb-bound}]
Proposition~\ref{prop:legendrian-injective} says that $F_n$ is injective for $n = 2g(K)-2$, which is positive.  If $K$ were W-shaped then $F_m$ would not be injective for any $m>0$, so $K$ must be V-shaped.  But then $F_m$ is injective precisely for all $m < \cinvt(K)$, so now we have
\[ 2g(K) - 2 < \cinvt(K) \leq 2g(K)-1, \]
where the right side comes from Theorem~\ref{thm:conc-invt}, and it follows that $\cinvt(K) = 2g(K)-1$.
\end{proof}

\section{Computations of framed instanton homology} \label{sec:computations-dim}

In this section, we compute the framed instanton homology of surgeries on infinite families of twist and pretzel knots, and then on most of the prime knots through 8 crossings, proving Theorems \ref{thm:main-twist}, \ref{thm:main-pretzels}, and \ref{thm:8-crossings}. A strategy  we employ repeatedly is to find homeomorphisms between  surgeries on different knots and then apply Theorem~\ref{thm:rational-surgeries} to determine the values of $r_0$ and $\cinvt$ of one of the knots in terms of those of the other. Indeed, this strategy, together with the properties of $\chominvt$ and $\cinvt$ we proved in \S\ref{sec:quasi-morphism} and \S\ref{sec:tools}, suffices for almost all of the knots listed above.

\subsection{Torus knots}
We record the following as it will be helpful for later calculations:
\begin{lemma} \label{lem:t2q-surgery}
If $K$ is the $(p,q)$-torus knot for some integers $p,q > 0$, then \[\nu^\#(K)=r_0(K) = pq-p-q.\]
\end{lemma}

\begin{proof}
This follows immediately from Proposition \ref{prop:nu-l-space}, since $(pq-1)$-surgery on $K=T_{p,q}$ is a lens space \cite{moser} and hence (by Corollary~\ref{cor:lens-spaces}) an instanton L-space surgery of positive slope, and since $2g(K)-1 = pq-p-q$.
\end{proof}

\subsection{Twist knots} We will prove Theorem \ref{thm:main-twist}  from a combination of Propositions \ref{prop:twist-2n+2-surgery} and \ref{prop:twist-2k-1-surgery}, which treat the twist knots with an even and odd number of half-twists, respectively.
In order to prove these propositions, we first find homeomorphisms between surgeries on different two-bridge knots, as follows:

\begin{proposition} \label{prop:two-bridge-surgeries}
For  integers $a$ and $b$ which are not both odd, let $K(a,b)$ denote the two-bridge knot with two twist regions having $a$ and $b$ crossings (counted with signs), respectively, as in Figure~\ref{fig:two-bridge-surgeries}.  Then we have
\begin{align}
S^3_{4n-1}(K(2m+1,2n)) &\cong S^3_{(4n-1)/n}(K(2,2m+1)) \label{eq:surgery-2-bridge-odd} \\
S^3_{4n+1}(K(2m+1,2n)) &\cong S^3_{-(4n+1)/n}(K(-2,2m+1)) \label{eq:surgery-2-bridge-odd-2} \\
S^3_{\pm 1}(K(2m,2n)) &\cong S^3_{-1/n}(K(\mp 2,2m)) \label{eq:surgery-2-bridge-even}
\end{align}
for all $m$ and $n$.
\end{proposition}

\begin{proof}
The proof for $(4n-1)$-surgery on $K=K(2m+1,2n)$ is given in Figure~\ref{fig:two-bridge-surgeries}.  After performing a Rolfsen twist (see \cite[p.~162]{gompf-stipsicz}) along a curve $c$ which links $K$ twice, $K$ becomes a $-1$-framed unknot and we blow it down to turn $c$ into a $K(2,2m+1)$ with surgery coefficient $-\frac{1}{n}+4 = \frac{4n-1}{n}$. If instead we do $(4n+1)$-surgery, then the Rolfsen twist changes the framing on $K$ to $+1$, and blowing it down produces a $K(-2,2m+1)$ with coefficient $-\frac{1}{n}-4$.

The proof for $\pm 1$-surgery on $K=K(2m,2n)$ is identical except for the coefficients.  In this case $c$ and $K$ have linking number zero, so the Rolfsen twist and blow-down preserve the coefficients of $K$ and of $c$, respectively, and we end up with $-\frac{1}{n}$-surgery.  The resulting knot is $K(\mp 2,2m)$ because we have blown down a $\pm 1$-framed unknot.
\begin{figure}
\begin{tikzpicture}
\begin{scope}[xshift=-0.6cm] 
\draw[link] (0,0) to[out=180,in=180] ++(0,0.4) ++(0,0.4) to[out=180,in=180] ++(0,0.4) to[out=0,in=180] ++(2.6,0);
\draw[link,looseness=1] (0,0.4) to[out=0,in=180] ++(0.4,0.4) ++(0,-0.4) to[out=0,in=180] ++(0.4,0.4) ++(0,-0.4) to[out=0,in=180] ++(0.4,0.4) ++(0.4,-0.4) to[out=0,in=180] ++(0.4,0.4) ++(0,-0.4) to[out=0,in=180] ++(0.4,0.4);
\node at (1.4,0.6) {\tiny $\cdots$};
\draw[link,looseness=1] (0,0.8) to[out=0,in=180] ++(0.4,-0.4) ++(0,0.4) to[out=0,in=180] ++(0.4,-0.4) ++(0,0.4) to[out=0,in=180] ++(0.4,-0.4) ++(0.4,0.4) to[out=0,in=180] ++(0.4,-0.4) ++(0,0.4) to[out=0,in=180] ++(0.4,-0.4) ++(0,0.4) to[out=0,in=180] ++(0.2,0);
\draw[link,looseness=1] (2.6,0.8) to[out=0,in=180] ++(0.4,0.4) ++(0,-0.4) to[out=0,in=180] ++(0.4,0.4) ++(0.4,-0.4) to[out=0,in=180] ++(0.4,0.4) ++(0,-0.4) to[out=0,in=180] ++(0.4,0.4) ++(0,-0.4);
\draw[link,looseness=1] (2.6,1.2) to[out=0,in=180] ++(0.4,-0.4) ++(0,0.4) to[out=0,in=180] ++(0.4,-0.4) ++(0.4,0.4) to[out=0,in=180] ++(0.4,-0.4) ++(0,0.4) to[out=0,in=180] ++(0.4,-0.4) ++(0,0.4);
\node at (3.6,1) {\tiny $\cdots$};
\draw[semithick] (5,1.4) arc (90:270:0.2 and 0.4);
\draw[link] (4.6,0.8) -- ++(0.6,0) to[out=0,in=0] ++(0,-0.4) to[out=180,in=0] (2.4,0.4);
\draw[link] (4.6,1.2) -- ++(0.6,0) to[out=0,in=0] ++(0,-1.2) to[out=180,in=0] (0,0);
\draw[line width=2.4pt,white] (5,1.4) arc (90:-90:0.2 and 0.4);
\draw[semithick] (5,1.4) node[above right,inner sep=2pt] {\tiny$c$} arc (90:-90:0.2 and 0.4);
\node at (1.2,1.5) {$4n-1$};
\draw[|-|] (0,-0.3) to node[midway,below] {\tiny $2m+1$} ++(2.4,0);
\draw[|-|] (2.6,1.5) to node[midway,above] {\tiny $2n$} ++(2,0);
\end{scope}

\node at (5.65,0.6) {\Large $\rightarrow$};

\begin{scope}[xshift=6.25cm]
\draw[link] (0,0) to[out=180,in=180] ++(0,0.4) ++(0,0.4) to[out=180,in=180] ++(0,0.4) to[out=0,in=180] ++(1.2,0);
\draw[link,looseness=1] (0,0.4) to[out=0,in=180] ++(0.4,0.4);
\draw[link,looseness=1] (0,0.8) to[out=0,in=180] ++(0.4,-0.4);
\draw[link] (0,0) to[out=180,in=180] (0,0.4);
\draw[link] (0,1.2) to[out=180,in=180] (0,0.8);
\node at (0.6,0.6) {\tiny $\cdots$};
\draw[link,looseness=1] (0.8,0.4) to[out=0,in=180] ++(0.4,0.4);
\draw[link,looseness=1] (0.8,0.8) to[out=0,in=180] ++(0.4,-0.4);
\draw[link] (1.6,0.6) arc (270:90:0.2 and 0.4);
\draw[link] (1.2,0.4) to[out=0,in=180] ++(0.8,0) (1.2,0.8) to[out=0,in=180] ++(0.8,0);
\draw[link] (1.2,1.2) to[out=0,in=180] ++(0.8,0);
\draw[link] (1.6,0.6) arc (-90:90:0.2 and 0.4);
\draw[link] (2,0.8) to[out=0,in=0] ++(0,-0.4);
\draw[link] (2,1.2) to[out=0,in=0] ++(0,-1.2) to[out=180,in=0] (0,0);
\node at (0.6,1.5) {$-1$};
\draw[|-|] (0,-0.3) to node[midway,below] {\tiny $2m+1$} ++(1.2,0);
\node at (1.95,1.5) {\tiny $-1/n$};
\end{scope}

\node at (9.3,0.6) {\Large $\cong$};

\begin{scope}[xshift=10cm]
\draw[link,looseness=1.5] (0,0.6) to[out=270,in=270] (1.6,0.6);
\draw[link] (0.2,0.4) to[out=180,in=0] (-0.1,0.4) to[out=180,in=180] (-0.1,0.8) to[out=0,in=180] (0.2,0.8);
\draw[link,looseness=1] (0.2,0.4) to[out=0,in=180] ++(0.4,0.4) ++(0.4,-0.4) to[out=0,in=180] ++(0.4,0.4);
\draw[link,looseness=1] (0.2,0.8) to[out=0,in=180] ++(0.4,-0.4) ++(0.4,0.4) to[out=0,in=180] ++(0.4,-0.4);
\node at (0.8,0.6) {\tiny $\cdots$};
\draw[link] (1.4,0.4) to[out=0,in=180] ++(0.3,0) to[out=0,in=0] ++(0,0.4) to[out=180,in=0] (1.4,0.8);
\draw[link,looseness=1.5] (0,0.6) to[out=90,in=90] (1.6,0.6);
\draw[|-|] (0.2,-0.3) to node[midway,below] {\tiny $2m+1$} ++(1.2,0);
\node at (2.05,0.25) {\tiny $-1/n$};
\node at (0.8,1.6) {$-1$};
\end{scope}

\node at (3.15,-2.4) {\Large $\rightarrow$};

\begin{scope}[xshift=2cm,yshift=-3cm] 
\draw[link,looseness=1] (2.2,0.8) to[out=0,in=180] ++(0.4,0.4) ++(0,-0.4) to[out=0,in=180] ++(0.4,0.4) ++(0,-0.4) to[out=0,in=180] ++(0.4,0.4) ++(0.4,-0.4) to[out=0,in=180] ++(0.4,0.4) ++(0,-0.4) to[out=0,in=180] ++(0.4,0.4) ++(0,-0.4);
\draw[link,looseness=1] (2.2,1.2) to[out=0,in=180] ++(0.4,-0.4) ++(0,0.4) to[out=0,in=180] ++(0.4,-0.4) ++(0,0.4) to[out=0,in=180] ++(0.4,-0.4) ++(0.4,0.4) to[out=0,in=180] ++(0.4,-0.4) ++(0,0.4) to[out=0,in=180] ++(0.4,-0.4) ++(0,0.4);
\node at (3.6,1) {\tiny $\cdots$};
\draw[link] (2.2,0.8) to[out=180,in=180] ++(0,-0.4) to[out=0,in=180] ++(1,0);
\draw[link] (2.2,1.2) to[out=180,in=180] ++(0,-1.2) to[out=0,in=180] ++(1,0);
\draw[link] (4.6,0.8) to[out=0,in=0] ++(0,-0.4) to[out=180,in=0] ++(-1,0);
\draw[link] (4.6,1.2) to[out=0,in=0] ++(0,-1.2) to[out=180,in=0] ++(-1,0);
\draw[link,looseness=2.5] (3.6,0.4) to[out=180,in=180] ++(0,-0.4);
\draw[link,looseness=2.5] (3.2,0.4) to[out=0,in=0] ++(0,-0.4);
\begin{scope} 
\clip (3.1,0.2) rectangle (3.7,0.6);
\draw[link,looseness=2.5] (3.6,0.4) to[out=180,in=180] ++(0,-0.4);
\end{scope}
\node at (3.4,-0.3) {$(4n-1)/n$};
\draw[|-|] (2.2,1.5) to node[midway,above] {\tiny $2m+1$} ++(2.4,0);
\end{scope}

\node at (7.6,-2.4) {\Large $\cong$};

\begin{scope}[xshift=7cm,yshift=-3cm] 
\draw[link] (1.2,0) to[out=180,in=180] ++(0,0.4) ++(0,0.4) to[out=180,in=180] ++(0,0.4) to[out=0,in=180] ++(1,0);
\draw[link,looseness=1] (1.2,0.4) to[out=0,in=180] ++(0.4,0.4) ++(0,-0.4) to[out=0,in=180] ++(0.4,0.4) to[out=0,in=180] ++(0.2,0);
\draw[link,looseness=1] (1.2,0.8) to[out=0,in=180] ++(0.4,-0.4) ++(0,0.4) to[out=0,in=180] ++(0.4,-0.4) to[out=0,in=180] ++(0.2,0);
\draw[link,looseness=1] (2.2,0.8) to[out=0,in=180] ++(0.4,0.4) ++(0,-0.4) to[out=0,in=180] ++(0.4,0.4) ++(0,-0.4) to[out=0,in=180] ++(0.4,0.4) ++(0.4,-0.4) to[out=0,in=180] ++(0.4,0.4) ++(0,-0.4) to[out=0,in=180] ++(0.4,0.4) ++(0,-0.4);
\draw[link,looseness=1] (2.2,1.2) to[out=0,in=180] ++(0.4,-0.4) ++(0,0.4) to[out=0,in=180] ++(0.4,-0.4) ++(0,0.4) to[out=0,in=180] ++(0.4,-0.4) ++(0.4,0.4) to[out=0,in=180] ++(0.4,-0.4) ++(0,0.4) to[out=0,in=180] ++(0.4,-0.4) ++(0,0.4);
\node at (3.6,1) {\tiny $\cdots$};
\draw[link] (4.6,0.8) to[out=0,in=0] ++(0,-0.4) to[out=180,in=0] (2.2,0.4);
\draw[link] (4.6,1.2) to[out=0,in=0] ++(0,-1.2) to[out=180,in=0] (1.2,0);
\node at (3,-0.3) {$(4n-1)/n$};
\draw[|-|] (2.2,1.5) to node[midway,above] {\tiny $2m+1$} ++(2.4,0);
\end{scope}

\end{tikzpicture}
\caption{Comparing $(4n-1)$-surgery on the two-bridge knot $K(2m+1,2n)$ to $\frac{4n-1}{n}$-surgery on $K(2,2m+1)$.}
\label{fig:two-bridge-surgeries}
\end{figure}
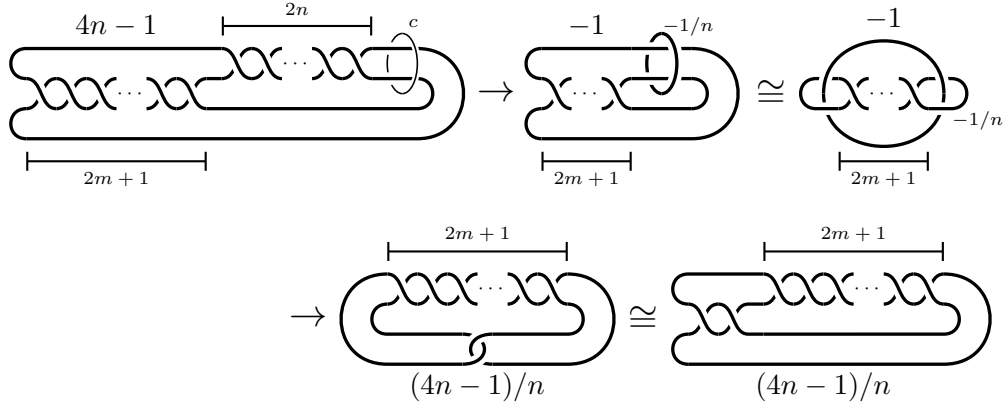
\end{proof}

\begin{lemma} \label{lem:figure-eight}
The figure eight knot $4_1$ satisfies $\cinvt(4_1) = 0$ and $r_0(4_1) = 2$.
\end{lemma}

\begin{proof}
We have $\cinvt(4_1)=0$ because $4_1$ is amphichiral.  In the notation of Proposition~\ref{prop:two-bridge-surgeries}, we have
\begin{align*}
4_1 &= K(-2,2), &
T_{2,3} = \mirror{3_1} = K(-2,-2),
\end{align*}
and thus the case $(m,n)=(-1,1)$ of \eqref{eq:surgery-2-bridge-even} says that
\[ S^3_1(4_1) = S^3_1(K(-2,2)) \cong S^3_{-1}(K(-2,-2)) = S^3_{-1}(T_{2,3}). \]
We have $\cinvt(T_{2,3}) = r_0(T_{2,3}) = 1$ by Lemma~\ref{lem:t2q-surgery}, so 
\[ \dim I^\#(S^3_1(4_1)) = \dim I^\#(S^3_{-1}(T_{2,3})) = 3 \]
by Theorem~\ref{thm:rational-surgeries}, and now this together with $\cinvt(4_1)=0$ implies that $r_0(4_1)=2$.
\end{proof}

We now compute $r_0$ and $\cinvt$ for the twist knots in Theorem \ref{thm:main-twist}. Recall that $K_n$ denotes the twist knot with a positive clasp and $n\geq 1$ positive half-twists, as shown in Figure~\ref{fig:main-twist}.

\begin{proposition} \label{prop:twist-2n+2-surgery}
For $K=K_{2n}$ with $n\geq 1$, we have  \[\cinvt(K) = 0\quad\textrm{and}\quad r_0(K)=2n,\] which implies that 
\[\dim I^\#(S^3_{p/q}(K)) = 2nq+|p|\] for all coprime integers $p,q$ with $p\neq 0$ and $q\geq 1$, by Theorem \ref{thm:rational-surgeries}.
\end{proposition}

\begin{proof}
When $n=1$, $K_{2n} = K_2$ is the figure eight knot $4_1$, and we have already computed $\cinvt(4_1)$ and $r_0(4_1)$ in Lemma~\ref{lem:figure-eight}.  We may now assume that $n \geq 2$.

The knot $K$ is the two-bridge knot $K(-2,2n)$ in the notation of Proposition~\ref{prop:two-bridge-surgeries}, so two applications of \eqref{eq:surgery-2-bridge-even} give
\begin{align*}
S^3_1(K) &\cong S^3_{-1/n}(K(-2,-2)), &
S^3_{-1}(K) &\cong S^3_{-1/n}(K(2,-2)).
\end{align*}
But $K(-2,-2)$ is the right-handed trefoil $T_{2,3}=\mirror{3_1}$, which has $\cinvt(\mirror{3_1}) = r_0(\mirror{3_1}) = 1$, by Lemma~\ref{lem:t2q-surgery}. Therefore, $I^\#(S^3_1(K))$ has dimension
\[ \dim I^\#(S^3_{-1/n}(\mirror{3_1})) = n\cdot r_0(\mirror{3_1}) + |(-1) - n\cinvt(\mirror{3_1})| = 2n+1 \]
by Theorem~\ref{thm:rational-surgeries}. 
Similarly, $K(2,-2)$ is the figure eight $4_1$, so we have
\[ \dim I^\#(S^3_{-1}(K)) = \dim I^\#(S^3_{-1/n}(4_1)) = 2n+1, \]
as determined above.  Since $\dim I^\#(S^3_1(K)) = \dim I^\#(S^3_{-1}(K))$, we must have $\cinvt(K)=0$, and then the fact that these dimensions are both $2n+1$ says that $r_0(K) = 2n$, as desired.
\end{proof}

\begin{proposition} \label{prop:twist-2k-1-surgery}
For $K = K_{2n-1}$ with $n\geq 1$, we have \[\cinvt(K)=-1\quad \textrm{and}\quad r_0(K) = 2n-1,\] which implies that \[ \dim I^\#(S^3_{p/q}(K)) = (2n-1)q + |p+q| \]
for all coprime integers $p,q$ with $q \geq 1$, by Theorem \ref{thm:rational-surgeries}.
\end{proposition}

\begin{proof}
By inspection, $K$ has Seifert genus 1, and its mirror is positive, so $\cinvt(K)=-1$ by Proposition~\ref{prop:slice-bennequin-positive}.  Thus, $r_0(K)=\dim I^\#(S^3_{-1}(K)).$ To compute this dimension, note that we can identify $K$ with the two-bridge knot $K(2,2n)$ in the notation of Proposition~\ref{prop:two-bridge-surgeries}. Since $K(2,2)$ is the left-handed trefoil $3_1$, we have
\[ S^3_{-1}(K) = S^3_{-1}(K(2,2n)) \cong S^3_{-1/n}(K(2,2)) = S^3_{-1/n}(3_1) \]
by \eqref{eq:surgery-2-bridge-even}.  But $\cinvt(3_1)=-1$ and $r_0(3_1)=1$, so
\[ r_0(K)=\dim I^\#(S^3_{-1/n}(3_1)) = n\cdot r_0(3_1) + |(-1) - n\cinvt(3_1)| = 2n-1, \]
as claimed.
\end{proof}

\begin{proof}[Proof of Theorem \ref{thm:main-twist}]
This follows immediately from Propositions \ref{prop:twist-2n+2-surgery} and \ref{prop:twist-2k-1-surgery}.
\end{proof}

\subsection{Pretzel knots} \label{ssec:pretzels}

We prove Theorem \ref{thm:main-pretzels} below, which computes the framed instanton homology of nonzero surgeries on the pretzel knots $P(n,3,-3)$. Note that each these pretzels is smoothly slice, since we can build it by attaching an oriented band to the two-component unlink $P(3,-3)$. This  implies that $\cinvt=0$ for each of these knots, so we just need to show that $r_0=4$ for all such knots, as claimed in the theorem. To show that, we use the following proposition, which was discovered experimentally in SnapPy:

\begin{proposition} \label{prop:pretzel-homeo}
There is a homeomorphism
\[ S^3_{-2}(P(n,3,-3)) \cong S^3_2(P(n+3,3,-3)) \]
for all integers $n$.
\end{proposition}

We will prove Proposition~\ref{prop:pretzel-homeo} at the end of this subsection. In the meantime, let us use it to complete the proof of   Theorem~\ref{thm:main-pretzels}, restated here:

\begin{reptheorem}{thm:main-pretzels}
For all integers $n$, we have \[\cinvt(P(n,3,-3)) = 0 \quad\textrm{and}\quad r_0(P(n,3,-3)) = 4, \] 
which implies that \[ \dim I^\#(S^3_{p/q}(P(n,3,-3))) = 4q + |p|\] for all coprime integers $p,q$ with $p\neq 0$ and $q\geq 1$, by Theorem \ref{thm:rational-surgeries}.
\end{reptheorem}

\begin{proof}
Since $\cinvt(P(n,3,-3)) = 0$ as noted above (these pretzels are slice), we have that
\[\dim I^\#(S^3_{k}(P(n,3,-3))) = r_0(P(n,3,-3)) + |k|\]
for all nonzero integers $k$, by Theorem \ref{thm:rational-surgeries}. In particular, we record the fact that
\begin{align}
\label{eq:equivs}
r_0(P(n,3,-3))=4&\iff\dim I^\#(S^3_{\pm 1}(P(n,3,-3))) = 5\\
& \iff \dim I^\#(S^3_{\pm 2}(P(n,3,-3))) = 6.
\end{align}

Let us start by showing that \begin{equation*}\label{eq:P1}r_0(P(-1,3,-3)) = r_0(P(1,3,-3)) = 4.\end{equation*} Since these two knots are mirror to one another, it suffices to show that the second equality holds. For this, note that $P(1,3,-3)$ is the twist knot $K_{4}=6_1$. Therefore, Proposition \ref{prop:twist-2n+2-surgery} tells us that $r_0(P(1,3,-3)) =4,$ as desired. As above, this implies that \[\dim I^\#(S^3_{\pm 2}(P(-1,3,-3))) =6=\dim I^\#(S^3_{\pm 2}(P(1,3,-3))).\] Together with Proposition \ref{prop:pretzel-homeo}, the first equality  implies that 
\[\dim I^\#(S^3_{\pm 2}(P(n,3,-3))) =6\iff r_0(P(n,3,-3))=4\] for all $n \equiv -1 \pmod{3}$, while the second equality implies that \[\dim I^\#(S^3_{\pm 2}(P(n,3,-3))) =6\iff r_0(P(n,3,-3))=4\] for all $n \equiv 1 \pmod{3}$. 
It  remains to address the case $n \equiv 0 \pmod{3}$. 

Meier \cite{meier} proved that 1-surgery on each $P(n,3,-3)$ for $2 \leq n \leq 6$ is a small Seifert-fibered manifold, and in particular 
\[ S^3_1(P(3,3,-3)) \cong -S^3_1(P(4,3,-3)) \cong S^2(1/2,-1/5,-2/7), \]
where the notation on the right describes the Seifert invariants of the manifold in question. We therefore have that \[\dim I^\#(S^3_{1}(P(3,3,-3)))=\dim I^\#(S^3_{1}(P(4,3,-3))) = 5,\] where the second equality follows from our calculation that $r_0(P(4,3,-3))=4$ above. By \eqref{eq:equivs}, this  implies that \[\dim I^\#(S^3_{\pm 2}(P(3,3,-3))) =6\iff r_0(P(3,3,-3))=4,\] which, together with Proposition \ref{prop:pretzel-homeo},  shows that \[\dim I^\#(S^3_{\pm 2}(P(n,3,-3))) =6\iff r_0(P(n,3,-3))=4\] for all $n \equiv 0 \pmod{3}$.
\end{proof}

We now turn to the proof of Proposition~\ref{prop:pretzel-homeo}, which was provided to us by Ken Baker.  In preparation, we recall how to realize Dehn surgery on a strongly invertible knot $K$ as the branched double cover of a link in $S^3$, following Montesinos \cite{montesinos}.  By definition, there is an involution $\iota$ of $S^3$ with $\iota(K)=K$, having fixed set an unknot $U$ which intersects $K$ in two points.  In the quotient $S^3/\iota \cong S^3$, we remove a neighborhood of the arc $K/\iota$, and what remains of the unknot $U/\iota$ is a tangle with four endpoints, whose branched double cover is the complement of $K$.  We can thus fill this tangle in with a rational tangle to get a link $L \subset S^3$ whose branched double cover is any fixed Dehn surgery on $K$.

Similarly, suppose that $K$ is a periodic knot of period $2$; now there is an involution $\iota: S^3 \to S^3$ with $\iota(K)=K$ whose fixed set is an unknot $U$ disjoint from $K$.  If we do Dehn surgery of slope $p/q$ on $K$, then $\iota$ extends across the surgery torus, where it acts freely if $p$ is odd and fixes the core if $p$ is even.  Taking quotients by $\iota$, we see that $S^3_{p/q}(K)$ is the branched double cover of a link $L \subset S^3_{p/2q}(K/\iota)$, where $L$ consists of the image $U/\iota$ if $p$ is odd and $U/\iota \cup K/\iota$ if $p$ is even.  We note that $L \subset S^3$ if $K/\iota$ is unknotted and $p$ is $\pm1$ or $\pm2$, though this need not hold in general.

\begin{proof}[Proof of Proposition~\ref{prop:pretzel-homeo}]
We wish to show for fixed $n \in \Z$ that
\[ S^3_{-2}(P(n,3,-3)) \cong S^3_2(P(n+3,3,-3)). \]
In fact, it will suffice to prove this for $n$ even: assuming we have done so, then since $P(-m,3,-3)$ is the mirror of $P(m,3,-3)$ for all $m$, we have
\begin{align*}
S^3_{-2}(P(2k+1,3,-3)) &\cong -S^3_2(P(-2k-1,3,-3)) \\
&\cong -S^3_{-2}(P(-2k-4,3,-3)) \\
&\cong S^3_2(P(2k+4,3,-3)),
\end{align*}
where the second homeomorphism comes from the case $n=-2k-4$ of the proposition, establishing the result for odd $n=2k+1$ as well.  We therefore wish to prove for arbitrary $k\in\Z$ that
\[ S^3_{-2}(P(2k,3,-3)) \cong S^3_2(2k+3,3,-3), \]
and we will do so by exhibiting both of these as branched double covers of the same 2-component link.

\begin{figure}
\tikzset{twistregion/.style={draw, fill=white, thick, minimum width=0.6cm}}
\newcommand{\AxisRotator}{\tikz [x=0.10cm,y=0.25cm,line width=.2ex,-stealth] \draw[thin] (0,0) arc (-165:165:1 and 1);}
\begin{tikzpicture}

\begin{scope}[yshift=5cm] 
\begin{scope}[xshift=-4cm]
\draw[dotted] (0,0) circle (1);
\clip (0,0) circle (1);
\draw[link] (-0.2,-1) -- ++(0,2) ++(0.4,0) -- ++(0,-2);
\node[twistregion] at (0,0.5) {\tiny$k$};
\node[twistregion] at (0,-0.5) {\tiny$k$};
\draw[semithick] (-1,0) -- ++(2,0);
\node at (-0.6,0) {\AxisRotator};
\end{scope}
\node at (-2,-0.025) {\small$\xrightarrow{\mathrm{quotient}}$};
\begin{scope}[xshift=0cm]
\draw[dotted] (0,0) circle (1);
\clip (0,0) circle (1);
\draw[link] (-0.2,1) -- ++(0,-1) ++(0.4,0) -- ++(0,1);
\node[twistregion] at (0,0.5) {\tiny$k$};
\draw[semithick] (-0.2,-1) -- ++(0,0.25) to[out=90,in=180,looseness=1.25] ++(0,0.75) -- ++(0.4,0) to[out=0,in=90,looseness=1.25] ++(0,-0.75) -- ++(0,-0.25);
\end{scope}
\node at (2,-0.025) {\small$\xrightarrow{\mathrm{isotopy}}$};
\begin{scope}[xshift=4cm]
\draw[dotted] (0,0) circle (1);
\clip (0,0) circle (1);
\draw[link] (-0.2,1) -- ++(0,-0.5) ++(0.4,0) -- ++(0,0.5);
\draw[semithick] (-0.2,-1) -- ++(0,0.75) to[out=90,in=180,looseness=1.25] ++(0,0.75) -- ++(0.4,0) to[out=0,in=90,looseness=1.25] ++(0,-0.75) -- ++(0,-0.75);
\node[twistregion,semithick] at (0,-0.5) {\tiny$k$};
\end{scope}
\end{scope}

\begin{scope}[yshift=2.5cm] 
\begin{scope}[xshift=-4cm]
\draw[dotted] (0,0) circle (1);
\clip (0,0) circle (1);
\draw[semithick] (-1,0) -- ++(2,0);
\node at (-0.6,0) {\AxisRotator};
\draw[link,looseness=1] (-0.2,1) -- ++(0,-0.55) to[out=270,in=90] ++(0.4,-0.3) ++(-0.4,0) to[out=270,in=90] ++(0.4,-0.3) ++(-0.4,0) to[out=270,in=90] ++(0.4,-0.3) -- ++(0,-0.55);
\draw[link,looseness=1] (0.2,1) -- ++(0,-0.55) to[out=270,in=90] ++(-0.4,-0.3) ++(0.4,0) to[out=270,in=90] ++(-0.4,-0.3) ++(0.4,0) to[out=270,in=90] ++(-0.4,-0.3) -- ++(0,-0.55);
\node[twistregion] at (0,0.65) {\tiny$k$};
\node[twistregion] at (0,-0.65) {\tiny$k$};
\end{scope}
\node at (-2,-0.025) {\small$\xrightarrow{\mathrm{quotient}}$};
\begin{scope}[xshift=0cm]
\draw[dotted] (0,0) circle (1);
\clip (0,0) circle (1);
\draw[link,looseness=1] (-0.2,1) -- ++(0,-0.55) to[out=270,in=90] ++(0.4,-0.4) to[out=270,in=270,looseness=2] ++(-0.4,0);
\draw[link,looseness=1] (-0.2,0.05) to[out=90,in=270] ++(0.4,0.4) -- ++(0,0.55);
\node[twistregion] at (0,0.65) {\tiny$k$};
\draw[line width=3pt,white] (-0.2,-1) -- ++(0,0.75) to[out=90,in=90,looseness=2] ++(0.4,0) -- ++(0,-0.75);
\draw[semithick] (-0.2,-1) -- ++(0,0.75) to[out=90,in=90,looseness=2] ++(0.4,0) -- ++(0,-0.75);
\begin{scope} 
\clip (0,-0.5) rectangle (0.4,0.5);
\draw[link] (0.2,0.05) to[out=270,in=270,looseness=2] ++(-0.4,0);
\end{scope}
\end{scope}
\node at (2,-0.025) {\small$\xrightarrow{\mathrm{isotopy}}$};
\begin{scope}[xshift=4cm]
\draw[dotted] (0,0) circle (1);
\clip (0,0) circle (1);
\draw[link] (-0.2,1) -- ++(0,-0.5) to[out=270,in=270,looseness=2] ++(0.4,0) -- ++(0,0.5);
\draw[semithick] (0.2,-1) -- ++(0,0.15) to[out=90,in=270,looseness=1] ++(-0.4,0.4) -- ++(0,0.65);
\draw[line width=3pt,white] (-0.2,0.2) to[out=90,in=90,looseness=2] ++(0.4,0) -- ++(0,-0.65) to[out=270,in=90,looseness=1] ++(-0.4,-0.4) -- ++(0,-0.15);
\draw[semithick] (-0.2,0.2) to[out=90,in=90,looseness=2] ++(0.4,0) -- ++(0,-0.65) to[out=270,in=90,looseness=1] ++(-0.4,-0.4) -- ++(0,-0.15);
\node[twistregion,semithick] at (0,-0.2) {\tiny$k$};
\begin{scope} 
\clip (0,0) rectangle (0.4,0.5);
\draw[link] (-0.2,0.5) to[out=270,in=270,looseness=2] ++(0.4,0);
\end{scope}
\end{scope}
\end{scope}

\begin{scope} 
\begin{scope}[xshift=-4cm]
\draw[dotted] (0,0) circle (1);
\clip (0,0) circle (1);
\draw[semithick] (-1,0) -- ++(2,0);
\node at (-0.78,0) {\AxisRotator};
\draw[link,looseness=1] (-0.6,1) -- ++(0,-0.25) to[out=270,in=90] ++(0.4,-0.5) ++(-0.4,0) to[out=270,in=90] ++(0.4,-0.5) ++(-0.4,0) to[out=270,in=90] ++(0.4,-0.5) to[out=270,in=270,looseness=1.5] ++(0.4,0) to[out=90,in=270] ++(0.4,0.5) ++(-0.4,0) to[out=90,in=270] ++(0.4,0.5) ++(-0.4,0) to[out=90,in=270] ++(0.4,0.5) -- ++(0,0.25);
\draw[link,looseness=1] (0.6,-1) -- ++ (0,0.25) to[out=90,in=270] ++(-0.4,0.5) ++(0.4,0) to[out=90,in=270] ++(-0.4,0.5) ++(0.4,0) to[out=90,in=270] ++(-0.4,0.5) to[out=90,in=90,looseness=1.5] ++(-0.4,0) to[out=270,in=90] ++(-0.4,-0.5) ++(0.4,0) to[out=270,in=90] ++(-0.4,-0.5) ++(0.4,0) to[out=270,in=90] ++(-0.4,-0.5) -- ++(0,-0.25);
\end{scope}
\node at (-2,-0.025) {\small$\xrightarrow{\mathrm{quotient}}$};
\begin{scope}[xshift=0cm]
\draw[dotted] (0,0) circle (1);
\clip (0,0) circle (1);
\draw[semithick] (-0.6,-1) -- ++(0,0.25) to[out=90,in=180,looseness=1.25] ++(0,0.75) -- ++(1.2,0) to[out=0,in=90,looseness=1.25] ++(0,-0.75) -- ++(0,-0.25);
\draw[link,looseness=1] (-0.6,1) -- ++(0,-0.25) to[out=270,in=90] ++(0.4,-0.5) to[out=270,in=0] ++(-0.2,-0.5) to[out=180,in=270] ++(-0.2,0.4) to[out=90,in=180] ++(0.6,0.4) to[out=0,in=90] ++(0.6,-0.4) to[out=270,in=0] ++(-0.2,-0.4) to[out=180,in=270] ++(-0.2,0.5) to[out=90,in=270] ++(0.4,0.5) -- ++(0,0.25);
\draw[link,looseness=1] (-0.6,0.15) to[out=90,in=180] ++(0.6,0.4) to[out=0,in=90] ++(0.6,-0.4);
\begin{scope} 
\draw[line width=3pt,white] (-0.6,-1) -- ++(0,0.25) to[out=90,in=180,looseness=1.25] ++(0,0.75) -- ++(0.2,0) ++(0.8,0) -- ++(0.2,0) to[out=0,in=90,looseness=1.25] ++(0,-0.75) -- ++(0,-0.25);
\draw[semithick] (-0.6,-1) -- ++(0,0.25) to[out=90,in=180,looseness=1.25] ++(0,0.75) -- ++(0.2,0) ++(0.8,0) -- ++(0.2,0) to[out=0,in=90,looseness=1.25] ++(0,-0.75) -- ++(0,-0.25);
\end{scope}
\end{scope}
\node at (2,-0.025) {\small$\xrightarrow{\mathrm{isotopy}}$};
\begin{scope}[xshift=4cm]
\draw[dotted] (0,0) circle (1);
\clip (0,0) circle (1);
\draw[link,looseness=1.5] (-0.6,1) -- ++(0,-0.25) to[out=270,in=270] ++(1.2,0) -- ++(0,0.25);
\draw[semithick] (0.6,-1) -- ++(0,0.25) to[out=90,in=270] ++(-0.5,0.5) -- ++(0,0.5) to[out=90,in=90,looseness=2] ++(0.25,0) to[out=270,in=0] (0,0);
\draw[semithick] (-0.6,-1) -- ++(0,0.25) to[out=90,in=270] ++(0.5,0.5) -- ++(0,0.5) to[out=90,in=90,looseness=2] ++(0.65,0) to[out=270,in=0] (0,-0.2);
\draw[semithick] (0,0) to[out=180,in=180,looseness=5] (0,-0.2);
\begin{scope} 
\draw[line width=3pt, white] (-0.1,-0.1) -- ++(0,0.2) ++(0.2,0) -- ++(0,-0.2);
\draw[semithick] (-0.1,-0.1) -- ++(0,0.2) ++(0.2,0) -- ++(0,-0.2);
\clip (-0.5,-0.3) rectangle (0.5,-0.1);
\draw[line width=3pt, white] (0,0) to[out=180,in=180,looseness=5] (0,-0.2) to[out=0,in=270] (0.55,0.25);
\draw[semithick] (0,0) to[out=180,in=180,looseness=5] (0,-0.2) to[out=0,in=270] (0.55,0.25);
\end{scope}
\begin{scope} 
\clip (-0.5,0) rectangle (0.25,0.75);
\draw[link,looseness=1.5] (-0.6,1) -- ++(0,-0.25) to[out=270,in=270] ++(1.2,0) -- ++(0,0.25);
\end{scope}
\begin{scope} \clip (0.25,0.25) rectangle (0.7,0.6);
\draw[line width=3pt,white] (0.55,0.25) to[out=90,in=90,looseness=2] ++(-0.65,0);
\draw[line width=3pt,white] (0.35,0.25) to[out=90,in=90,looseness=2] ++(-0.25,0);
\draw[semithick] (0.55,0.25) to[out=90,in=90,looseness=2] ++(-0.65,0);
\draw[semithick] (0.35,0.25) to[out=90,in=90,looseness=2] ++(-0.25,0);
\end{scope}
\end{scope}
\end{scope}

\end{tikzpicture}
\caption{Quotients of some tangles by a $180^\circ$ rotation about an axis of symmetry.  The topmost tangle intersects the axis in two points, while the other two are disjoint from it.}
\label{fig:tangle-quotients}
\end{figure}
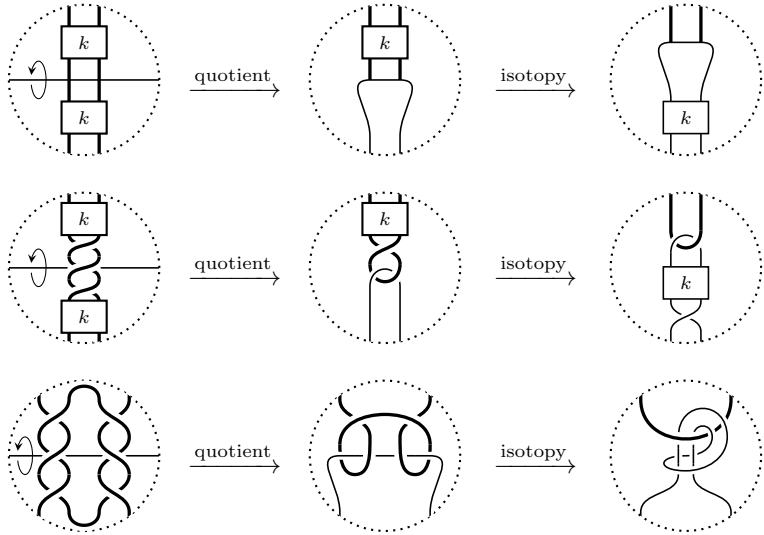
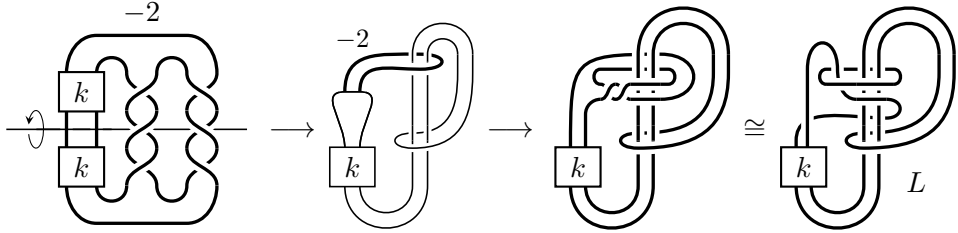
\begin{figure}
\tikzset{twistregion/.style={draw, fill=white, thick, minimum width=0.6cm}}
\tikzset{thinlink/.style = { white, double = black, line width = 1.2pt, double distance = 0.6pt, looseness=1.75 }}
\newcommand{\AxisRotator}{\tikz [x=0.10cm,y=0.25cm,line width=.2ex,-stealth] \draw[thin] (0,0) arc (-165:165:1 and 1);}
\begin{tikzpicture}
\begin{scope} 
\draw[semithick] (-1.6,0.75) -- (1.6,0.75); 
\draw[link,looseness=1] (0,1.5) to[out=270,in=90] ++(0.4,-0.5) ++(-0.4,0) to[out=270,in=90] ++(0.4,-0.5) ++(-0.4,0) to[out=270,in=90] ++(0.4,-0.5) to[out=270,in=270,looseness=1.75] ++(0.4,0) to[out=90,in=270] ++(0.4,0.5) ++(-0.4,0) to[out=90,in=270] ++(0.4,0.5) ++(-0.4,0) to[out=90,in=270] ++(0.4,0.5);
\draw[link,looseness=1] (1.2,0) to[out=90,in=270] ++(-0.4,0.5) ++(0.4,0) to[out=90,in=270] ++(-0.4,0.5) ++(0.4,0) to[out=90,in=270] ++(-0.4,0.5) to[out=90,in=90,looseness=1.75] ++(-0.4,0) to[out=270,in=90] ++(-0.4,-0.5) ++(0.4,0) to[out=270,in=90] ++(-0.4,-0.5) ++(0.4,0) to[out=270,in=90] ++(-0.4,-0.5);
\draw[link] (0,1.5) to[out=90,in=90] ++(-0.4,0) to[out=270,in=90] ++(0,-1.5) to[out=270,in=270] ++(0.4,0);
\draw[link,looseness=1] (1.2,1.5) to[out=90,in=0] ++(-0.4,0.5) to[out=180,in=0] node[midway,above,black] {$-2$} ++(-1.2,0) to[out=180,in=90] ++(-0.4,-0.5) to[out=270,in=90] ++(0,-1.5) to[out=270,in=180] ++(0.4,-0.5) to[out=0,in=180] ++(1.2,0) to[out=0,in=270] ++(0.4,0.5);
\node[twistregion] at (-0.6,1.25) {$k$};
\node[twistregion] at (-0.6,0.25) {$k$};
\draw[semithick] (-1.2,0.75) -- ++(1,0); 
\node at (-1.2,0.75) {\AxisRotator};
\end{scope}

\node at (2.2,0.725) {$\longrightarrow$};

\begin{scope}[xshift=3cm]
\draw[thinlink] (-0.1,0) -- ++(0,0.5) to[out=90,in=180,looseness=1.25] ++(0,0.75) -- ++(0.2,0) to[out=0,in=90,looseness=1.25] ++(0,-0.75) -- ++(0,-0.5);
\draw[thinlink] (0.1,0) to[out=270,in=270] (0.8,0) -- ++(0,1.75) to[out=90,in=90] ++(0.8,0);
\draw[thinlink] (-0.1,0) to[out=270,in=270] (1.0,0) -- ++(0,1.75) to[out=90,in=90] ++(0.4,0);
\draw[thinlink,looseness=1.25] (1.6,1.75) -- ++(0,-0.25) to[out=270,in=0] ++(-0.8,-1) to[out=180,in=180,looseness=4] ++(0,0.2) to[out=0,in=270] ++(0.6,0.8) -- ++(0,0.25);
\draw[thinlink] (0.8,0.6) -- ++(0,0.2) ++(0.2,-0.2) -- ++(0,0.2); 
\draw[link,looseness=1.25] (-0.1,1.25) to[out=90,in=180] ++(1,0.5) to[out=0,in=0,looseness=5] ++(0,-0.2) to[out=180,in=90] ++(-0.8,-0.3);
\draw[thinlink] (0.8,1.65) -- ++(0,0.1) to[out=90,in=90] ++(0.8,0); 
\draw[thinlink] (1.0,1.65) -- ++(0,0.1) to[out=90,in=90] ++(0.4,0);
\draw[semithick] (-0.1,0.5) to[out=90,in=180,looseness=1.25] ++(0,0.75) -- ++(0.2,0) to[out=0,in=90,looseness=1.25] ++(0,-0.75);
\node[twistregion,semithick] at (0,0.25) {$k$};
\node at (0,1.95) {\small$-2$};
\end{scope}

\node at (5.1,0.725) {$\longrightarrow$};

\begin{scope}[xshift=6cm]
\draw[link] (-0.1,0) -- ++(0,0.5) to[out=90,in=180] ++(1.2,1.25) to[out=0,in=0,looseness=2.5] ++(0,-0.6) -- ++(-0.4,0) to[out=180,in=0,looseness=1] ++(-0.2,0.2) to[out=180,in=0,looseness=1] ++(-0.2,-0.2) to[out=180,in=90,looseness=1] ++(-0.2,-0.2) -- (0.1,0);
\draw[link,looseness=1.5] (0.3,1.35) to[out=180,in=180] ++(0,0.2) -- ++(0.9,0) to[out=0,in=0] ++(0,-0.2) -- ++(-0.5,0) to[out=180,in=0,looseness=1] ++(-0.2,-0.2) to[out=180,in=0,looseness=1] ++(-0.2,0.2);
\draw[link] (0.1,0) to[out=270,in=270] (0.8,0) -- ++(0,1.75) to[out=90,in=90] ++(1.2,0);
\draw[link] (-0.1,0) to[out=270,in=270] (1.0,0) -- ++(0,1.75) to[out=90,in=90] ++(0.8,0);
\draw[link,looseness=1.25] (2,1.75) -- ++(0,-0.25) to[out=270,in=0] ++(-1.2,-1) to[out=180,in=180,looseness=4] ++(0,0.2) to[out=0,in=270] ++(1,0.8) -- ++(0,0.25);
\draw[link] (0.8,0.6) -- ++(0,0.2) ++(0.2,-0.2) -- ++(0,0.2); 
\draw[link] (0.3,1.15) to[out=0,in=180,looseness=1] ++(0.2,0.2) ++(0.2,0) -- ++(0.4,0) ++(0,-0.2) -- ++(-0.4,0);
\node[twistregion] at (0,0.25) {$k$};
\end{scope}

\node at (8.35,0.75) {$\cong$};

\begin{scope}[xshift=9cm]
\draw[link] (-0.1,0) -- ++(0,0.6) to[out=90,in=180,looseness=1] ++(1.2,0.3) to[out=0,in=0,looseness=2.5] ++(0,0.25) -- ++(-0.4,0);
\draw[link,looseness=1] (0.7,1.15) to[out=180,in=270] ++(-0.25,0.4) to[out=90,in=90,looseness=3] ++(-0.4,0) -- ++(0,-1.55);
\draw[link,looseness=1.5] (0.3,1.35) to[out=180,in=180] ++(0,0.2) -- ++(0.9,0) to[out=0,in=0] ++(0,-0.2) -- ++(-0.9,0);
\draw[link] (0.1,0) to[out=270,in=270] (0.8,0) -- ++(0,1.75) to[out=90,in=90] ++(1.2,0);
\draw[link] (-0.1,0) to[out=270,in=270] (1.0,0) -- ++(0,1.75) to[out=90,in=90] ++(0.8,0);
\draw[link,looseness=1.25] (2,1.75) -- ++(0,-0.25) to[out=270,in=0] ++(-1.2,-1) to[out=180,in=180,looseness=4] ++(0,0.2) to[out=0,in=270] ++(1,0.8) -- ++(0,0.25);
\draw[link] (0.8,0.6) -- ++(0,0.2) ++(0.2,-0.2) -- ++(0,0.2); 
\draw[link] (1.2,1.35) -- ++(-0.5,0) ++(0,-0.2) -- ++(0.4,0);
\begin{scope} \clip (0.35,1.45) rectangle (0.75,1.65);
\draw[link] (0.7,1.15) to[out=180,in=270,looseness=1] ++(-0.25,0.4) to[out=90,in=90,looseness=3] ++(-0.4,0);
\end{scope}
\node[twistregion] at (0,0.25) {$k$};
\node at (1.5,0.05) {$L$};
\end{scope}

\end{tikzpicture}
\caption{Realizing $S^3_{-2}(P(2k,3,-3))$ as a branched double cover.  We quotient by rotation around the indicated axis of symmetry, replace a neighborhood of the remaining arc with a rational tangle, and then isotope to get a link $L$ whose branched double cover is $S^3_{-2}(P(2k,3,-3))$.}
\label{fig:P-2k-dcover}
\end{figure}
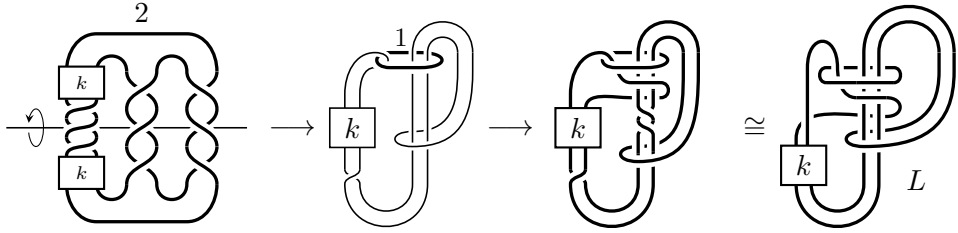
\begin{figure}
\tikzset{twistregion/.style={draw, fill=white, thick, minimum width=0.6cm}}
\tikzset{thinlink/.style = { white, double = black, line width = 1.2pt, double distance = 0.6pt, looseness=1.75 }}
\newcommand{\AxisRotator}{\tikz [x=0.10cm,y=0.25cm,line width=.2ex,-stealth] \draw[thin] (0,0) arc (-165:165:1 and 1);}
\begin{tikzpicture}
\begin{scope} 
\draw[semithick] (-1.6,0.75) -- (1.6,0.75); 
\draw[link,looseness=1] (0,1.5) to[out=270,in=90] ++(0.4,-0.5) ++(-0.4,0) to[out=270,in=90] ++(0.4,-0.5) ++(-0.4,0) to[out=270,in=90] ++(0.4,-0.5) to[out=270,in=270,looseness=1.75] ++(0.4,0) to[out=90,in=270] ++(0.4,0.5) ++(-0.4,0) to[out=90,in=270] ++(0.4,0.5) ++(-0.4,0) to[out=90,in=270] ++(0.4,0.5);
\draw[link,looseness=1] (1.2,0) to[out=90,in=270] ++(-0.4,0.5) ++(0.4,0) to[out=90,in=270] ++(-0.4,0.5) ++(0.4,0) to[out=90,in=270] ++(-0.4,0.5) to[out=90,in=90,looseness=1.75] ++(-0.4,0) to[out=270,in=90] ++(-0.4,-0.5) ++(0.4,0) to[out=270,in=90] ++(-0.4,-0.5) ++(0.4,0) to[out=270,in=90] ++(-0.4,-0.5);
\draw[link] (0,1.5) to[out=90,in=90] ++(-0.4,0) -- ++(0,-0.375) ++(0,-0.75) -- ++(0,-0.375) to[out=270,in=270] ++(0.4,0);
\draw[link,looseness=1] (1.2,1.5) to[out=90,in=0] ++(-0.4,0.5) to[out=180,in=0] node[midway,above,black] {$2$} ++(-1.2,0) to[out=180,in=90] ++(-0.4,-0.5) -- ++(0,-0.375) ++(0,-0.75) -- ++(0,-0.375) to[out=270,in=180] ++(0.4,-0.5) to[out=0,in=180] ++(1.2,0) to[out=0,in=270] ++(0.4,0.5);
\draw[link,looseness=1] (-0.4,0.375) to[out=90,in=270] ++(-0.4,0.25) ++(0.4,0) to[out=90,in=270] ++(-0.4,0.25) ++(0.4,0) to[out=90,in=270] ++(-0.4,0.25);
\draw[link,looseness=1] (-0.8,0.375) to[out=90,in=270] ++(0.4,0.25) ++(-0.4,0) to[out=90,in=270] ++(0.4,0.25) ++(-0.4,0) to[out=90,in=270] ++(0.4,0.25);
\node[twistregion] at (-0.6,1.35) {\tiny$k$};
\node[twistregion] at (-0.6,0.15) {\tiny$k$};
\node at (-1.2,0.75) {\AxisRotator};
\end{scope}

\node at (2.2,0.725) {$\longrightarrow$};

\begin{scope}[xshift=3cm]
\draw[thinlink] (-0.1,0.2) to[out=270,in=90,looseness=1] (0.1,0) to[out=270,in=270] (0.8,0) -- ++(0,1.75) to[out=90,in=90] ++(0.8,0);
\draw[thinlink] (0.1,0.2) to[out=270,in=90,looseness=1] (-0.1,0) to[out=270,in=270] (1.0,0) -- ++(0,1.75) to[out=90,in=90] ++(0.4,0);
\draw[thinlink] (-0.1,0.2) -- ++(0,1) to[out=90,in=180,looseness=1] (0.3,1.75) to[out=0,in=0,looseness=3] ++(0,-0.2) to[out=180,in=90,looseness=1] (0.1,1.2) -- ++(0,-1);
\draw[thinlink,looseness=1.25] (1.6,1.75) -- ++(0,-0.25) to[out=270,in=0] ++(-0.8,-1) to[out=180,in=180,looseness=4] ++(0,0.2) to[out=0,in=270] ++(0.6,0.8) -- ++(0,0.25);
\draw[thinlink] (0.8,0.6) -- ++(0,0.2) ++(0.2,-0.2) -- ++(0,0.2); 
\draw[link,looseness=1.25] (0.5,1.75) -- ++(0.4,0) to[out=0,in=0,looseness=5] ++(0,-0.2) -- ++(-0.4,0) to[out=180,in=180,looseness=3] ++(0,0.2);
\draw[thinlink] (0.8,1.65) -- ++(0,0.1) to[out=90,in=90] ++(0.8,0); 
\draw[thinlink] (1.0,1.65) -- ++(0,0.1) to[out=90,in=90] ++(0.4,0);
\begin{scope} 
\clip (0,1.65) rectangle (0.5,1.85);
\draw[thinlink] (-0.1,1.2) to[out=90,in=180,looseness=1] (0.3,1.75) to[out=0,in=0,looseness=3] ++(0,-0.2);
\end{scope}
\node[twistregion,semithick] at (0,0.75) {$k$};
\node at (0.65,1.95) {\small$1$};
\end{scope}

\node at (5.1,0.725) {$\longrightarrow$};

\begin{scope}[xshift=6cm]
\draw[link] (-0.1,0.2) to[out=270,in=90,looseness=1] (0.1,0) to[out=270,in=270] (0.8,0) -- ++(0,1.75) to[out=90,in=90] ++(0.8,0);
\draw[link] (0.1,0.2) to[out=270,in=90,looseness=1] (-0.1,0) to[out=270,in=270] (1.0,0) -- ++(0,1.75) to[out=90,in=90] ++(0.4,0);
\draw[link] (-0.1,0.2) -- ++(0,1) to[out=90,in=180,looseness=1] (0.3,1.75) to[out=0,in=180,looseness=1] ++(0.4,-0.4) -- ++(0.4,0) to[out=0,in=0] ++(0,-0.2) -- ++(-0.4,0) to[out=180,in=90,looseness=1] (0.1,1) -- ++(0,-0.8);
\draw[link,looseness=1.25] (1.6,1.75) -- ++(0,-0.25) to[out=270,in=0] ++(-0.8,-1.2) to[out=180,in=180,looseness=4] ++(0,0.2) to[out=0,in=270] ++(0.6,1.0) -- ++(0,0.25);
\draw[link,looseness=1] (0.8,0.41) -- ++(0,0.24) to[out=90,in=270] ++(0.2,0.2) ++(-0.2,0) to[out=90,in=270] ++(0.2,0.2) -- ++(0,0.2); 
\draw[link,looseness=1] (1.0,0.43) -- ++(0,0.22) to[out=90,in=270] ++(-0.2,0.2) ++(0.2,0) to[out=90,in=270] ++(-0.2,0.2) -- ++(0,0.2);
\draw[link,looseness=1.25] (0.5,1.75) -- ++(0.4,0) to[out=0,in=0,looseness=5] ++(0,-0.2) -- ++(-0.4,0) to[out=180,in=180,looseness=3] ++(0,0.2);
\draw[link] (0.8,1.65) -- ++(0,0.1) to[out=90,in=90] ++(0.8,0); 
\draw[link] (1.0,1.65) -- ++(0,0.1) to[out=90,in=90] ++(0.4,0);
\begin{scope} 
\clip (0,1.65) rectangle (0.6,1.85);
\draw[link] (-0.1,1.2) to[out=90,in=180,looseness=1] (0.3,1.75) to[out=0,in=180,looseness=1] ++(0.4,-0.4);
\end{scope}
\node[twistregion] at (0,0.75) {$k$};
\end{scope}

\node at (8.35,0.75) {$\cong$};

\begin{scope}[xshift=9cm]
\draw[link] (-0.1,0) -- ++(0,0.6) to[out=90,in=180,looseness=1] ++(1.2,0.3) to[out=0,in=0,looseness=2.5] ++(0,0.25) -- ++(-0.4,0);
\draw[link,looseness=1] (0.7,1.15) to[out=180,in=270] ++(-0.25,0.4) to[out=90,in=90,looseness=3] ++(-0.4,0) -- ++(0,-1.55);
\draw[link,looseness=1.5] (0.3,1.35) to[out=180,in=180] ++(0,0.2) -- ++(0.9,0) to[out=0,in=0] ++(0,-0.2) -- ++(-0.9,0);
\draw[link] (0.1,0) to[out=270,in=270] (0.8,0) -- ++(0,1.75) to[out=90,in=90] ++(1.2,0);
\draw[link] (-0.1,0) to[out=270,in=270] (1.0,0) -- ++(0,1.75) to[out=90,in=90] ++(0.8,0);
\draw[link,looseness=1.25] (2,1.75) -- ++(0,-0.25) to[out=270,in=0] ++(-1.2,-1) to[out=180,in=180,looseness=4] ++(0,0.2) to[out=0,in=270] ++(1,0.8) -- ++(0,0.25);
\draw[link] (0.8,0.6) -- ++(0,0.2) ++(0.2,-0.2) -- ++(0,0.2); 
\draw[link] (1.2,1.35) -- ++(-0.5,0) ++(0,-0.2) -- ++(0.4,0);
\begin{scope} \clip (0.35,1.45) rectangle (0.75,1.65);
\draw[link] (0.7,1.15) to[out=180,in=270,looseness=1] ++(-0.25,0.4) to[out=90,in=90,looseness=3] ++(-0.4,0);
\end{scope}
\node[twistregion] at (0,0.25) {$k$};
\node at (1.5,0.05) {$L$};
\end{scope}

\end{tikzpicture}
\caption{Realizing $S^3_2(P(2k+3,3,-3))$ as a branched double cover.  We quotient by rotation around the axis of symmetry, then add the core of the resulting 1-surgery torus to the branch locus and blow it down to get the same $L$ as in Figure~\ref{fig:P-2k-dcover}.}
\label{fig:P-2k3-dcover}
\end{figure}

First, the pretzel $P=P(2k,3,-3)$ is strongly invertible, as realized by a $180^\circ$ rotation around the axis shown on the left side of Figure~\ref{fig:P-2k-dcover}.  Thus in Figure~\ref{fig:P-2k-dcover} we construct a 2-component link $L$ with branched double cover $S^3_{-2}(P)$ by the process described above.  We use Figure~\ref{fig:tangle-quotients} to simplify our work: we can build $P$ by gluing the top left and bottom left tangles together, so the quotient is built out of the top right and bottom right tangles in Figure~\ref{fig:tangle-quotients}.

Next, the pretzel $P'=P(2k+3,3,-3)$ is periodic of period $2$, so we can describe $S^3_2(P')$ as the branched double cover of a link $L = U/\iota \cup P'/\iota$ in $S^3_1(P'/\iota)$ where $\iota$ is the corresponding involution and $U$ its axis of symmetry.  This is illustrated in Figure~\ref{fig:P-2k3-dcover}, using the middle and bottom rows of Figure~\ref{fig:tangle-quotients} to simplify.  Since $P'/\iota$ is unknotted, this is a 2-component link in $S^3$, and in fact it is the same link $L$ whose branched double cover was $S^3_{-2}(P)$, so the two 3-manifolds must be homeomorphic.
\end{proof}

\subsection{Low-crossing knots}

We  first determine the value of $\chominvt$ for all prime knots with up to eight crossings, and compute $\cinvt$ for all but two of these knots (see Table~\ref{table:small-values}). We then use these computations,   the results of the previous two subsections, and a bit more topology to compute $r_0$ for more than half of these knots (see Tables \ref{table:main} and \ref{table:isharp-values}), proving Theorem \ref{thm:8-crossings}.

The computations of $\chominvt$ require minimal effort.  Of the prime knots through eight crossings, all but $8_{19}$, $8_{20}$, and $8_{21}$ are alternating, so  $\chominvt = -\sigma/2$ for these knots by Corollary~\ref{cor:alternating}.  Moreover, each of $8_{19}$, $8_{20}$, and $\mirror{8_{21}}$ is quasipositive, so Corollary~\ref{cor:chom-qp} says that
\begin{align*}
\chominvt(8_{19}) &= g_s(8_{19}) = 3\\
\chominvt(8_{20}) &= g_s(8_{20}) = 0\\
\chominvt(8_{21}) &=-\chominvt(\mirror{8_{21}})  =-g_s(8_{21}) = -1.
\end{align*}

The computations of $\cinvt$ are slightly more involved.  In some cases,  Theorem~\ref{thm:conc-invt} suffices:\begin{itemize}
\item If $K$ is the unknot, $6_1$, $8_8$, $8_9$, or $8_{20}$, then $K$ is slice and so $\cinvt(K) = 0$.
\item If $K$ is $4_1$, $6_3$, $8_3$, $8_{12}$, $8_{17}$, or $8_{18}$ then $K$ is amphichiral and so $\cinvt(K) = 0$.
\end{itemize}
In most other cases, we have $|\chominvt(K)| = g_s(K) > 0$, which implies that $|\cinvt(K)| = 2g_s(K)-1$, with $\cinvt(K)$ and $\chominvt(K)$ having the same sign, by Corollary~\ref{cor:tau-maximal}.  In fact, this determines $\cinvt$ for all of the remaining knots  except for $7_7,8_1,8_{13}$, and we have  $\cinvt(8_1) = 0$ by Proposition~\ref{prop:twist-2n+2-surgery}, since $8_1$ is the twist knot $K_6$. The values of $\cinvt(7_7)$ and $\cinvt(8_{13})$ remain unknown, though by Proposition~\ref{prop:tau-invariant}  they must each be $-1$, $0$, or $1$. These results  are summarized in  Table~\ref{table:small-values}.

\begin{table}
\centering
\begin{tabular}{ccc}
$K$ & $\cinvt(K)$ & $\chominvt(K)$ \\
\hline
$0_1$ & $0$ & $0$  \\
$3_1$ & $-1$ & $-1$ \\
$4_1$ & $0$ & $0$  \\
$5_1$ & $-3$ & $-2$ \\
$5_2$ & $-1$ & $-1$ \\
$6_1$ & $0$ & $0$ \\
$6_2$ & $-1$ & $-1$ \\
$6_3$ & $0$ & $0$ \\
$7_1$ & $-5$ & $-3$ \\
$7_2$ & $-1$ & $-1$  \\
$7_3$ & $3$ & $2$ \\
$7_4$ & $1$ & $1$ \\
$7_5$ & $-3$ & $-2$ \\
$7_6$ & $-1$ & $-1$ \\
$7_7$ & & $0$ \\
$8_1$ & $0$ & $0$ \\
$8_2$ & $-3$ & $-2$ \\
$8_3$ & $0$ & $0$ \\
\hline
\end{tabular}
\hspace{0.5cm}
\begin{tabular}{ccc}
$K$ & $\cinvt(K)$ & $\chominvt(K)$ \\
\hline
$8_4$ & $-1$ & $-1$ \\
$8_5$ & $3$ & $2$ \\
$8_6$ & $-1$ & $-1$ \\
$8_7$ & $1$ & $1$ \\
$8_8$ & $0$ & $0$ \\
$8_9$ & $0$ & $0$ \\
$8_{10}$ & $1$ & $1$ \\
$8_{11}$ & $-1$ & $-1$ \\
$8_{12}$ & $0$ & $0$ \\
$8_{13}$ & & $0$ \\
$8_{14}$ & $-1$ & $-1$ \\
$8_{15}$ & $-3$ & $-2$ \\
$8_{16}$ & $-1$ & $-1$ \\
$8_{17}$ & $0$ & $0$ \\
$8_{18}$ & $0$ & $0$ \\
$8_{19}$ & $5$ & $3$ \\
$8_{20}$ & $0$ & $0$ \\
$8_{21}$ & $-1$ & $-1$ \\
\hline
\end{tabular}
\caption{The known values of $\cinvt(K)$ and $\chominvt(K)$ through eight crossings.} \label{table:small-values}
\end{table}

We next compute $r_0$ for 20 of the 35 nontrivial prime knots through eight crossings. These values are listed in Table \ref{table:main} from the introduction and again, with some further information, in Table \ref{table:isharp-values} here. Together with Theorem \ref{thm:rational-surgeries} and the values of $\cinvt$ we  computed, these values of $r_0$ allow us to compute the framed instanton homology of nonzero rational surgeries on any of these 20 knots (where $\cinvt$ is nonzero, we can also compute $I^\#$ of $0$-surgery). 

\begin{table}
\centering
\begin{tabular}{ccccclc}
$K$ & $n$ & $\dim I^\#(S^3_n(K))$ & $\cinvt(K)$ & $r_0(K)$ & Proof \\
\hline
$3_1$ & $-5$ & $5$ & $-1$ & $1$ & Lemma~\ref{lem:t2q-surgery} ($T_{-2,3}$) \\
$4_1$ & $1$ & $3$ & $0$ & $2$ & Lemma \ref{lem:figure-eight}
\\
$5_1$ & $-9$ & $9$ & $-3$ & $3$ & Lemma~\ref{lem:t2q-surgery} ($T_{-2,5}$) \\
$5_2$ & $-1$ & $3$ & $-1$ & $3$ & Proposition~\ref{prop:twist-2k-1-surgery} ($K_3$) \\
$6_1$ & $1$ & $5$ & $0$ & $4$ & Proposition~\ref{prop:twist-2n+2-surgery} ($K_4$) \\
$6_2$ & $-9$ & $13$ & $-1$ & $5$ & Proposition~\ref{prop:isharp-5_2-surgeries} \\
$6_3$ & $-1$ & $7$ & $0$ & $6$ & Proposition~\ref{prop:6_3-surgeries} \\
$7_1$ & $-13$ & $13$ & $-5$ & $5$ & Lemma~\ref{lem:t2q-surgery} ($T_{-2,7}$) \\
$7_2$ & $-1$ & $5$ & $-1$ & $5$ & Proposition~\ref{prop:twist-2k-1-surgery} ($K_5$) \\
$7_3$ & $7$ & $11$ & $3$ & $7$ & Proposition~\ref{prop:isharp-5_2-surgeries} \\
$7_4$ & $1$ & $7$ & $1$ & $7$ & Proposition~\ref{prop:isharp-5_2-surgeries} \\
$8_1$ & $1$ & $7$ & $0$ & $6$ & Proposition~\ref{prop:twist-2n+2-surgery} ($K_6$)\\
$8_2$ & $-13$ & $19$ & $-3$ & $9$ & Proposition~\ref{prop:dim-8_234} \\
$8_3$ & $1$ & $9$ & $0$ & $8$ & Proposition~\ref{prop:dim-8_234} \\
$8_4$ & $-7$ & $15$ & $-1$ & $9$ & Proposition~\ref{prop:dim-8_234} \\
$8_5$ & $7$ & $15$ & $3$ & $11$ & Theorem~\ref{thm:main-pretzels-2} ($P(3,3,2)$) \\
$8_6$ & $-1$ & $11$ & $-1$ & $11$ & Proposition~\ref{prop:6_3-surgeries} \\
$8_8$ & $-1$ & $13$ & $0$ & $12$ & Proposition~\ref{prop:6_3-surgeries} \\
$8_{19}$ & $11$ & $11$ & $5$ & $5$ & Lemma~\ref{lem:t2q-surgery} ($T_{3,4}$) \\
$8_{20}$ & $1$ & $5$ & $0$ & $4$ & Theorem~\ref{thm:main-pretzels} ($P(2,3,-3)$) \\
\hline
\end{tabular}
\caption{Some values of $\dim I^\#(S^3_n(K))$ which together with $\cinvt(K)$ determine $r_0(K)$ by Theorem~\ref{thm:rational-surgeries}.} \label{table:isharp-values}
\end{table}

As Table \ref{table:isharp-values} indicates, 10 of these 20 knots are either torus knots ($3_1,5_1,7_1,8_{19}$), or twist knots ($4_1,5_2,6_1,7_2,8_1$), or pretzels ($8_{20}$) of the form we have considered previously. The values of $r_0$ for these are thus computed in Lemma~\ref{lem:t2q-surgery} and Theorems~\ref{thm:main-twist} and \ref{thm:main-pretzels}. For the others, recall from Theorem~\ref{thm:rational-surgeries} that once we know $\cinvt(K)$, the value of $r_0(K)$ is determined by \[\dim I^\#(S^3_{r}(K))\] for any nonzero rational number $r$. The values of $r_0$ for the remaining 10 knots are therefore determined by Proposition~\ref{prop:isharp-5_2-surgeries} (for $6_2,7_3,7_4$), Proposition~\ref{prop:dim-8_234} (for $8_2,8_3,8_4$), Theorem~\ref{thm:main-pretzels-2} (proved below, for $8_5$), and Proposition~\ref{prop:6_3-surgeries} (for $6_3,8_6,8_8$) below. The rest of this section is devoted to proving these propositions.

\begin{proposition} \label{prop:isharp-5_2-surgeries}
We have
\begin{align*}
\dim I^\#(S^3_{-9}(6_2)) &= 13, & \dim I^\#(S^3_7(7_3)) &= 11, & \dim I^\#(S^3_1(7_4)) &= 7.
\end{align*}
\end{proposition}

\begin{proof}
In the notation of Proposition~\ref{prop:two-bridge-surgeries}, we have
\begin{align*}
\mirror{5_2} &= K(2,-3) = K(-2,-4) &
6_2 &= K(-3,-4) \\
7_3 &= K(-3,4) &
7_4 &= K(-4,-4).
\end{align*}
Thus two applications of \eqref{eq:surgery-2-bridge-odd} and one of \eqref{eq:surgery-2-bridge-even}, respectively, give
\begin{align} \label{eq:mirror-5_2-surgeries}
S^3_{-9}(6_2) &\cong S^3_{9/2}(\mirror{5_2}), &
S^3_{7}(7_3) &\cong S^3_{7/2}(\mirror{5_2}), &
S^3_{1}(7_4) &\cong S^3_{1/2}(\mirror{5_2}).
\end{align}
Since $\cinvt(\mirror{5_2}) = 1$ and $r_0(\mirror{5_2}) =  3$, as follows from   Table \ref{table:isharp-values}, Theorem~\ref{thm:rational-surgeries} says that
\begin{equation} \label{eq:5_2-surgeries}
\dim I^\#(S^3_{p/q}(\mirror{5_2})) = 3q + |p-q|
\end{equation}
for all rational $\frac{p}{q}$, so we apply \eqref{eq:mirror-5_2-surgeries} to get
\begin{align*}
\dim I^\#(S^3_{-9}(6_2)) &= \dim I^\#(S^3_{9/2}(\mirror{5_2})) = 13 \\
\dim I^\#(S^3_7(7_3)) &= \dim I^\#(S^3_{7/2}(\mirror{5_2})) = 11 \\
\dim I^\#(S^3_1(7_4)) &= \dim I^\#(S^3_{1/2}(\mirror{5_2})) = 7.\qedhere
\end{align*} \end{proof}

The next proposition is proved almost identically.

\begin{proposition} \label{prop:dim-8_234}
We have
\begin{align*}
\dim I^\#(S^3_{-13}(8_2)) &= 19, &
\dim I^\#(S^3_1(8_3)) &= 9, &
\dim I^\#(S^3_{-7}(8_4)) &= 15.
\end{align*}
\end{proposition}

\begin{proof}
In the notation of Proposition~\ref{prop:two-bridge-surgeries}, we have
\begin{align*}
8_2 &= K(-3,-6), &
8_3 &= K(-4,4), &
8_4 &= K(-5,-4).
\end{align*} Noting that \begin{align*}
\mirror{5_2} &= K(2,-3) =K(-2,-4) &
4_1 &= K(-2,-3) \\
\mirror{6_1} &= K(2,-4)=K(-2,-5) &
\mirror{7_2} &= K(2,-5), &
\end{align*}
 Proposition~\ref{prop:two-bridge-surgeries} tells us that
\begin{align*}
S^3_{-13}(8_2) &\cong S^3_{13/3}(\mirror{5_2}), &
S^3_{-11}(8_2) &\cong S^3_{11/3}(4_1), \\
S^3_1(8_3) &\cong S^3_{-1/2}(\mirror{5_2}), &
S^3_{-1}(8_3) &\cong S^3_{-1/2}(\mirror{6_1}), \\
S^3_{-9}(8_4) &\cong S^3_{9/2}(\mirror{7_2}), &
S^3_{-7}(8_4) &\cong S^3_{7/2}(\mirror{6_1}).
\end{align*}
Two applications of \eqref{eq:5_2-surgeries} yield
\begin{align*}
\dim I^\#(S^3_{-13}(8_2)) &= \dim I^\#(S^3_{13/3}(\mirror{5_2})) = 19, \\
\dim I^\#(S^3_1(8_3)) &= \dim I^\#(S^3_{-1/2}(\mirror{5_2})) = 9.
\end{align*}
Similarly, Theorem~\ref{thm:rational-surgeries} says that
\[ \dim I^\#(S^3_{p/q}(\mirror{6_1})) = 4q + |p| \]
for $\frac{p}{q} \neq 0$ since $\cinvt(\mirror{6_1}) = 0$ and $r_0(\mirror{6_1}) = 4$, so
\[ \dim I^\#(S^3_{-7}(8_4)) = \dim I^\#(S^3_{7/2}(\mirror{6_1})) = 15. \qedhere \]
\end{proof}

The computation of $r_0(6_2)$ from Proposition~\ref{prop:isharp-5_2-surgeries} also allows us to prove Theorem~\ref{thm:main-pretzels-2}, which asserts that
\begin{align*}
\cinvt(P(2n-1,3,2)) &= 2n-1, &
r_0(P(2n-1,3,2)) &= 6n-1
\end{align*}
for all $n \geq 1$.  In particular, this determines the value of $r_0$ for $8_5 \cong P(3,3,2)$.

\begin{proof}[Proof of Theorem~\ref{thm:main-pretzels-2}]
Let $K_n = P(2n-1,3,2)$; we note that $K_0 = P(-1,3,2)$ is the unknot and that $K_1 = \mirror{6_2}$ and $K_2 = 8_5$.  For all $n\geq 1$, the diagram of $K_n$ on the left side of Figure~\ref{fig:P-2n-1-3-2} is alternating, and we can compute its signature to be $\sigma(K_n) = -2n$ by the method of Gordon and Litherland \cite{gordon-litherland}.  This tells us that $\chominvt(K_n) = n$ by Corollary~\ref{cor:alternating}.

For any $n \geq 1$, we can change $n$ of the crossings in the $(2n-1)$-twist region of $K_n$ from positive to negative and thus turn $K_n$ into a knot isotopic to $K_0$.  This provides a genus-$n$ cobordism from $K_n$ to the unknot, so
\[ g_s(K_n) \leq n = \chominvt(K_n) \leq g_s(K_n), \]
where the last inequality comes from Proposition~\ref{prop:tau-invariant}.  But then $\chominvt(K_n) = g_s(K_n) = n$, and so
\[ \cinvt(K_n) = 2n-1 \textrm{ for all } n \geq 1\]
by Corollary~\ref{cor:tau-maximal}.

Figure~\ref{fig:P-2n-1-3-2} shows that there is a knot $Q \subset S^3$ such that
\begin{equation} \label{eq:P-2n-1-3-2-identity}
S^3_{4n-1}(K_n) \cong S^3_{(4n-1)/n}(Q)
\end{equation}
for all $n \in \Z$.
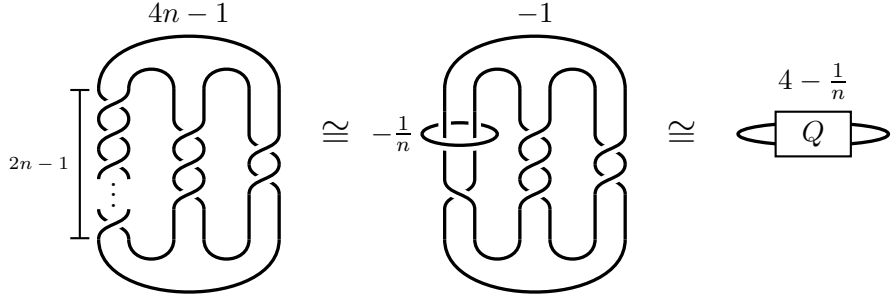
\begin{figure}
\begin{tikzpicture}[link/.append style = { looseness=1 }]
\begin{scope}
\draw[link] (0,2) to[out=270,in=90] ++(0.4,-0.4) ++(-0.4,0) to[out=270,in=90] ++(0.4,-0.4) ++(-0.4,0) to[out=270,in=90] ++(0.4,-0.4) ++(0,-0.4) ++(-0.4,0) to[out=270,in=90] ++(0.4,-0.4) ++(-0.4,0);
\draw[link] (0.4,2) to[out=270,in=90] ++(-0.4,-0.4) ++(0.4,0) to[out=270,in=90] ++(-0.4,-0.4) ++(0.4,0) to[out=270,in=90] ++(-0.4,-0.4) ++(0,-0.4) ++(0.4,0) to[out=270,in=90] ++(-0.4,-0.4) ++(0.4,0);
\node at (0.2,0.7) {$\vdots$};
\draw[link] (1.0,2) -- ++(0,-0.4) to[out=270,in=90] ++(0.4,-0.4) ++(-0.4,0) to[out=270,in=90] ++(0.4,-0.4) ++(-0.4,0) to[out=270,in=90] ++(0.4,-0.4) ++(-0.4,0) -- ++(0,-0.4);
\draw[link] (1.4,2) -- ++(0,-0.4) to[out=270,in=90] ++(-0.4,-0.4) ++(0.4,0) to[out=270,in=90] ++(-0.4,-0.4) ++(0.4,0) to[out=270,in=90] ++(-0.4,-0.4) ++(0.4,0) -- ++(0,-0.4);
\draw[link] (2.0,2) -- ++(0,-0.6) to[out=270,in=90] ++(0.4,-0.4) ++(-0.4,0) to[out=270,in=90] ++(0.4,-0.4) ++(-0.4,0) -- ++(0,-0.6);
\draw[link] (2.4,2) -- ++(0,-0.6) to[out=270,in=90] ++(-0.4,-0.4) ++(0.4,0) to[out=270,in=90] ++(-0.4,-0.4) ++(0.4,0) -- ++(0,-0.6);
\draw[link,looseness=1.5] (0.4,2) to[out=90,in=90] ++(0.6,0) ++(0.4,0) to[out=90,in=90] ++(0.6,0);
\draw[link,looseness=1.5] (0.4,0) to[out=270,in=270] ++(0.6,0) ++(0.4,0) to[out=270,in=270] ++(0.6,0);
\draw[link] (0,2) to[out=90,in=90] ++(2.4,0);
\draw[link] (0,0) to[out=270,in=270] ++(2.4,0);
\draw[|-|] (-0.25,0) to node[midway,left] {\tiny $2n-1$} ++(0,2);
\node at (1.2,3) {$4n-1$};
\end{scope}

\node at (3.15,1.4) {\Large $\cong$};

\begin{scope}[xshift=4.6cm]
\draw[link] (0.4,2) -- ++(0,-1.2) to[out=270,in=90] ++(-0.4,-0.4) -- ++(0,-0.4);
\draw[link] (0,2) -- ++(0,-1.2) to[out=270,in=90] ++(0.4,-0.4) -- ++(0,-0.4);
\draw[link] (0.2,1.4) ellipse (0.5 and 0.15);
\draw[link] (0,1.4) -- ++(0,0.3) (0.4,1.4) -- ++(0,0.3);
\node at (-0.7,1.4) {$-\frac{1}{n}$};

\draw[link] (1.0,2) -- ++(0,-0.4) to[out=270,in=90] ++(0.4,-0.4) ++(-0.4,0) to[out=270,in=90] ++(0.4,-0.4) ++(-0.4,0) to[out=270,in=90] ++(0.4,-0.4) ++(-0.4,0) -- ++(0,-0.4);
\draw[link] (1.4,2) -- ++(0,-0.4) to[out=270,in=90] ++(-0.4,-0.4) ++(0.4,0) to[out=270,in=90] ++(-0.4,-0.4) ++(0.4,0) to[out=270,in=90] ++(-0.4,-0.4) ++(0.4,0) -- ++(0,-0.4);
\draw[link] (2.0,2) -- ++(0,-0.6) to[out=270,in=90] ++(0.4,-0.4) ++(-0.4,0) to[out=270,in=90] ++(0.4,-0.4) ++(-0.4,0) -- ++(0,-0.6);
\draw[link] (2.4,2) -- ++(0,-0.6) to[out=270,in=90] ++(-0.4,-0.4) ++(0.4,0) to[out=270,in=90] ++(-0.4,-0.4) ++(0.4,0) -- ++(0,-0.6);
\draw[link,looseness=1.5] (0.4,2) to[out=90,in=90] ++(0.6,0) ++(0.4,0) to[out=90,in=90] ++(0.6,0);
\draw[link,looseness=1.5] (0.4,0) to[out=270,in=270] ++(0.6,0) ++(0.4,0) to[out=270,in=270] ++(0.6,0);
\draw[link] (0,2) to[out=90,in=90] ++(2.4,0);
\draw[link] (0,0) to[out=270,in=270] ++(2.4,0);
\node at (1.2,3) {$-1$};
\end{scope}

\node at (7.75,1.4) {\Large $\cong$};

\begin{scope}[xshift=9.3cm]
\draw[link] (0.2,1.4) ellipse (1 and 0.15);
\node at (0.2,2.1) {$4-\frac{1}{n}$};
\node[draw,fill=white,thick,minimum width=1cm] at (0.2,1.4){$Q$};
\end{scope}

\end{tikzpicture}
\caption{Comparing $(4n-1)$-surgery on $P(2n-1,3,2)$ to $\frac{4n-1}{n}$-surgery on a knot $Q$, obtained by blowing down the $(-1)$-framed unknot $P(-1,3,2)$ in the middle.  (We make no effort to identify $Q$, except to note that it does not depend on $n \in \Z$.)}
\label{fig:P-2n-1-3-2}
\end{figure}
We therefore have
\begin{align*}
n=1: & \dim I^\#(S^3_3(Q)) = \dim I^\#(S^3_3(\mirror{6_2})) = 7 \\
n=-1: & \dim I^\#(S^3_5(Q)) = \dim I^\#(S^3_{-5}(P(2,3,-3)) = 9,
\end{align*}
since we know that $(\cinvt(\mirror{6_2}), r_0(\mirror{6_2})) = (1,5)$ and $(\cinvt(P(2,3,-3)), r_0(P(2,3,-3))) = (0,4)$ by Proposition~\ref{prop:isharp-5_2-surgeries} and Theorem~\ref{thm:main-pretzels} respectively.  These computations do not suffice to determine $\cinvt(Q)$, but they do ensure that $\cinvt(Q) \leq 3$, and hence (by the $n=1$ case) that
\[ 7 = r_0(Q) + (3-\cinvt(Q)) \quad\Longrightarrow\quad r_0(Q) - \cinvt(Q) = 4 \]
by Theorem~\ref{thm:rational-surgeries}.  For all integers $n \geq 1$ we now have
\[ \frac{4n-1}{n} \geq 3 \geq \cinvt(Q) \quad\Longrightarrow\quad (4n-1) - n\cinvt(Q) \geq 0, \]
and so we apply \eqref{eq:P-2n-1-3-2-identity} and Theorem~\ref{thm:rational-surgeries} to get
\begin{align*}
\dim I^\#(S^3_{4n-1}(K_n)) &= \dim I^\#(S^3_{(4n-1)/n}(Q)) \\
&= n\cdot r_0(Q) + \big((4n-1) - n\cinvt(Q)\big) \\
&= n(r_0(Q) - \cinvt(Q)) + (4n-1) = 8n-1.
\end{align*}
Since $\cinvt(K_n) = 2n-1$, it now follows that
\[ r_0(K_n) = \dim I^\#(S^3_{2n-1}(K_n)) = \dim I^\#(S^3_{4n-1}(K_n))-2n = 6n-1, \]
as claimed.
\end{proof}

Our proof of Proposition \ref{prop:6_3-surgeries} requires the following, which combines work of Lim \cite{lim} with Fukaya's connected sum theorem for instanton homology, as applied in \cite[\S9.8]{scaduto}.
\begin{proposition} \label{prop:alexander-lower-bound}
Let $K \subset S^3$ be a knot of genus $g(K) \leq 2$, with Alexander polynomial
\[ \Delta_K(t) = a_2t^2 + a_1t + a_0 + a_1t^{-1} + a_2t^{-2}. \]
Then $\dim I^\#(S^3_0(K),\mu) = 4|a_2| + 2|a_1+2a_2| + 4k$ for some integer $k \geq 0$.
\end{proposition}

\begin{proof}
Let $I_*(S^3_0(K))_w$ be the instanton Floer homology of $S^3_0(K)$, with coefficients in  $\C$, where $w \to S^3_0(K)$ is a line bundle with  $c_1(w)=\PD(\mu)$.  If $\hat\Sigma \subset S^3_0(K)$ is a capped-off Seifert surface for $K$, having genus $g=g(K)$, then there are commuting operators $\mu(\hat\Sigma)$ and $\mu(\pt)$ acting on the relatively $\Z/8\Z$-graded $I_*(S^3_0(K))_w$ with degrees $-2$ and $-4$ respectively.  Kronheimer and Mrowka proved in \cite{km-excision}, based on work of Mu\~noz \cite{munoz-ring}, that their simultaneous eigenvalues are a subset of
\[ (2k\cdot i^r, (-2)^r), \quad 0 \leq k \leq g-1,\ 0 \leq r \leq 3 \]
and observed that for grading reasons the generalized $(\lambda_{\hat\Sigma}, \lambda_{\pt})$-eigenspace is always isomorphic to the generalized $(i\lambda_{\hat\Sigma},-\lambda_{\pt})$-eigenspace.

Let $I_*(S^3_0(K),\hat\Sigma,j)_w$ denote the generalized $(2j,2)$-eigenspace of $(\mu(\hat\Sigma),\mu(\pt))$, for $1-g \leq j \leq g-1$.  Each of these inherits an absolute $\Z/2\Z$ grading from $I_*(S^3_0(K))_w$, and Lim \cite[Corollary~1.2]{lim} proved that
\[ \frac{\Delta_K(t)-1}{t-2+t^{-1}} = \sum_{j=1-g}^{g-1} \chi(I_*(S^3_0(K),\hat\Sigma,j)_w) t^j. \]
In our situation, where $g\leq 2$ and $\Delta_K(t)$ is as above, the left hand side is equal to
\[ a_2t + (a_1+2a_2) + a_2t^{-1}, \]
so the generalized $(2,2)$- and $(0,2)$-eigenspaces of $(\mu(\hat\Sigma),\mu(\pt))$ have dimension
\[ |a_2| + 2k_1 \quad\mathrm{and}\quad |a_1+2a_2|+2k_0 \]
for some nonnegative integers $k_1$ and $k_0$, respectively.  There are four $(2i^r,(-2)^r)$-eigenspaces, all of them isomorphic to each other, and likewise two isomorphic $(0,\pm2)$-eigenspaces, so in total we have
\[ \dim I_*(S^3_0(K))_w = 4|a_2|+2|a_1+2a_2| +4(2k_1+k_0). \]

We now apply Fukaya's connected sum theorem, using the fact that $g(\hat\Sigma) = g(K) \leq 2$, to conclude that there is an isomorphism
\[ I^\#(S^3_0(K),\mu) \otimes H_*(S^4) \cong I_*(S^3_0(K))_w \otimes H_*(S^3) \]
with coefficients in $\C$ (see \cite[\S9.8]{scaduto}).  Thus
\[ \dim I^\#(S^3_0(K),\mu) = \dim I_*(S^3_0(K))_w = 4|a_2|+2|a_1+2a_2| + 4k, \]
where $k = 2k_1+k_0 \geq 0$, exactly as claimed.
\end{proof}

We can now prove Proposition \ref{prop:6_3-surgeries}. We remark that the main new calculation needed for the proof boils down to showing $\dim I^\#(S^3_0(8_8))\geq 14$, which is enabled by Proposition \ref{prop:alexander-lower-bound} (in fact, we prove that equality holds).

\begin{proposition} \label{prop:6_3-surgeries}
We have
\begin{align*}
\dim I^\#(S^3_{-1}(6_3)) &= 7, &
\dim I^\#(S^3_{-1}(8_6)) &= 11, &
\dim I^\#(S^3_{-1}(8_8)) &= 13.
\end{align*}
\end{proposition}

\begin{proof}
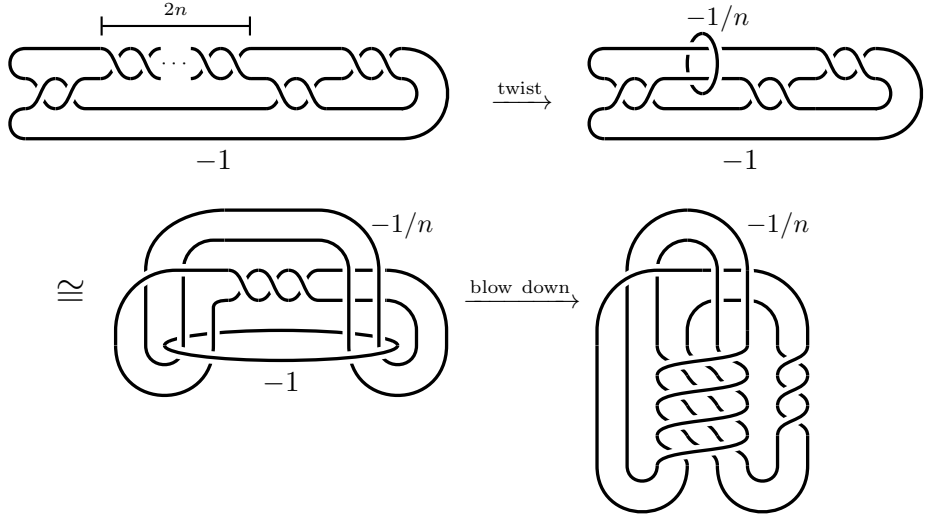
\begin{figure}
\begin{tikzpicture}
\begin{scope}
\draw[link] (1.2,0) to[out=180,in=180] ++(0,0.4) ++(0,0.4) to[out=180,in=180] ++(0,0.4) to[out=0,in=180] ++(1,0);
\draw[link,looseness=1] (1.2,0.8) to[out=0,in=180] ++(0.4,-0.4) ++(0,0.4) to[out=0,in=180] ++(0.4,-0.4) to[out=0,in=180] ++(2.4,0);
\draw[link,looseness=1] (1.2,0.4) to[out=0,in=180] ++(0.4,0.4) ++(0,-0.4) to[out=0,in=180] ++(0.4,0.4) to[out=0,in=180] ++(0.2,0);
\draw[link,looseness=1] (2.2,0.8) to[out=0,in=180] ++(0.4,0.4) ++(0,-0.4) to[out=0,in=180] ++(0.4,0.4) ++(0.4,-0.4) to[out=0,in=180] ++(0.4,0.4) ++(0,-0.4) to[out=0,in=180] ++(0.4,0.4) ++(0,-0.4)  to[out=0,in=180] ++(0.2,0);
\draw[link,looseness=1] (2.2,1.2) to[out=0,in=180] ++(0.4,-0.4) ++(0,0.4) to[out=0,in=180] ++(0.4,-0.4) ++(0.4,0.4) to[out=0,in=180] ++(0.4,-0.4) ++(0,0.4) to[out=0,in=180] ++(0.4,-0.4) ++(0,0.4)  to[out=0,in=180] ++(1.2,0);
\node at (3.2,1) {\tiny $\cdots$};
\draw[link,looseness=1] (4.4,0.4) to[out=0,in=180] ++(0.4,0.4) ++(0,-0.4) to[out=0,in=180] ++(0.4,0.4) to[out=0,in=180] ++(0.2,0);
\draw[link,looseness=1] (4.4,0.8) to[out=0,in=180] ++(0.4,-0.4) ++(0,0.4) to[out=0,in=180] ++(0.4,-0.4) to[out=0,in=180] ++(0.2,0);
\draw[link,looseness=1] (5.4,0.8) to[out=0,in=180] ++(0.4,0.4) ++(0,-0.4) to[out=0,in=180] ++(0.4,0.4);
\draw[link,looseness=1] (5.4,1.2) to[out=0,in=180] ++(0.4,-0.4) ++(0,0.4) to[out=0,in=180] ++(0.4,-0.4);
\draw[link] (6.2,0.8) to[out=0,in=0] ++(0,-0.4) to[out=180,in=0] (6.2,0.4) to[out=180,in=0] (5.2,0.4);
\draw[link] (6.2,1.2) to[out=0,in=0] ++(0,-1.2) to[out=180,in=0] (1.2,0);
\node at (3.7,-0.3) {$-1$};
\draw[|-|] (2.2,1.5) to node[midway,above] {\tiny $2n$} ++(2,0);
\end{scope}

\node at (7.8,0.6) {\small $\xrightarrow{\mathrm{twist}}$};

\begin{scope}[xshift=7.7cm]
\draw[link] (2.5,0.6) arc (270:90:0.2 and 0.4);
\draw[link] (1.2,0) to[out=180,in=180] ++(0,0.4) ++(0,0.4) to[out=180,in=180] ++(0,0.4) to[out=0,in=180] ++(2.8,0);
\draw[link,looseness=1] (1.2,0.8) to[out=0,in=180] ++(0.4,-0.4) ++(0,0.4) to[out=0,in=180] ++(0.4,-0.4) to[out=0,in=180] ++(1,0);
\draw[link,looseness=1] (1.2,0.4) to[out=0,in=180] ++(0.4,0.4) ++(0,-0.4) to[out=0,in=180] ++(0.4,0.4) to[out=0,in=180] ++(1,0);
\draw[link] (2.5,0.6) arc (270:450:0.2 and 0.4);
\draw[link,looseness=1] (3,0.4) to[out=0,in=180] ++(0.4,0.4) ++(0,-0.4) to[out=0,in=180] ++(0.4,0.4) to[out=0,in=180] ++(0.2,0);
\draw[link,looseness=1] (3,0.8) to[out=0,in=180] ++(0.4,-0.4) ++(0,0.4) to[out=0,in=180] ++(0.4,-0.4) to[out=0,in=180] ++(0.2,0);
\draw[link,looseness=1] (4,0.8) to[out=0,in=180] ++(0.4,0.4) ++(0,-0.4) to[out=0,in=180] ++(0.4,0.4);
\draw[link,looseness=1] (4,1.2) to[out=0,in=180] ++(0.4,-0.4) ++(0,0.4) to[out=0,in=180] ++(0.4,-0.4);
\draw[link] (4.8,0.8) to[out=0,in=0] ++(0,-0.4) to[out=180,in=0] (4.8,0.4) to[out=180,in=0] (4,0.4);
\draw[link] (4.8,1.2) to[out=0,in=0] ++(0,-1.2) to[out=180,in=0] (1.2,0);
\node at (3,-0.3) {$-1$};
\node at (2.7,1.6) {\small $-1/n$};
\end{scope}

\begin{scope}[xshift=3.2cm,yshift=-2.75cm]
\draw[link] (-0.15,0) arc (180:0:1.55 and 0.2);
\draw[link] (0.65,1.4) to[out=180,in=0] (0.5,1.4) to[out=180,in=90,looseness=1] (0.1,1) to[out=270,in=90] ++(0,-1) to[out=270,in=270] (-0.4,0) to[out=90,in=270] ++(0,1) to[out=90,in=180,looseness=1] (0.65,1.8);
\draw[link] (0.7,0.6) to[out=180,in=90,looseness=1] (0.5,0.5) to[out=270,in=90] ++(0,-0.5) to[out=270,in=270] (-0.8,0);
\draw[link] (-0.8,0) to[out=90,in=270] ++(0,0.2) to[out=90,in=180,looseness=1] ++(0.8,0.8) to[out=0,in=180] (0.7,1);
\draw[link,looseness=1] (0.7,0.6) to[out=0,in=180] ++(0.4,0.4) ++(0,-0.4) to[out=0,in=180] ++(0.4,0.4) ++(0,-0.4) to[out=0,in=180] ++(0.4,0.4) to[out=0,in=180] ++(0.9,0);
\draw[link,looseness=1] (0.7,1) to[out=0,in=180] ++(0.4,-0.4) ++(0,0.4) to[out=0,in=180] ++(0.4,-0.4) ++(0,0.4) to[out=0,in=180] ++(0.4,-0.4) to[out=0,in=180] ++(0.9,0);
\draw[link] (0.65,1.4) to[out=0,in=180] (1.9,1.4) to[out=0,in=90,looseness=1] ++(0.4,-0.4) to[out=270,in=90] ++(0,-1) to[out=270,in=270] ++(1.3,0) to[out=90,in=270] ++(0,0.2) to[out=90,in=0,looseness=1] ++(-0.8,0.8);
\draw[link] (0.65,1.8) to[out=0,in=180] (1.9,1.8) to[out=0,in=90,looseness=1] ++(0.8,-0.8) to[out=270,in=90] ++(0,-1) to[out=270,in=270] ++(0.5,0) to[out=90,in=270] ++(0,0.2) to[out=90,in=0,looseness=1] ++(-0.4,0.4);
\draw[link] (-0.15,0) arc (180:360:1.55 and 0.2);
\node at (1.4,-0.5) {$-1$};
\node at (3,1.6) {\small $-1/n$};
\node at (-1.4,0.7) {\Large $\cong$};
\end{scope}

\begin{scope}[xshift=9.5cm,yshift=-2.75cm]
\draw[link] (0.5,1.4) to[out=180,in=90,looseness=1] ++(-0.4,-0.4) to[out=270,in=90] ++(0,-1) ++(0,-1.6) to[out=270,in=270] ++(-0.4,0) to[out=90,in=270] ++(0,2.6) to[out=90,in=180,looseness=1] (0.5,1.8);
\draw[link] (0.9,0.6) to[out=180,in=90,looseness=1] ++(-0.4,-0.4) to[out=270,in=90] ++(0,-0.2) ++(0,-1.6) to[out=270,in=270] ++(-1.2,0) to[out=90,in=270] ++(0,1.6);
\draw[link] (-0.7,0) to[out=90,in=270] ++(0,0.2) to[out=90,in=180,looseness=1] ++(0.8,0.8) to[out=0,in=180] (0.7,1);
\draw[link] (0.9,0.6) to[out=0,in=180] ++(0.4,0) (0.7,1) to[out=0,in=180] ++(0.6,0);
\draw[link] (0.5,1.4) to[out=0,in=90,looseness=1] ++(0.4,-0.4) to[out=270,in=90] ++(0,-1) ++(0,-1.6) to[out=270,in=270] ++(1.2,0) to[out=90,in=270] ++(0,0.4) ++(0,1.2) to[out=90,in=270] ++(0,0.2) to[out=90,in=0,looseness=1] ++(-0.8,0.8);
\draw[link] (0.5,1.8) to[out=0,in=90,looseness=1] ++(0.8,-0.8) to[out=270,in=90] ++(0,-1) ++(0,-1.6) to[out=270,in=270] ++(0.4,0) to[out=90,in=270] ++(0,0.4) ++(0,1.2) to[out=90,in=270] ++(0,0.2) to[out=90,in=0,looseness=1] ++(-0.4,0.4);
\draw[link,looseness=1] (2.1,-1.2) to[out=90,in=270] ++(-0.4,0.4) ++(0.4,0) to[out=90,in=270] ++(-0.4,0.4) ++(0.4,0) to[out=90,in=270] ++(-0.4,0.4);
\draw[link,looseness=1] (1.7,-1.2) to[out=90,in=270] ++(0.4,0.4) ++(-0.4,0) to[out=90,in=270] ++(0.4,0.4) ++(-0.4,0) to[out=90,in=270] ++(0.4,0.4);
\draw[link] (1.3,1) to[out=270,in=90] ++(0,-0.8); 
\draw[link,looseness=1] (0.1,0) to[out=270,in=120] ++(0.4,-0.4) ++(0,0.4) to[out=270,in=120] ++(0.4,-0.4) ++(0,0.4) to[out=270,in=90] ++(0.4,-0.4);
\draw[link,looseness=1] (0.1,-0.4) to[out=270,in=120] ++(0.4,-0.4) ++(0,0.4) to[out=300,in=120] ++(0.4,-0.4) ++(0,0.4) to[out=300,in=90] ++(0.4,-0.4);
\draw[link,looseness=1] (0.1,-0.8) to[out=270,in=120] ++(0.4,-0.4) ++(0,0.4) to[out=300,in=120] ++(0.4,-0.4) ++(0,0.4) to[out=300,in=90] ++(0.4,-0.4);
\draw[link,looseness=1] (0.1,-1.2) to[out=270,in=90] ++(0.4,-0.4) ++(0,0.4) to[out=300,in=90] ++(0.4,-0.4) ++(0,0.4) to[out=300,in=90] ++(0.4,-0.4);
\draw[link,looseness=0.5] (1.3,0) to[out=270,in=90] ++(-1.2,-0.4) ++(1.2,0) to[out=270,in=90] ++(-1.2,-0.4) ++(1.2,0) to[out=270,in=90] ++(-1.2,-0.4) ++(1.2,0) to[out=270,in=90] ++(-1.2,-0.4);
\node at (1.7,1.6) {\small $-1/n$};
\node at (-1.7,0.7) {\small $\xrightarrow{\mathrm{blow\ down}}$};
\end{scope}
\end{tikzpicture}
\caption{$(-1)$-surgery on the two-bridge knot $K(-2,2n,2,2)$ with continued fraction $[-2,2n,2,2] = -\frac{12n-1}{6n-2}$ is homeomorphic to $-\frac{1}{n}$-surgery on the knot $P$, which is the pretzel knot $P(5,5,-3)$.}
\label{fig:13n3596-surgeries}
\end{figure} The knots $6_2,6_3,8_6,8_8$ are equal to   the two-bridge knot $K(-2,2n,2,2)$ of Figure~\ref{fig:13n3596-surgeries}, for $n=1,-1,2,-2$, respectively.
The figure shows that $-1$-surgery on $K(-2,2n,2,2)$ is homeomorphic to $-\frac{1}{n}$-surgery on a fixed knot $P$ for all integers $n$.  (In fact $P$ is isotopic to the pretzel knot $P(5,5,-3)$, though we do not need this.) Therefore,
\begin{align*}
S^3_{-1}(6_2) &\cong S^3_{-1}(P), &
S^3_{-1}(6_3) &\cong S^3_1(P), \\
S^3_{-1}(8_6) &\cong S^3_{-1/2}(P), &
S^3_{-1}(8_8) &\cong S^3_{1/2}(P).
\end{align*}
As listed in Table \ref{table:isharp-values}, we have  $\cinvt(6_2)=-1$ and $r_0(6_2)=5$, where the latter followed from  Proposition~\ref{prop:isharp-5_2-surgeries}. Theorem \ref{thm:rational-surgeries}, combined with the first identification above, then says that\[ \dim I^\#(S^3_{-1}(P)) = \dim I^\#(S^3_{-1}(6_2)) = 5. \]
From the surgery exact triangles 
\begin{align*}
\dots \to I^\#(S^3) &\to I^\#(S^3_{-1}(P)) \to I^\#(S^3_{0}(P)) \to \dots\\
\dots \to I^\#(S^3) &\to I^\#(S^3_{0}(P)) \to I^\#(S^3_{1}(P)) \to \dots,
\end{align*} we therefore have the inequalities
\begin{align}
\dim I^\#(S^3_0(P)) &\leq 6, &
\dim I^\#(S^3_1(P)) &\leq 7. \label{eq:13n3596-0-1}
\end{align}

We next note that $8_8$  has Seifert genus $2$ and Alexander polynomial \[\Delta_{8_8}(t) = 2t^2-6t+9-6t^{-1}+2t^{-2}.\]  Moreover, $8_8$ is slice and thus W-shaped by Theorem~\ref{thm:conc-invt}, so we have  \begin{align*}
\dim I^\#(S^3_0(8_8)) &= \dim I^\#(S^3_0(8_8),\mu) + 2\\
&=(12+4k)+2 \\
&= 14+4k
\end{align*} for some $k\geq 0$,
where the second equality  is by Proposition~\ref{prop:alexander-lower-bound}. Since $I^\#(S^3_{-1}(8_8))$ fits into a surgery  exact triangle with $I^\#(S^3))$ and $I^\#(S^3_0(8_8))$, we  have \begin{align*}\dim I^\#(S^3_{1/2}(P))&=\dim I^\#(S^3_{-1}(8_8))\\
& \geq \dim I^\#(S^3_0(8_8)) - 1\\
&=13+4k.
\end{align*} 
The fact that $I^\#(S^3_{1/2}(P))$ fits into a surgery  exact triangle with $I^\#(S^3_0(P))$ and $I^\#(S^3_1(P))$, combined with the inequalities \eqref{eq:13n3596-0-1}, then forces $k=0$ and these inequalities to be strict. That is, we have
\begin{align*}
\dim I^\#(S^3_{0}(P)) &= 6, &
\dim I^\#(S^3_{1}(P)) &= 7, &
\dim I^\#(S^3_{1/2}(P)) &= 13.
\end{align*}
We conclude that
\begin{align*}\dim I^\#(S^3_{-1}(6_3)) &= \dim I^\#(S^3_1(P)) = 7\\
\dim I^\#(S^3_{-1}(8_8)) &= \dim I^\#(S^3_{1/2}(P)) = 13, \end{align*}
as claimed. It remains to complete the computation for $8_6$.

For this, first note that $\cinvt(P)\leq -1 $ since $\dim I^\#(S^3_{-1}(P)) < \dim I^\#(S^3_{0}(P))$. By Theorem \ref{thm:rational-surgeries} and the calculations above, we therefore have \[\dim I^\#(S^3_{0}(P)) = r_0(P) + |{-}\cinvt(P)|=r_0(P)-\cinvt(P)=6.\] Applying Theorem \ref{thm:rational-surgeries} again, we then find that 
 \begin{align*}\dim I^\#(S^3_{-1}(8_6))&=\dim I^\#(S^3_{-1/2}(P)) \\
 &= 2r_0(P) + |{-}1-2\cinvt(P)|\\
 &=2r_0(P)-2\cinvt(P)-1\\
 &=2\cdot 6-1=11,\end{align*} as claimed.
\end{proof}

\section{On the Khovanov to framed instanton spectral sequence} \label{sec:ss}
Recall from Theorem~\ref{thm:kh-to-i-ss} that there is a spectral sequence
\begin{equation*} \label{eq:ss}
\Khoddr(K) \Rightarrow I^\#(-\dcover(K)),
\end{equation*}
whose $E^2$ page is the reduced odd Khovanov homology of $K$. We will study the corresponding spectral sequence with coefficients in $\C$,
\begin{equation} \label{eq:ssC}
\Khoddr(K;\C) \Rightarrow I^\#(-\dcover(K);\C).
\end{equation}
By Proposition~\ref{prop:ss-thin}, this spectral sequence collapses at  the $E^2$ page  whenever $\Khoddr(K;\C)$ is thin. In \cite[\S5]{orsz}, Ozsv\'ath, Rasmussen, and Szab\'o computed $\Khoddr(K)$ for prime knots of at most 10 crossings and found that there are exactly seven knots $K$ for which $\Khoddr(K;\Q)$ and hence $\Khoddr(K;\C)$ is not thin:
\begin{equation}\label{eq:sevenknots} 10_{124}, 10_{139}, 10_{145}, 10_{152}, 10_{153}, 10_{154}, 10_{161}. \end{equation}
In this section, we  show for each of these $K$ apart from $10_{152}$ that \begin{equation}\label{eq:dimI}\dim \Khoddr(K;\C)>\dim I^\#(\dcover(K);\C),\end{equation} and therefore that the spectral sequence in \eqref{eq:ssC} does not collapse at  the $E^2$ page,  proving Theorem \ref{thm:main-ss}. To compute the framed instanton homology groups in \eqref{eq:dimI}, we identify  each of the relevant branched double covers as surgery on a knot, and then apply our  machinery for computing $I^\#$ of surgeries. Our results are summarized in Table \ref{table:ss-table}.

\begin{table}
\bgroup
\def\arraystretch{1.2}
\begin{tabular}{c|cc|cc}
$K$ & $\det(K)$ & $\dim \Khoddr(K;\C)$ & $\dcover(K)$ & $\dim I^\#(-\dcover(K);\C)$ \\ \hline
$10_{124}$ & $1$ & $3$ & $S^3_{-1}(3_1)$ & $1$ \\
$10_{139}$ & $3$ & $7$ & $S^3_{-3}(4_1)$ & $5$ \\
$10_{145}$ & $3$ & $7$ & $S^3_{-3}(5_2)$ & $5$ \\
$10_{152}$ & $11$ & $15$ && \\
$10_{153}$ & $1$ & $9$ & $S^3_1(P(7,3,-3))$ & $5$ \\
$10_{154}$ & $13$ & $17$ & $S^3_{13}(10_{139})$ & $13$ or $15$ \\
$10_{161}$ & $5$ & $9$ & $S^3_5(4_1)$ & $7$ \\
\end{tabular}
\egroup
\caption{The $10$-crossing knots with non-thin $\Khoddr(K;\C)$.}
\label{table:ss-table}
\end{table}

The rest of this subsection is devoted to justifying the entries in Table \ref{table:ss-table}. The reduced odd Khovanov calculations are from \cite[\S5]{orsz}. The other entries are from Lemmas \ref{lem:dcover-124}, \ref{lem:dcover-139-161}, \ref{lem:dcover-145}, \ref{lem:dcover-153}, and \ref{lem:dcover-154}. All calculations of framed instanton homology below are with coefficients in $\C$, as usual.

\begin{lemma} \label{lem:dcover-124}
We have $\dcover(10_{124}) \cong \dcover(T_{3,5})\cong S^3_{-1}(3_1)$, which implies that \[\dim I^\#(\dcover(10_{124})) = 1.\]
\end{lemma}

\begin{proof}
This is simply the fact that $10_{124} = T_{3,5}$, whose branched double cover is well-known to be the Poincar{\'e} homology sphere $\Sigma(2,3,5) \cong S^3_{-1}(3_1).$ The framed instanton homology calculation then  follows from Lemma \ref{lem:t2q-surgery}.
\end{proof}

\begin{lemma} \label{lem:dcover-139-161}
We have $\dcover(10_{139}) \cong S^3_{-3}(4_1)$ and $\dcover(10_{161}) \cong S^3_5(4_1),$ which implies that \[\dim I^\#(\dcover(10_{139})) = 5\quad \textrm{and} \quad \dim I^\#(\dcover(10_{161}))=7.\]
\end{lemma}

\begin{proof}
Montesinos and Whitten \cite[pp. 426--427]{montesinos-whitten} gave for each rational $\frac{p}{q}$ an explicit diagram of a knot $K_{p/q}$ such that $ \dcover(K_{p/q})=S^3_{p/q}(4_1)$.  For $\frac{p}{q} = -n$, where $n$ is a positive integer, we take the diagram on the left side of \cite[Figure~17]{montesinos-whitten} and add $n$ left-handed half twists in the region labeled ``$\gamma/\delta$''.  The case $n=3$ is a knot shown explicitly in \cite[Figure~19]{montesinos-whitten}, and SnapPy recognizes this as $10_{139}$; for the chirality, note that both the knot shown in \cite{montesinos-whitten} and $10_{139}$ in \cite{rolfsen} are positive knots.

Similarly, SnapPy identifies the knot $K_{-5}$ as $10_{161}$ or its mirror, and changing one crossing among the left-handed half twists from negative to positive produces $K_{-3} = 10_{139}$, which is a positive knot of genus 4, so that
\[ 0 \leq \chominvt(10_{139}) - \chominvt(K_{-5}) \leq 1 \qquad\Longleftrightarrow\qquad 3 \leq \chominvt(K_{-5}) \leq 4, \]
by Propositions~\ref{prop:tau-change} and \ref{prop:slice-bennequin-positive}.  Similarly, the diagram of $10_{161}$ in \cite{rolfsen} has a single positive crossing, and changing this to negative produces $\mirror{10_{139}}$, so we have
\[ 0 \leq \chominvt(10_{161}) - \chominvt(\mirror{10_{139}}) \leq 1 \qquad\Longleftrightarrow\qquad -4 \leq \chominvt(10_{161}) \leq -3. \]
It follows that $K_{-5}$ must be $\mirror{10_{161}}$.  Now since $4_1$ is amphichiral, we observe that
\[ \dcover(10_{161}) \cong -\dcover(\mirror{10_{161}}) \cong -S^3_{-5}(4_1) \cong S^3_5(\mirror{4_1}) \cong S^3_5(4_1).\] The claimed framed instanton homology calculations then follow from Theorem \ref{thm:rational-surgeries} together with the fact that $\cinvt(4_1)=0$ and $r_0(4_1) =2$, as in Table \ref{table:isharp-values}. 
\end{proof}

\begin{lemma} \label{lem:dcover-145}
We have $\dcover(10_{145})\cong \dcover(T_{3,10}) \cong S^3_{-3}(5_2)$, which implies that \[\dim I^\#(\dcover(10_{145}))=5.\]
\end{lemma}

\begin{proof}
The knot $5_2$ is strongly invertible, so its complement is the branched double cover of a tangle as shown in Figure~\ref{fig:5_2-surgery}, again following Montesinos \cite{montesinos}.  Filling in a rational tangle and taking branched double covers then gives a Dehn surgery on $5_2$, though it is not immediately clear which tangles correspond to which slopes because we have not been careful to follow the framing through the isotopy.
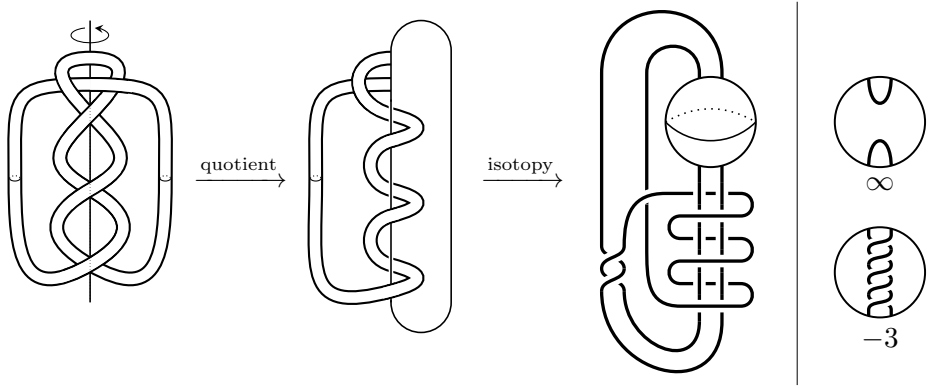
\begin{figure}
\tikzset{tunnel/.style = { black, double = white, thick, double distance = 4pt, looseness=1.75 }}
\tikzset{tunnel two/.style = { white, line width=4pt, looseness=1.75 }}
\newcommand{\AxisRotator}{\tikz [x=0.25cm,y=0.10cm,line width=.2ex,-stealth] \draw[thin] (0,0) arc (105:435:1 and 1);}
\begin{tikzpicture}

\begin{scope} 
\draw[semithick] (0,-1.25) -- (0,2.5);
\draw[tunnel] (-1,0.4) to[out=270,in=210] (0,-0.75) to[out=30,in=330] (0,0.25) to[out=150,in=225] (0,1.25) to[out=45,in=0,looseness=2] (0,2) to[out=180,in=135,looseness=2] (0,1.25) to[out=315,in=30] (0,0.25) to[out=210, in=150] (0,-0.75) to[out=330,in=270] (1,0.4) to[out=90, in=0] (0,1.65) to[out=180,in=90] (-1,0.4);
\draw[thin] ($(-1,0.4)+(-2pt,0)$) arc(180:360:2pt and 1pt);
\draw[thin,densely dotted] ($(-1,0.4)+(-2pt,0)$) arc(180:0:2pt and 1pt);
\draw[thin] ($(1,0.4)+(-2pt,0)$) arc(180:360:2pt and 1pt);
\draw[thin,densely dotted] ($(1,0.4)+(-2pt,0)$) arc(180:0:2pt and 1pt);
\begin{scope} 
\clip (-0.75,1.45) rectangle (0.55,1.85);
\begin{scope}
\clip (-0.75,1.5) rectangle (0.5,1.8);
\draw[tunnel] (0,2) to[out=180,in=135,looseness=2] (0,1.25);
\draw[tunnel] (1,0) to[out=90, in=0] (0,1.65);
\end{scope}
\draw[tunnel two] (0,2) to[out=180,in=135,looseness=2] (0,1.25);
\draw[tunnel two] (1,0) to[out=90, in=0] (0,1.65);
\end{scope}
\begin{scope} 
\clip (-0.4,-1) rectangle (0.4,1.5);
\begin{scope} \clip (-0.35,-0.95) rectangle (0.35,1.45);
\draw[tunnel] (-1,0) to[out=270,in=210] (0,-0.75) to[out=30,in=330] (0,0.25) to[out=150,in=225] (0,1.25) (0,1.25) to[out=45,in=0,looseness=2] (0,2);
\end{scope}
\draw[tunnel two] (-1,0) to[out=270,in=210] (0,-0.75) to[out=30,in=330] (0,0.25) to[out=150,in=225] (0,1.25) (0,1.25) to[out=45,in=0,looseness=2] (0,2);
\end{scope}
\begin{scope} 
\clip (-0.4,0) rectangle (0.4,0.5);
\begin{scope} \clip (-0.35,0.05) rectangle (0.35,0.45);
\draw[tunnel] (0,1.25) to[out=315,in=30] (0,0.25) to[out=210, in=150] (0,-0.75);
\end{scope}
\draw[tunnel two] (0,1.25) to[out=315,in=30] (0,0.25) to[out=210, in=150] (0,-0.75);
\end{scope}
\draw[very thin,densely dotted] (0,-1.2) -- (0,1.5);
\node at (0,2.3) {\AxisRotator};
\end{scope}

\node at (2,0.5) {\small$\xrightarrow{\mathrm{quotient}}$};

\begin{scope}[xshift=4cm] 
\draw[tunnel] (0,2) to[out=180,in=165,looseness=2] (0,1.25) to[out=345,in=15,looseness=3.5] (0,0.9) to[out=195,in=165,looseness=1.5] (0,0.25) to[out=345,in=15,looseness=3.5] (0,-0.1) to[out=195,in=165,looseness=1.5] (0,-0.75) to[out=345,in=15,looseness=3.5] (0,-1.1) to[out=195,in=270] (-1,0.4) to[out=90,in=180] (0,1.65);
\draw[thin] ($(-1,0.4)+(-2pt,0)$) arc(180:360:2pt and 1pt);
\draw[thin,densely dotted] ($(-1,0.4)+(-2pt,0)$) arc(180:0:2pt and 1pt);
\begin{scope} 
\clip (-0.75,1.25) rectangle (-0.15,1.85);
\begin{scope} \clip (-0.7,1.3) rectangle (-0.2,1.8);
\draw[tunnel] (0,2) to[out=180,in=165,looseness=2] (0,1.25);
\end{scope}
\draw[tunnel two] (0,2) to[out=180,in=165,looseness=2] (0,1.25);
\end{scope}
\draw[semithick] (0,2.1) -- (0,-1.25) arc (180:360:0.4) -- (0.8,2.1) arc(0:180:0.4); 
\draw[line width=1.75pt,white] (0,1.42) -- (0,-1.22);
\draw[semithick] (0,1.45) -- (0,-1.25);
\begin{scope} 
\clip (-0.1, 0.65) rectangle (0.1,1.1);
\begin{scope} \clip (-0.05,0.7) rectangle (0.05,1.05);
\draw[tunnel] (0,1.25) to[out=345,in=15,looseness=3.5] (0,0.9) to[out=195,in=165,looseness=1.5] (0,0.25);
\end{scope}
\draw[tunnel two] (0,1.25) to[out=345,in=15,looseness=3.5] (0,0.9) to[out=195,in=165,looseness=1.5] (0,0.25);
\end{scope}
\begin{scope} 
\clip (-0.1, -0.35) rectangle (0.1,0.1);
\begin{scope} \clip (-0.05,-0.3) rectangle (0.05,0.05);
\draw[tunnel] (0,0.25) to[out=345,in=15,looseness=3.5] (0,-0.1) to[out=195,in=165,looseness=1.5] (0,-0.75);
\end{scope}
\draw[tunnel two] (0,0.25) to[out=345,in=15,looseness=3.5] (0,-0.1) to[out=195,in=165,looseness=1.5] (0,-0.75);
\end{scope}
\begin{scope} 
\clip (-0.1, -1.35) rectangle (0.1,-0.9);
\begin{scope} \clip (-0.05,-1.3) rectangle (0.05,-0.95);
\draw[tunnel] (0,-0.75) to[out=345,in=15,looseness=3.5] (0,-1.1) to[out=195,in=270] (-1,0);
\end{scope}
\draw[tunnel two] (0,-0.75) to[out=345,in=15,looseness=3.5] (0,-1.1) to[out=195,in=270] (-1,0);
\end{scope}
\end{scope}

\node at (5.75,0.5) {\small$\xrightarrow{\mathrm{isotopy}}$};

\begin{scope}[xshift=8.25cm,yshift=0.75cm] 
\clip (-1.75,-3.1) rectangle (0.65,2); 
\draw[link] (-0.15,-2.25) to[looseness=0] (-0.15,1.05) to[out=90,in=90] (-0.85,1.05) to[looseness=0] (-0.85,-0.7);
\draw[link] (0.15,-2.25) to[looseness=0] (0.15,1.05) to[out=90,in=90] (-1.45,1.05) to[looseness=0] (-1.45,-1.25) to[out=270,in=90,looseness=1] (-1.15,-1.55) to[out=270,in=90,looseness=1] (-1.45,-1.85) to[out=270,in=270] (0.15,-2.25);
\draw[link] (-1.15,-1.85) to[out=90,in=270,looseness=1] (-1.45,-1.55) to[out=90,in=270,looseness=1] (-1.15,-1.25) to[out=90,in=180,looseness=1.25] (-0.6,-0.55) to[looseness=0] (0.4,-0.55) to[out=0,in=0] ++(0,-0.3) to[looseness=0] ++(-0.8,0) to[out=180,in=180] ++(0,-0.3) to[looseness=0] ++(0.8,0) to[out=0,in=0] ++(0,-0.3) to[looseness=0] ++(-0.8,0) to[out=180,in=180] ++(0,-0.3) to[looseness=0] ++(0.8,0) to[out=0,in=0] ++(0,-0.3) to[looseness=0] ++(-0.8,0) to[out=180,in=270,looseness=1.25] (-0.85,-1.25) to[looseness=0] (-0.85,-0.7);
\draw[link] (-1.15,-1.85) to[out=270,in=270] (-0.15,-2.25);
\draw[line width=0.75pt,fill=white] (0,0.4) circle (0.6); 
\begin{scope} \clip (0,0.4) circle (0.6);
\draw[semithick] (0,1) circle (0.8484);
\draw[semithick,dotted] (-0.6,0.4) to[out=30,in=150] (0.6,0.4);
\end{scope}
\begin{scope} \clip(-1.5,-1.55) rectangle (-1.1,-1.85);
\draw[link] (-1.15,-1.55) to[out=270,in=90,looseness=1] (-1.45,-1.85);
\end{scope}
\draw[link,looseness=0] (-0.15,-0.4) to ++(0,-0.3) ++(0,-0.3) to ++(0,-0.3) ++(0,-0.3) to ++(0,-0.3);
\draw[link,looseness=0] (0.15,-0.4) to ++(0,-0.3) ++(0,-0.3) to ++(0,-0.3) ++(0,-0.3) to ++(0,-0.3);
\end{scope}

\coordinate (divider) at (9.4,0);
\draw[very thin] (current bounding box.north -| divider) -- (current bounding box.south -| divider);

\begin{scope}[xshift=10.5cm,yshift=0.75cm]
\begin{scope}
\clip (0,0.4) circle (0.6);
\draw[link,looseness=4] (0.15,1) to[out=270,in=270] (-0.15,1) (0.15,-0.2) to[out=90,in=90] (-0.15,-0.2);
\end{scope}
\draw[line width=0.75pt] (0,0.4) circle (0.6);
\node[below] at (0,-0.2) {$\infty$};

\begin{scope}[link/.append style = { line width = 1.5pt, looseness=1 }]
\clip (0,-1.6) circle (0.6);
\draw[link] (-0.15,-1) to[looseness=0] (-0.15,-1.1) to[out=270,in=90] (0.15,-1.3) to[out=270,in=90] (-0.15,-1.5) to[out=270,in=90] (0.15,-1.7) to[out=270,in=90] (-0.15,-1.9) to[out=270,in=90] (0.15,-2.1) to[looseness=0] (0.15,-2.2);
\draw[link] (0.15,-1) to[looseness=0] (0.15,-1.1) to[out=270,in=90] (-0.15,-1.3) to[out=270,in=90] (0.15,-1.5) to[out=270,in=90] (-0.15,-1.7) to[out=270,in=90] (0.15,-1.9) to[out=270,in=90] (-0.15,-2.1) to[looseness=0] (-0.15,-2.2);
\draw[link] (0.15,-1.3) to[out=270,in=90] (-0.15,-1.5);
\draw[link] (0.15,-1.7) to[out=270,in=90] (-0.15,-1.9);
\end{scope}
\draw[line width=0.75pt] (0,-1.6) circle (0.6);
\node[below] at (0,-2.2) {$-3$};
\end{scope}
\end{tikzpicture}
\caption{Left, constructing a tangle $T$ whose branched double cover is $S^3 \setminus N(5_2)$.  Right, some tangles to plug into $T$ to get knots $K_\infty$ and $K_{-3}$.}
\label{fig:5_2-surgery}
\end{figure}

On the right side of Figure~\ref{fig:5_2-surgery}, we can insert the rational tangle labeled ``$\infty$'' into $T$ to get a knot $K_\infty$ which is unknotted.  Since $\dcover(K_\infty) \cong S^3$, this tangle corresponds to $\infty$-surgery on $5_2$.  If instead we use the ``$-3$'' tangle, consisting of five positive half-twists, then the branched double cover of the resulting knot $K_{-3}$ is an integral surgery on $5_2$, since it has distance $1$ from the $\infty$-surgery.

By inspection we see that $K_{-3}$ is isotopic to the closure of a 3-braid
\[ \beta = (\sigma_1\sigma_2\sigma_2\sigma_1)^3 \sigma_2^7 \sigma_1, \]
and $(\sigma_1\sigma_2\sigma_2\sigma_1)\sigma_2\sigma_2 = \sigma_1\sigma_2(\sigma_1\sigma_2\sigma_1)\sigma_2 = (\sigma_2\sigma_1\sigma_2)(\sigma_1\sigma_2\sigma_1)$ is the central element $(\sigma_2\sigma_1)^3$ of the braid group $B_3$, so it follows that
\[ \beta = ((\sigma_2\sigma_1)^3)^3 \cdot \sigma_2\sigma_1 = (\sigma_2\sigma_1)^{10}. \]
In other words, $K_{-3}$ is the $(3,10)$ torus knot.  If its branched double cover is $k$-surgery on $5_2$, where $k$ is an integer, then we have
\[ |k| = |H_1(\dcover(T_{3,10});\Z)| = \det T_{3,10} = 3, \]
so $\dcover(T_{3,10})$ is either $+3$-surgery or $-3$-surgery on $5_2$.  In fact, $\dcover(T_{3,10})$ is Seifert fibered, and $5_2$ is a two-bridge knot, so Brittenham--Wu \cite{brittenham-wu} determined which surgeries on it are not hyperbolic: in their notation we have $5_2 = [4,2]$, so $-3$-surgery is exceptional while $+3$-surgery is not.  Thus $\dcover(T_{3,10}) \cong S^3_{-3}(5_2)$, as claimed.

The Seifert invariants of the Brieskorn manifold $\dcover(T_{3,10}) \cong \Sigma(2,3,10)$ were computed by Neumann and Raymond \cite[Theorem~2.1]{neumann-raymond} to be 
\[ \Sigma(2,3,10) \cong M(0;(1,\beta_1),(3,\beta_2),(3,\beta_2),(5,\beta_3)) \]
for any integers $\beta_i$ satisfying $15\beta_1 + 10\beta_2 + 3\beta_3 = 1$.  Taking $(\beta_1,\beta_2,\beta_3)=(0,1,-3)$ and deleting the pair $(1,0)$ gives
\[ \Sigma(2,3,10) \cong M(0; (3,1),(3,1),(5,-3)). \]
This is the branched double cover of the Montesinos knot with three rational tangles corresponding to the continued fractions $\frac{3}{1} = [3]$,  $\frac{3}{1} = [3]$, and $-\frac{5}{3} = [-2,-3]$, as shown in Figure~\ref{fig:10_145-montesinos}, and we can identify this knot as $10_{145}$.
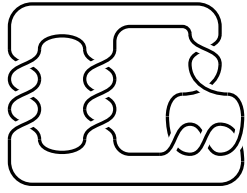
\begin{figure}
\begin{tikzpicture}[link/.append style = {looseness=1}]
\begin{scope} 
\foreach \x in {0,1cm}{
\begin{scope}[xshift=\x]
\draw[link] (0.4,0) to[out=90,in=270] ++(-0.4,0.4) to[out=90,in=270] ++(0.4,0.4) to[out=90,in=270] ++(-0.4,0.4);
\draw[link] (0,0) to[out=90,in=270] ++(0.4,0.4) to[out=90,in=270] ++(-0.4,0.4) to[out=90,in=270] ++(0.4,0.4);
\draw[link] (0,0.4) to[out=90,in=270] (0.4,0.8); 
\end{scope}
}
\draw[link] (0.4,1.2) to[out=90,in=90] (1,1.2) (0.4,0) to[out=270,in=270] (1,0);
\end{scope}
\begin{scope} 
\draw[link] (2.3,0.6) to[out=0,in=270] ++(0.5,0.4) to[out=90,in=270] ++(-0.4,0.4);
\draw[link] (2.9,0.6) to[out=180,in=270] ++(-0.5,0.4) to[out=90,in=270] ++(0.4,0.4);
\draw[link] (2.3,0.6) to[out=180,in=90] ++(-0.2,-0.3);
\draw[link] (2.9,0.6) to[out=0,in=90] ++(0.2,-0.3);
\draw[link] (2.1,0.3) to[out=270,in=180] ++(0.3,-0.5) to[out=0,in=180] ++(0.4,0.4) to[out=0,in=90] ++(0.3,-0.5);
\draw[link] (2,-0.2) to[out=0,in=180] ++(0.4,0.4) to[out=0,in=180] ++(0.4,-0.4) to[out=0,in=270] ++(0.3,0.5);
\draw[link] (2.4,-0.2) to[out=0,in=180] ++(0.4,0.4); 
\draw[link] (2.3,0.6) ++(0.5,0.4) to[out=90,in=270] ++(-0.4,0.4); 
\end{scope}
\draw[link] (1.4,0) to[out=270,in=180] (1.6,-0.2) to[out=0,in=180] (2,-0.2);
\draw[link] (2.4,1.4) to[out=90,in=0] ++(-0.1,0.1) to[out=180,in=0] (1.6,1.5) to[out=180,in=90] (1.4,1.3) to[out=270,in=90] (1.4,1.2);
\draw[link] (3.1,-0.3) to[out=270,in=0] ++(-0.3,-0.3) to[out=180,in=0] (0.3,-0.6) to[out=180,in=270] (0,-0.3) to[out=90,in=270] (0,0);
\draw[link] (2.8,1.4) to[out=90,in=270] ++(0,0.1) to[out=90,in=0] ++(-0.3,0.3) to[out=180,in=0] (0.3,1.8) to[out=180,in=90] (0,1.5) to[out=270,in=90] (0,1.2);
\end{tikzpicture}
\caption{A Montesinos knot with branched double cover $S^3_{-3}(5_2)$.}
\label{fig:10_145-montesinos}
\end{figure}

To sort out the chirality, we note that the knot $K$ in Figure~\ref{fig:10_145-montesinos} has only two positive crossings, namely the crossings in the $[-2]$ twist region of the $-\frac{5}{3}$ tangle, and changing one of them from positive to negative produces the negative knot $\mirror{9_{10}}$.  Since $9_{10}$ is positive of genus $2$, we have $\chominvt(\mirror{9_{10}}) = -2$, hence
\[ 0 \leq \chominvt(K) - \chominvt(\mirror{9_{10}}) \leq 1 \quad\Longleftrightarrow\quad -2 \leq \chominvt(K) \leq -1. \]
Meanwhile, the diagram of $10_{145}$ in \cite{rolfsen} has a single positive crossing, and changing it to negative produces the diagram $K_-$ of a negative knot which must be nontrivial, since $10_{145}$ has unknotting number 2.  Thus $\chominvt(K_-) \leq -1$ and we have
\[ \chominvt(10_{145}) \leq 1 + \chominvt(K_-) \leq 0. \]
But then $\chominvt(\mirror{10_{145}}) \geq 0 > \chominvt(K)$, so we conclude that $K$ must be isotopic to $10_{145}$. The  framed instanton homology calculation then follows from Theorem \ref{thm:rational-surgeries} together with the fact that $\cinvt(5_2)=-1$ and $r_0(5_2) =3$, as in Table \ref{table:isharp-values}. 
\end{proof}

\begin{lemma} \label{lem:dcover-153}
We have $\dcover(10_{153})\cong S^3_1(P(7,3,-3))$, which implies that \[\dim I^\#(\dcover(10_{153})) = 5.\]
\end{lemma}

\begin{proof}
In Figure~\ref{fig:dcover-10_153}, we repeat the procedure described in \S\ref{ssec:pretzels} for surgeries on periodic knots to find a homeomorphism
\[ S^3_1(P(7,3,-3)) \cong \dcover(10_{153}). \]
The steps labeled ``$\cong$'' in this figure carry out an isotopy which turns the diagram resulting from this procedure into the $10_{153}$ diagram shown in \cite{rolfsen}. The  framed instanton homology calculation then follows from Theorem \ref{thm:rational-surgeries} together with the fact that $\cinvt(P(7,3,-3))=0$ and $r_0(P(7,3,-3)) =4$, per Theorem \ref{thm:main-pretzels}.
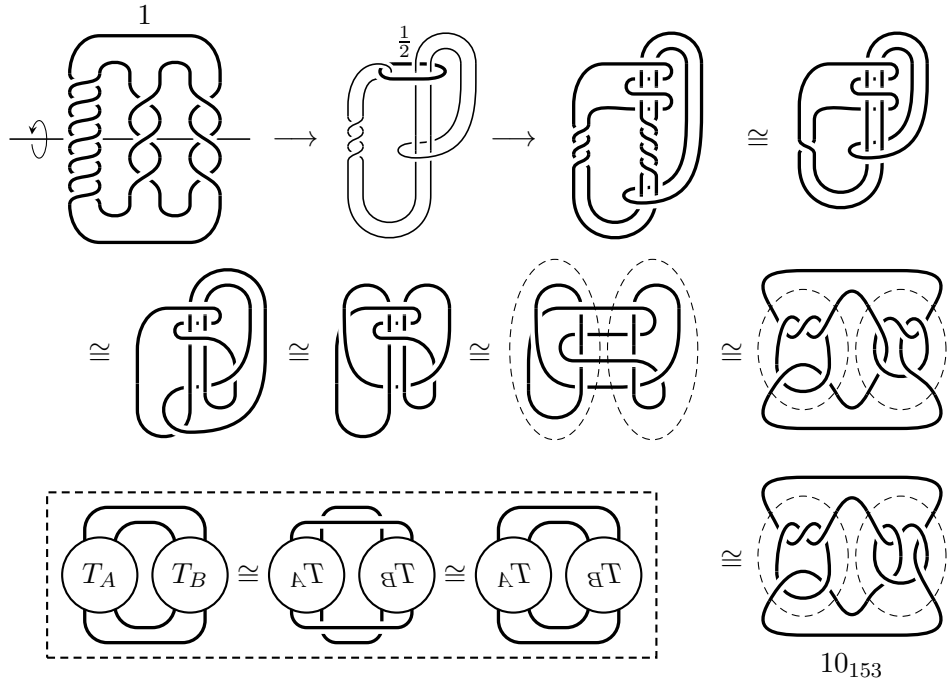
\begin{figure}
\tikzset{twistregion/.style={draw, fill=white, thick, minimum width=0.6cm}}
\tikzset{thinlink/.style = { white, double = black, line width = 1.2pt, double distance = 0.6pt, looseness=1.75 }}
\newcommand{\AxisRotator}{\tikz [x=0.10cm,y=0.25cm,line width=.2ex,-stealth] \draw[thin] (0,0) arc (-165:165:1 and 1);}
\begin{tikzpicture}
\begin{scope} 
\draw[semithick] (-1.6,0.75) -- (1.6,0.75); 
\draw[link,looseness=1] (0,1.5) to[out=270,in=90] ++(0.4,-0.5) ++(-0.4,0) to[out=270,in=90] ++(0.4,-0.5) ++(-0.4,0) to[out=270,in=90] ++(0.4,-0.5) -- ++(0,-0.125) to[out=270,in=270,looseness=1.25] ++(0.4,0) -- ++(0,0.125) to[out=90,in=270] ++(0.4,0.5) ++(-0.4,0) to[out=90,in=270] ++(0.4,0.5) ++(-0.4,0) to[out=90,in=270] ++(0.4,0.5);
\draw[link,looseness=1] (1.2,0) to[out=90,in=270] ++(-0.4,0.5) ++(0.4,0) to[out=90,in=270] ++(-0.4,0.5) ++(0.4,0) to[out=90,in=270] ++(-0.4,0.5) -- ++(0,0.125) to[out=90,in=90,looseness=1.25] ++(-0.4,0) -- ++(0,-0.125) to[out=270,in=90] ++(-0.4,-0.5) ++(0.4,0) to[out=270,in=90] ++(-0.4,-0.5) ++(0.4,0) to[out=270,in=90] ++(-0.4,-0.5);
\draw[link] (0,1.5) -- ++(0,0.125) to[out=90,in=90,looseness=1.25] ++(-0.4,0) ++(0,-1.75) to[out=270,in=270,looseness=1.25] ++(0.4,0) -- ++(0,0.125);
\draw[link,looseness=1] (1.2,1.5) -- ++(0,0.125) to[out=90,in=0] ++(-0.4,0.5) to[out=180,in=0] node[midway,above,black] {$1$} ++(-1.2,0) to[out=180,in=90] ++(-0.4,-0.5) -- ++(0,-1.75) to[out=270,in=180] ++(0.4,-0.5) to[out=0,in=180] ++(1.2,0) to[out=0,in=270] ++(0.4,0.5) -- ++(0,0.125);
\draw[link,looseness=1] (-0.4,-0.125) to[out=90,in=270] ++(-0.4,0.25) ++(0.4,0) to[out=90,in=270] ++(-0.4,0.25) ++(0.4,0) to[out=90,in=270] ++(-0.4,0.25) ++(0.4,0) to[out=90,in=270] ++(-0.4,0.25) ++(0.4,0) to[out=90,in=270] ++(-0.4,0.25) ++(0.4,0) to[out=90,in=270] ++(-0.4,0.25) ++(0.4,0) to[out=90,in=270] ++(-0.4,0.25);
\draw[link,looseness=1] (-0.8,-0.125) to[out=90,in=270] ++(0.4,0.25) ++(-0.4,0) to[out=90,in=270] ++(0.4,0.25) ++(-0.4,0) to[out=90,in=270] ++(0.4,0.25) ++(-0.4,0) to[out=90,in=270] ++(0.4,0.25) ++(-0.4,0) to[out=90,in=270] ++(0.4,0.25) ++(-0.4,0) to[out=90,in=270] ++(0.4,0.25) ++(-0.4,0) to[out=90,in=270] ++(0.4,0.25);
\node at (-1.2,0.75) {\AxisRotator};
\end{scope}

\node at (2.2,0.725) {$\longrightarrow$};

\begin{scope}[xshift=3cm]
\draw[thinlink] (-0.1,1.05) to[out=270,in=90,looseness=1] ++(0.2,-0.2) ++(-0.2,0) to[out=270,in=90,looseness=1] ++(0.2,-0.2) ++(-0.2,0) to[out=270,in=90,looseness=1] ++(0.2,-0.2) -- (0.1,0) to[out=270,in=270] (0.8,0) -- ++(0,1.75) to[out=90,in=90] ++(0.8,0);
\draw[thinlink] (0.1,1.05) to[out=270,in=90,looseness=1] ++(-0.2,-0.2) ++(0.2,0) to[out=270,in=90,looseness=1] ++(-0.2,-0.2) ++(0.2,0) to[out=270,in=90,looseness=1] ++(-0.2,-0.2) -- (-0.1,0) to[out=270,in=270] (1.0,0) -- ++(0,1.75) to[out=90,in=90] ++(0.4,0);
\draw[thinlink] (-0.1,1.05) -- (-0.1,1.2) to[out=90,in=180,looseness=1] (0.3,1.75) to[out=0,in=0,looseness=3] ++(0,-0.2) to[out=180,in=90,looseness=1] (0.1,1.2) -- (0.1,1.05);
\draw[thinlink,looseness=1.25] (1.6,1.75) -- ++(0,-0.25) to[out=270,in=0] ++(-0.8,-1) to[out=180,in=180,looseness=4] ++(0,0.2) to[out=0,in=270] ++(0.6,0.8) -- ++(0,0.25);
\draw[thinlink] (0.8,0.6) -- ++(0,0.2) ++(0.2,-0.2) -- ++(0,0.2); 
\draw[link,looseness=1.25] (0.5,1.75) -- ++(0.4,0) to[out=0,in=0,looseness=5] ++(0,-0.2) -- ++(-0.4,0) to[out=180,in=180,looseness=3] ++(0,0.2);
\draw[thinlink] (0.8,1.65) -- ++(0,0.1) to[out=90,in=90] ++(0.8,0); 
\draw[thinlink] (1.0,1.65) -- ++(0,0.1) to[out=90,in=90] ++(0.4,0);
\begin{scope} 
\clip (0,1.65) rectangle (0.5,1.85);
\draw[thinlink] (-0.1,1.2) to[out=90,in=180,looseness=1] (0.3,1.75) to[out=0,in=0,looseness=3] ++(0,-0.2);
\end{scope}
\node at (0.65,2.05) {\small$\tfrac{1}{2}$};
\end{scope}

\node at (5.1,0.725) {$\longrightarrow$};

\begin{scope}[xshift=6cm]
\draw[link] (-0.1,0.95) to[out=270,in=90,looseness=1] ++(0.2,-0.2) ++(-0.2,0) to[out=270,in=90,looseness=1] ++(0.2,-0.2) ++(-0.2,0) to[out=270,in=90,looseness=1] ++(0.2,-0.2) -- (0.1,0) to[out=270,in=270] (0.8,0) -- ++(0,1.75) to[out=90,in=90] ++(0.8,0);
\draw[link] (0.1,0.95) to[out=270,in=90,looseness=1] ++(-0.2,-0.2) ++(0.2,0) to[out=270,in=90,looseness=1] ++(-0.2,-0.2) ++(0.2,0) to[out=270,in=90,looseness=1] ++(-0.2,-0.2) -- (-0.1,0) to[out=270,in=270] (1.0,0) -- ++(0,1.75) to[out=90,in=90] ++(0.4,0);
\draw[link] (-0.1,0.95) -- (-0.1,1.2) to[out=90,in=180,looseness=1] (0.3,1.75) -- ++(0.8,0) to[out=0,in=0] ++(0,-0.2) -- ++(-0.4,0) to[out=180,in=180] ++(0,-0.2) -- ++(0.4,0) to[out=0,in=0] ++(0,-0.2) -- ++(-0.4,0) to[out=180,in=90,looseness=1] (0.1,1) -- (0.1,0.95);
\draw[link,looseness=1.25] (1.6,1.75) -- ++(0,-0.65) to[out=270,in=0] ++(-0.8,-1.2) to[out=180,in=180,looseness=4] ++(0,0.2) to[out=0,in=270] ++(0.6,1.0) -- ++(0,0.65);
\draw[link,looseness=1] (0.8,0.01) -- ++(0,0.24) to[out=90,in=270] ++(0.2,0.2) ++(-0.2,0) to[out=90,in=270] ++(0.2,0.2) ++(-0.2,0) to[out=90,in=270] ++(0.2,0.2) ++(-0.2,0) to[out=90,in=270] ++(0.2,0.2) -- ++(0,0.2) ++(0,0.2) -- ++(0,0.2); 
\draw[link,looseness=1] (1.0,0.03) -- ++(0,0.22) to[out=90,in=270] ++(-0.2,0.2) ++(0.2,0) to[out=90,in=270] ++(-0.2,0.2) ++(0.2,0) to[out=90,in=270] ++(-0.2,0.2) ++(0.2,0) to[out=90,in=270] ++(-0.2,0.2) -- ++(0,0.2) ++(0,0.2) -- ++(0,0.2);
\end{scope}

\node at (8.35,0.75) {$\cong$};

\begin{scope}[xshift=9cm]
\draw[link] (0.1,0.75) to[out=270,in=90,looseness=1] ++(-0.2,-0.2) -- (-0.1,0.4) to[out=270,in=270] (1.0,0.4) -- ++(0,1.35) to[out=90,in=90] ++(0.4,0);
\draw[link] (-0.1,0.75) to[out=270,in=90,looseness=1] ++(0.2,-0.2) -- (0.1,0.4) to[out=270,in=270] (0.8,0.4) -- ++(0,1.35) to[out=90,in=90] ++(0.8,0);
\draw[link] (-0.1,0.75) -- (-0.1,1.2) to[out=90,in=180,looseness=1] (0.3,1.75) -- ++(0.8,0) to[out=0,in=0] ++(0,-0.2) -- ++(-0.4,0) to[out=180,in=180] ++(0,-0.2) -- ++(0.4,0) to[out=0,in=0] ++(0,-0.2) -- ++(-0.4,0) to[out=180,in=90,looseness=1] (0.1,1) -- (0.1,0.75);
\draw[link,looseness=1.25] (1.6,1.75) -- ++(0,-0.25) to[out=270,in=0] ++(-0.8,-1) to[out=180,in=180,looseness=4] ++(0,0.2) to[out=0,in=270] ++(0.6,0.8) -- ++(0,0.25);
\draw[link,looseness=1] (0.8,0.61) -- ++(0,0.44) -- ++(0,0.2) ++(0,0.2) -- ++(0,0.2); 
\draw[link,looseness=1] (1.0,0.63) -- ++(0,0.42) -- ++(0,0.2) ++(0,0.2) -- ++(0,0.2);
\end{scope}

\node at (-0.4,-2.1) {$\cong$};

\begin{scope}[xshift=0cm,yshift=-3.25cm]
\draw[link] (1.0,1.65) -- ++(0,0.1) to[out=90,in=90] ++(0.6,0);
\draw[link] (0.1,0.75) -- (0.1,0.4) to[out=270,in=270] (0.8,0.4) -- ++(0,1.35) to[out=90,in=90] ++(1.0,0);
\draw[link] (0.1,0.75) -- (0.1,1.2) to[out=90,in=180,looseness=1] (0.5,1.75) -- ++(0.6,0) to[out=0,in=0] ++(0,-0.2) -- ++(-0.4,0) to[out=180,in=180] ++(0,-0.2) -- ++(0.3,0) to[out=0,in=90,looseness=1] ++(0.4,-0.4) -- ++(0,-0.35) to[out=270,in=270,looseness=1.5] ++(-0.4,0) -- ++(0,0.03);
\draw[link,looseness=1.25] (1.8,1.75) -- ++(0,-0.65) to[out=270,in=0] ++(-1.0,-1) to[out=180,in=180,looseness=2] ++(0,0.6) to[out=0,in=270] ++(0.8,0.8) -- ++(0,0.25);
\draw[link,looseness=1] (0.8,0.61) -- ++(0,0.44) -- ++(0,0.2) ++(0,0.2) -- ++(0,0.2); 
\draw[link,looseness=1] (1.0,0.63) -- ++(0,0.42) -- ++(0,0.2) ++(0,0.2) -- ++(0,0.2);
\end{scope}

\node at (2.25,-2.1) {$\cong$};

\begin{scope}[xshift=2.65cm,yshift=-3.25cm]
\draw[link] (1.0,1.65) -- ++(0,0.1) to[out=90,in=90] ++(0.6,0);
\draw[link] (0.1,0.75) -- (0.1,0.4) to[out=270,in=270] (0.8,0.4) -- ++(0,1.35) to[out=90,in=90] ++(-0.6,0);
\draw[link] (0.1,0.75) -- (0.1,1.2) to[out=90,in=180,looseness=1] (0.5,1.75) -- ++(0.6,0) to[out=0,in=0] ++(0,-0.2) -- ++(-0.4,0) to[out=180,in=180] ++(0,-0.2) -- ++(0.3,0) to[out=0,in=90,looseness=1] ++(0.4,-0.4) -- ++(0,-0.35) to[out=270,in=270,looseness=1.5] ++(-0.4,0) -- ++(0,0.03);
\draw[link,looseness=1.25] (0.2,1.75) to[out=270,in=180] (0.8,0.7) to[out=0,in=270] ++(0.8,0.8) -- ++(0,0.25);
\draw[link,looseness=1] (0.8,0.61) -- ++(0,0.44) -- ++(0,0.2) ++(0,0.2) -- ++(0,0.2); 
\draw[link,looseness=1] (1.0,0.63) -- ++(0,0.42) -- ++(0,0.2) ++(0,0.2) -- ++(0,0.2);
\end{scope}

\node at (4.65,-2.1) {$\cong$};

\begin{scope}[xshift=5.7cm,yshift=-3.25cm]
\draw[link] (1.0,0.65) -- (1.0,1.65) -- ++(0,0.1) to[out=90,in=90] ++(0.6,0);
\draw[link] (-0.4,0.75) -- (-0.4,0.65) to[out=270,in=270] (0.3,0.65) -- ++(0,1.1) to[out=90,in=90] ++(-0.6,0);
\draw[link] (-0.4,0.75) -- (-0.4,1) to[out=90,in=180,looseness=1] (0,1.75) -- ++(1.1,0) to[out=0,in=0] ++(0,-0.35) -- ++(-0.9,0) to[out=180,in=180] ++(0,-0.35) -- ++(0.8,0) to[out=0,in=90,looseness=1] ++(0.4,-0.4) -- ++(0,-0.05) to[out=270,in=270,looseness=1.5] ++(-0.4,0) -- ++(0,0.03);
\draw[link,looseness=1.25] (-0.3,1.75) to[out=270,in=180] (0.3,0.7) to[out=0,in=180] (0.8,0.7) to[out=0,in=270] ++(0.8,0.8) -- ++(0,0.25);
\draw[link] (-0.4,0.65) to[out=270,in=270] (0.3,0.65) -- ++(0,0.3) ++(0,0.3) -- ++(0,0.3); 
\draw[link,looseness=1] (1.0,0.6) -- ++(0,0.35) ++(0,0.3) -- ++(0,0.3);
\draw[thin,densely dashed] (-0.05,1.2) ellipse (0.6 and 1.2);
\draw[thin,densely dashed] (1.25,1.2) ellipse (0.6 and 1.2);
\end{scope}

\node at (8,-2.1) {$\cong$};

\begin{scope}[xshift=9cm,yshift=-3.25cm]
\draw[link] (0.4,1.8) to[out=60,in=120] (0.8,1.8) (0.4,0.6) to[out=-60,in=-120] (0.8,0.6);
\draw[link] (-0.5,1.8) to[out=120,in=180] (-0.05,2.25) -- (1.25,2.25) to[out=0,in=60] (1.7,1.8);
\draw[link] (-0.5,0.6) to[out=-120,in=180] (-0.05,0.15) -- (1.25,0.15) to[out=0,in=-60] (1.7,0.6);
\draw[link] (-0.5,0.6) to[out=60,in=120] (0.4,0.6);
\draw[link] (-0.5,1.8) to[out=-60,in=240] (-0.05,1.5) to[out=60,in=90] (0.3,1) to[out=270,in=270] (-0.4,1);
\draw[link] (-0.4,1) to[out=90,in=120] (-0.05,1.5) to[out=-60,in=240] (0.4,1.8);
\begin{scope} 
\clip (-0.5,0.6) rectangle (-0.05,1.2);
\draw[link] (-0.5,0.6) to[out=60,in=120] (0.4,0.6);
\end{scope}
\begin{scope} 
\clip (-0.2,1.3) rectangle (0.1,1.7);
\draw[link] (-0.5,1.8) to[out=-60,in=240] (-0.05,1.5) to[out=60,in=90] (0.3,1);
\end{scope}
\draw[link] (1.7,1.8) to[out=240,in=300] (1.25,1.5) to[out=120,in=90] (0.9,1.2);
\draw[link] (0.8,1.8) to[out=300,in=60,looseness=1.5] (0.8,0.6);
\draw[link] (0.9,1.2) to[out=270,in=270] (1.6,1.2) to[out=90,in=90,looseness=4] (1.25,1.2);
\draw[link] (1.25,1.2) to[out=270,in=120,looseness=1.25] (1.7,0.6);
\begin{scope} 
\clip (1.1,1.25) rectangle (1.4,1.65);
\draw[link] (1.7,1.8) to[out=240,in=300] (1.25,1.5) to[out=120,in=90] (0.9,1.2);
\end{scope}
\draw[thin,densely dashed] (-0.05,1.2) ellipse (0.6 and 0.8);
\draw[thin,densely dashed] (1.25,1.2) ellipse (0.6 and 0.8);
\end{scope}

\node at (8,-4.85) {$\cong$};
\begin{scope}[xshift=9cm,yshift=-6cm]
\draw[link] (0.4,1.8) to[out=60,in=120] (0.8,1.8) (0.4,0.6) to[out=-60,in=-120] (0.8,0.6);
\draw[link] (-0.5,1.8) to[out=120,in=180] (-0.05,2.25) -- (1.25,2.25) to[out=0,in=60] (1.7,1.8);
\draw[link] (-0.5,0.6) to[out=-120,in=180] (-0.05,0.15) -- (1.25,0.15) to[out=0,in=-60] (1.7,0.6);
\draw[link] (-0.5,0.6) to[out=60,in=120] (0.4,0.6);
\draw[link] (-0.5,1.8) to[out=-60,in=240] (-0.05,1.5) to[out=60,in=90] (0.3,1) to[out=270,in=270] (-0.4,1);
\draw[link] (-0.4,1) to[out=90,in=120] (-0.05,1.5) to[out=-60,in=240] (0.4,1.8);
\begin{scope} 
\clip (-0.5,0.6) rectangle (-0.05,1.2);
\draw[link] (-0.5,0.6) to[out=60,in=120] (0.4,0.6);
\end{scope}
\begin{scope} 
\clip (-0.2,1.3) rectangle (0.1,1.7);
\draw[link] (-0.5,1.8) to[out=-60,in=240] (-0.05,1.5) to[out=60,in=90] (0.3,1);
\end{scope}
\draw[link] (1.25,1.2) to[out=270,in=60,looseness=1.25] (0.8,0.6);
\draw[link] (1.6,1.2) to[out=270,in=270] (0.9,1.2) to[out=90,in=90,looseness=4] (1.25,1.2);
\draw[link] (1.7,1.8) to[out=240,in=120,looseness=1.5] (1.7,0.6);
\draw[link] (0.8,1.8) to[out=300,in=240] (1.25,1.5) to[out=60,in=90] (1.6,1.2);
\begin{scope} 
\clip (1.1,1.25) rectangle (1.4,1.65);
\draw[link] (0.9,1.2) to[out=90,in=90,looseness=4] (1.25,1.2);
\end{scope}
\draw[thin,densely dashed] (-0.05,1.2) ellipse (0.6 and 0.8);
\draw[thin,densely dashed] (1.25,1.2) ellipse (0.6 and 0.8);
\node at (0.6,-0.25) {$10_{153}$};
\end{scope}

\begin{scope}[xshift=0cm,yshift=-6.25cm,local bounding box=tanglescope]
\begin{scope}
\draw[link,rounded corners=3mm] (-0.6,2.1) rectangle (1.0,0.3);
\draw[link,rounded corners=3mm] (-0.2,1.9) rectangle (0.6,0.5);
\node[circle,draw,fill=white,inner sep=0pt,minimum size=1cm] at (-0.4,1.2) {$T_A$};
\node[circle,draw,fill=white,inner sep=0pt,minimum size=1cm] at (0.8,1.2) {$T_B$};
\end{scope}

\begin{scope}[xshift=2.75cm]
\draw[link,rounded corners=2mm] (-0.2,2.1) rectangle (0.6,0.3);
\draw[link,rounded corners=2mm] (-0.6,1.9) rectangle (1.0,0.5);
\node[circle,draw,fill=white,inner sep=0pt,minimum size=1cm,xscale=-1] at (-0.4,1.2) {$T_A$};
\node[circle,draw,fill=white,inner sep=0pt,minimum size=1cm,xscale=-1] at (0.8,1.2) {$T_B$};
\end{scope}

\begin{scope}[xshift=5.5cm]
\draw[link,rounded corners=3mm] (-0.6,2.1) rectangle (1.0,0.3);
\draw[link,rounded corners=3mm] (-0.2,1.9) rectangle (0.6,0.5);
\node[circle,draw,fill=white,inner sep=0pt,minimum size=1cm,xscale=-1] at (-0.4,1.2) {$T_A$};
\node[circle,draw,fill=white,inner sep=0pt,minimum size=1cm,xscale=-1] at (0.8,1.2) {$T_B$};
\end{scope}

\node at (1.575,1.2) {$\cong$};
\node at (4.325,1.2) {$\cong$};
\draw[densely dashed] ($(tanglescope.north east)+(0.2,0.2)$) rectangle ($(tanglescope.south west)+(-0.2,-0.2)$);
\end{scope}
\end{tikzpicture}
\caption{Realizing $S^3_1(P(7,3,-3))$ as the branched double cover of $10_{153}$, following the same process as in Figure~\ref{fig:P-2k3-dcover}.  In the last step we flip both tangles as illustrated in the inset box.}
\label{fig:dcover-10_153}
\end{figure}
\end{proof}

\begin{lemma} \label{lem:dcover-154}
We have $ \dcover(10_{154})\cong S^3_{13}(10_{139})$, and \[\dim I^\#(\dcover(10_{154}))\] is either $13$ or $15$.
\end{lemma}

\begin{proof}
The knot $10_{139}$ is strongly invertible, so we use the procedure described in \S\ref{ssec:pretzels} to find a knot $K$ whose branched double cover is $S^3_{13}(10_{139})$.  This is illustrated in Figure~\ref{fig:dcover-10_154}, and the resulting diagram of $K$ has 25 crossings, but SnapPy easily recognizes $K$ as either $10_{154}$ or its mirror.
\begin{figure}
\tikzset{twistregion/.style={draw, fill=white, thick, minimum width=0.6cm}}
\tikzset{thinlink/.style = { white, double = black, line width = 1.2pt, double distance = 0.6pt, looseness=1.75 }}
\newcommand{\AxisRotator}{\tikz [x=0.25cm,y=0.10cm,line width=.2ex,-stealth] \draw[thin] (0,0) arc (105:435:1 and 1);}
\begin{tikzpicture}

\begin{scope} 
\clip (-1.4,-1.5) rectangle (1.4,2.15); 
\draw[semithick] (0,-0.9) -- (0,2.1); 
\draw[link] (0,1.5) to[out=180,in=90] (-0.8,0.5) to[out=270,in=210] (0,-0.5) to[out=30,in=270,looseness=1] (0.5,0.5) to[out=90,in=45] (0,1);
\draw[link] (0,1.5) to[out=0,in=90] (0.8,0.5) to[out=270,in=-30] (0,-0.5) to[out=150,in=270,looseness=1] (-0.5,0.5) to[out=90,in=135] (0,1);
\draw[link] (0,1) to[out=315,in=180] (0.75,0.75) to[out=0,in=0,looseness=2.5] (0,0);
\draw[link] (0,1) to[out=225,in=0] (-0.75,0.75) to[out=180,in=180,looseness=2.5] (0,0);
\begin{scope} \clip (0.25,1) rectangle (1,0.5);
\draw[link] (0,1.5) to[out=0,in=90] (0.8,0.5);
\draw[link] (0.5,0.5) to[out=90,in=45] (0,1);
\end{scope}
\begin{scope} \clip (-0.3,0.7) rectangle (0.3,1.3);
\draw[link] (-0.5,0.5) to[out=90,in=135] (0,1) to[out=315,in=180] (0.75,0.75);
\end{scope}
\begin{scope} \clip (-1,-0.25) rectangle (-0.25,0.4);
\draw[link] (-0.8,0.5) to[out=270,in=210] (0,-0.5);
\draw[link] (0,-0.5) to[out=150,in=270,looseness=1] (-0.5,0.5);
\end{scope}
\draw[semithick] (0,-0.25) -- (0,0.25) (0,1.25) -- (0,1.75);
\node at (0,1.85) {\AxisRotator};
\node at (0,-1.25) {$10_{139}$};
\end{scope}

\node at (2,0.5) {$\xrightarrow{\textrm{\tiny quotient}}$};

\begin{scope}[xshift=4cm] 
\draw[semithick,looseness=2] (0,-0.8) -- (0,1.6) to[out=90,in=90] (0.5,1.6) -- (0.5,-0.8) to[out=270,in=270] (0,-0.8); 
\draw[link] (0,1.5) to[out=180,in=90] (-0.8,0.5) to[out=270,in=180,looseness=1] (-0.4,-0.7);
\draw[link] (-0.4,-0.7) to[out=0,in=180,looseness=1] (0,-0.7) to[out=0,in=0,looseness=2.5] (0,-0.3);
\draw[link] (0,-0.3) to[out=180,in=270,looseness=1] (-0.5,0.5);
\draw[link] (-0.5,0.5) to[out=90,in=180,looseness=1] (0,1.15) to[out=0,in=0,looseness=2.5] (0,0.75);
\draw[link] (0,0.75) to[out=180,in=0,looseness=0.5] (-0.3,0.75) -- (-0.75,0.75) to[out=180,in=180,looseness=2.5] (0,0);
\begin{scope} \clip (-1,-0.25) rectangle (-0.25,0.4);
\draw[link] (-0.8,0.5) to[out=270,in=180,looseness=1] (-0.4,-0.7);
\draw[link] (0,-0.3) to[out=180,in=270,looseness=1] (-0.5,0.5);
\end{scope}
\draw[semithick] (0,-0.15) -- (0,0.25) (0,1.25) -- (0,1.6);
\draw[thinlink] (0,-0.8) -- (0,-0.5) (0,0.55) -- (0,0.95);
\end{scope}

\node at (5.2,0.5) {$\xrightarrow{\textrm{\tiny isotope}}$};

\begin{scope}[xshift=7.7cm] 
\tikzset{thindouble/.style={line width=2.4pt,white,preaction={preaction={draw,line width=5pt,white},draw,line width=3.6pt,black},looseness=1.75}}
\draw[semithick,looseness=2] (0,-0.8) -- (0,0) (0,0.4) -- (0,1.3) (0,1.7) to[out=90,in=90] (0.5,1.7) -- (0.5,-0.8) to[out=270,in=270] (0,-0.8); 
\draw[semithick] ($(-0.2,1.5)+(0,1.5pt)$) to[out=0,in=270,looseness=1] (0,1.7);
\draw[semithick] ($(-0.2,1.5)+(0,-1.5pt)$) to[out=0,in=90,looseness=1] (0,1.3);
\draw[thindouble] (-0.2,1.5) to[out=180,in=90] (-1.2,0.5) -- (-1.2,-0.8);
\draw[thindouble] (-1.2,-1.2) to[out=270,in=270,looseness=2] (-0.4,-1.2);
\draw[thindouble] (-0.4,-1.2) to[out=90,in=180,looseness=1] (0,-0.7) to[out=0,in=0,looseness=2.5] (0,-0.3);
\draw[thindouble] (0,-0.3) to[out=180,in=270,looseness=1] (-0.75,0.5);
\draw[thindouble] (-0.75,0.5) to[out=90,in=180,looseness=1] (0,1.15) to[out=0,in=0,looseness=2.5] (0,0.75);
\draw[thindouble] (0,0.75) to[out=180,in=0,looseness=0.5] (-0.45,0.75) -- (-1.125,0.75) to[out=180,in=180,looseness=3] (-0.2,0.2);
\draw[semithick] ($(-0.2,0.2)+(0,1.5pt)$) to[out=0,in=270,looseness=1] (0,0.4);
\draw[semithick] ($(-0.2,0.2)+(0,-1.5pt)$) to[out=0,in=90,looseness=1] (0,0);
\begin{scope} \clip (-1.5,-0.25) rectangle (0,0.5);
\draw[thindouble] (-1.2,0.5) -- (-1.2,0);
\draw[thindouble] (0,-0.3) to[out=180,in=270,looseness=1] (-0.75,0.5);
\end{scope}
\draw[thinlink] (0,-0.8) -- (0,-0.5) (0,0.55) -- (0,0.95);
\node[twistregion,semithick] at (-1.2,-0.2) {\tiny$-12$};
\draw[line width=0.75pt,fill=white] (-1.2,-1.05) circle (0.4);
\begin{scope} \clip (-1.2,-1.05) circle (0.4);
\draw[semithick] (-1.2,-0.65) circle (0.5656);
\draw[semithick,dotted] (-1.6,-1.05) to[out=30,in=150] (-0.8,-1.05);
\end{scope}
\end{scope}

\node at (8.65,0.5) {$\xrightarrow{\textrm{\tiny fill}}$};

\begin{scope}[xshift=10.8cm] 
\tikzset{thickdouble/.style={line width=1.75pt,white,preaction={preaction={draw,line width=6pt,white},draw,line width=4.25pt,black},looseness=1.75}}
\draw[link,looseness=2] (0,-0.8) -- (0,0) (0,0.4) -- (0,1.3) (0,1.7) to[out=90,in=90] (0.5,1.7) -- (0.5,-0.8) to[out=270,in=270] (0,-0.8); 
\draw[link] ($(-0.2,1.5)+(0,1.5pt)$) to[out=0,in=270,looseness=1] (0,1.7);
\draw[link] ($(-0.2,1.5)+(0,-1.5pt)$) to[out=0,in=90,looseness=1] (0,1.3);
\draw[thickdouble] (-0.2,1.5) to[out=180,in=90] (-1.2,0.5) -- (-1.2,0.1);
\draw[thickdouble] (-1.2,-0.3) to[out=270,in=180,looseness=1] (-0.6,-0.7);
\draw[thickdouble] (-0.6,-0.7) to[out=0,in=180,looseness=1] (0,-0.7) to[out=0,in=0,looseness=2.5] (0,-0.3);
\draw[thickdouble] (0,-0.3) to[out=180,in=270,looseness=1] (-0.75,0.5);
\draw[thickdouble] (-0.75,0.5) to[out=90,in=180,looseness=1] (0,1.15) to[out=0,in=0,looseness=2.5] (0,0.75);
\draw[thickdouble] (0,0.75) to[out=180,in=0,looseness=0.5] (-0.45,0.75) -- (-1.125,0.75) to[out=180,in=180,looseness=3] (-0.2,0.2);
\draw[link] ($(-0.2,0.2)+(0,1.5pt)$) to[out=0,in=270,looseness=1] (0,0.4);
\draw[link] ($(-0.2,0.2)+(0,-1.5pt)$) to[out=0,in=90,looseness=1] (0,0);
\begin{scope} \clip (-1.5,-0.25) rectangle (0,0.5);
\draw[thickdouble] (-1.2,0.5) -- (-1.2,0);
\draw[thickdouble] (0,-0.3) to[out=180,in=270,looseness=1] (-0.75,0.5);
\end{scope}
\draw[link] (0,-0.8) -- (0,-0.5) (0,0.55) -- (0,0.95);
\draw[link] ($(-1.2,0.1)+(-1.5pt,0)$) to[out=270,in=90,looseness=1] ++(3pt,-0.4);
\draw[link] ($(-1.2,0.1)+(1.5pt,0)$) to[out=270,in=90,looseness=1] ++(-3pt,-0.4);
\node at (-0.6,-1.25) {$K$};
\draw[very thin] (0,0.75) circle (0.2);
\draw[very thin] (0,-0.7) circle (0.2);
\end{scope}

\end{tikzpicture}
\caption{Realizing the complement of $10_{139}$ as the branched double cover of a tangle, and then filling in a rational tangle to get a knot $K$ whose branched double cover is $S^3_{13}(10_{139})$.  The box labeled ``$-12$'' contains $12$ left-handed half-twists.  Changing the four positive crossings circled at right to negative ones turns $K$ into an unknot.}
\label{fig:dcover-10_154}
\end{figure}
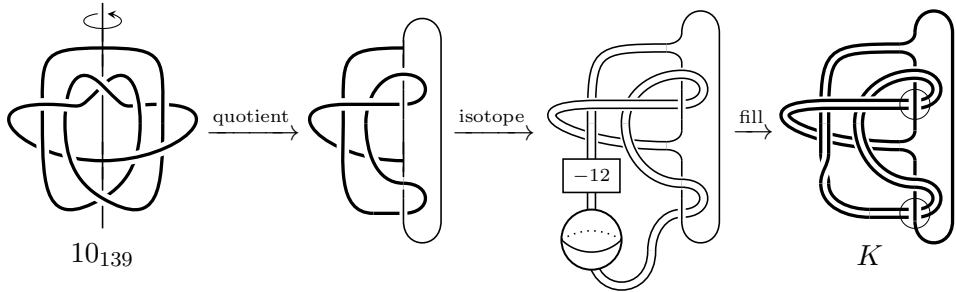
To determine the chirality of $K$, we note that $10_{154}$ is a positive knot of genus $3$, so that $\chominvt(10_{154}) = 3$ by Proposition~\ref{prop:slice-bennequin-positive} and then $\chominvt(\mirror{10_{154}})=-3$.  But we can change four positive crossings of $K$ to negative ones to unknot it, as shown in Figure~\ref{fig:dcover-10_154}, so repeated application of Proposition~\ref{prop:tau-change} says that $\chominvt(K) \geq \chominvt(U) = 0$, hence $K=10_{154}$.

In order to compute $\dim I^\#(\dcover(10_{154}))$, we first observe that $\cinvt(10_{139})=7$ by Proposition~\ref{prop:slice-bennequin-positive}, since $10_{139}$ is a positive knot of genus $4$.  Thus
\[ \dim I^\#(S^3_{13}(10_{139})) = \dim I^\#(S^3_{14}(10_{139})) - 1, \]
so it will suffice to show that $\dim I^\#(S^3_{14}(10_{139}))$ is either $14$ or $16$.

We can identify $S^3_{14}(10_{139})$ as a filling of the hyperbolic census manifold \texttt{m011}, and verify that another filling of \texttt{m011} is in fact $5$-surgery on the figure eight up to orientation:
\begin{verbatim}
In[1]: Manifold('10_139(14,1)').identify()
Out[1]: [m011(1,3)]
In[2]: Manifold('4_1(5,1)').is_isometric_to(Manifold('m011(1,2)'))
Out[2]: True
\end{verbatim}
(In order to verify that \texttt{10\char`_139(14,1)} really means $14$-surgery on $10_{139}$ rather than on its mirror, we compute the normal boundary slopes of \texttt{10\char`_139}:
\begin{verbatim}
In[3]: Manifold('10_139').normal_boundary_slopes()
Out[3]: [(12, 1), (13, 1), (14, 1), (18, 1), (20, 1)]
\end{verbatim}
The $13$-surgery is $\pm\dcover(10_{154})$, in which a Conway sphere lifts to an incompressible torus, so $13$ is a boundary slope for $10_{139}$ but not for its mirror, and hence \texttt{10\char`_139} is indeed $10_{139}$.) The manifold \texttt{m011(0,1)} is not hyperbolic, but it has fundamental group $\Z/9\Z$, and Regina identifies it as the lens space $L(9,2)$.

Since the filling slopes $(0,1)$, $(1,2)$, and $(1,3)$ have pairwise distance 1, there is a surgery exact triangle relating $I^\#(L(9,2))$, $I^\#(S^3_5(4_1))$, and $I^\#(S^3_{14}(10_{139}))$. We also know that \begin{align*}
\dim I^\#(L(9,2)) &=9, & \dim I^\#(S^3_5(4_1)) &=7, 
\end{align*} which implies that 
\[\dim I^\#(S^3_{14}(10_{139}))\leq 16,\] by exactness. Since the latter group has Euler characteristic 14 by \eqref{eq:framed-chi},  we have
\[ \dim I^\#(S^3_{14}(10_{139})) = 14 \mathrm{\ or\ } 16, \] which implies that \[ \dim I^\#(\dcover(10_{154})) = \dim I^\#(S^3_{13}(10_{139})) = 13 \mathrm{\ or\ } 15, \]
as claimed.
\end{proof}

\begin{remark}
We expect that $\dim I^\#(S^3_{13}(10_{139}))$ is $15$ rather than $13$, because $10_{139}$ is not an L-space knot in Heegaard Floer homology.
\end{remark}

\begin{proof}[Proof of Theorem \ref{thm:main-ss}]
As noted above, this follows immediately from the fact that \[\dim \Khoddr(K;\C)>\dim I^\#(\dcover(K);\C)\] for all of the knots $K$ in \eqref{eq:sevenknots} besides $10_{152}$, as indicated in Table \ref{table:ss-table}.
\end{proof}

\section{Hyperbolic manifolds with small volume}
\label{sec:small-hyperbolic}

In this section, we compute the framed instanton homology for many of the smallest manifolds in the Hodgson--Weeks census of closed hyperbolic 3-manifolds, proving Theorem \ref{thm:hyperbolic}, which states that these computations are as summarized in Table~\ref{table:hw-census}. Our proof follows from Propositions \ref{prop:hw-surgeries}, \ref{prop:hw-branched}, and \ref{prop:census-2-7-16}.

\begin{proposition} \label{prop:hw-surgeries}
Twelve of the first twenty Hodgson--Weeks census manifolds are surgeries on knots in $S^3$, and their framed instanton homology is as specified in Table~\ref{table:hw-table1}.
\end{proposition}
\begin{table}
\bgroup
\def\arraystretch{1.25}
\begin{tabular}{c|ccc|cc}
\# & $M$ & $N$ & Manifold & $|H_1|$ & $\dim I^\#$ \\ \hline
1 & \texttt{m003(-2,3)} & \texttt{4\char`_1(5,1)} & $S^3_5(4_1)$ & $5$ & $7$ \\
4 & \texttt{m004(6,1)} & \texttt{4\char`_1(6,1)} & $S^3_6(4_1)$ & $6$ & $8$ \\
5 & \texttt{m004(1,2)} & \texttt{4\char`_1(1,2)} & $S^3_{1/2}(4_1)$ & $1$ & $5$ \\
6 & \texttt{m009(4,1)} & \texttt{5\char`_2(6,1)} & $S^3_{6}(\mirror{5_2})$ & $6$ & $8$ \\
9 & \texttt{m004(3,2)} & \texttt{4\char`_1(3,2)} & $S^3_{3/2}(4_1)$ & $3$ & $7$ \\
10 & \texttt{m004(7,1)} & \texttt{4\char`_1(7,1)} & $S^3_7(4_1)$ & $7$ & $9$ \\
11 & \texttt{m004(5,2)} & \texttt{4\char`_1(5,2)} & $S^3_{5/2}(4_1)$ & $5$ & $9$ \\
12 & \texttt{m003(-5,3)} & \texttt{K12n242(35,2)} & $S^3_{35/2}(P)$ & $35$ & $35$ \\
13 & \texttt{m007(1,2)} & \texttt{K12n242(21,1)} & $S^3_{21}(P)$ & $21$ & $21$ \\
17 & \texttt{m003(-5,4)} & \texttt{K5\char`_1(30,1)} & $S^3_{30}(k5_1)$ & $30$ & $30$ \\
18 & \texttt{m006(-3,2)} & \texttt{K12n242(15,1)} & $S^3_{15}(P)$ & $15$ & $15$ \\
19 & \texttt{m015(5,1)} & \texttt{5\char`_2(7,1)} & $S^3_7(\mirror{5_2})$ & $7$ & $9$
\end{tabular}
\egroup
\caption{Some manifolds in the Hodgson--Weeks census which are surgeries on knots in $S^3$.  Here $P$ denotes the pretzel $P(-2,3,7)$, and all identifications are up to orientation.}
\label{table:hw-table1}
\end{table}

\begin{proof}
The proof is by computation in SnapPy.  The columns of Table~\ref{table:hw-table1} are as follows: the label refers to the index within the census, and $M$ is the corresponding description within SnapPy as a filling of a cusped hyperbolic manifold, e.g.,
\begin{verbatim}
In[1]: OrientableClosedCensus[1]
Out[1]: m003(-2,3)
\end{verbatim}
The name in $N$ is a SnapPy description of another 3-manifold which can be rigorously proved isometric to $M$ as follows:
\begin{verbatim}
In[2]: M = Manifold("m003(-2,3)")
In[3]: N = Manifold("4_1(5,1)")
In[4]: M.is_isometric_to(N)
Out[4]: True
\end{verbatim}
These computations do not keep track of orientation, and a knot description like \texttt{5\char`_2} may refer to either $5_2$ or its mirror $\mirror{5_2}$, so we need to determine which is which for the column labeled ``Manifold''.  (For the amphichiral $4_1$ this does not matter.)  For \texttt{K12n242}, which is the pretzel $P = P(-2,3,7)$ up to mirroring, we can check that $18$-surgery is cyclic:
\begin{verbatim}
In[5]: Manifold('K12n242(18,1)').fundamental_group()
Out[5]: 
Generators:
   a
Relators:
   aaaaaaaaaaaaaaaaaa
\end{verbatim}
and since $S^3_{18}(P)$ is a lens space but $S^3_{18}(\mirror{P})$ is not, we have $\texttt{K12n242} = P$.  Similarly, \texttt{K5\char`_1} describes a twisted torus knot in the Callahan--Dean--Weeks census of hyperbolic knots \cite{cdw}:
\[ k5_1 = T(5,6)_{2,1} = \textrm{braid closure of } (\sigma_4\sigma_3\sigma_2\sigma_1)^6\sigma_1^2 \in B_5 \]
and \texttt{K5\char`_1(31,1)} has cyclic fundamental group, meaning it must be a positive surgery on the braid-positive knot $k5_1$, so $\texttt{K5\char`_1}=k5_1$.

This leaves only the question of \texttt{5\char`_2}.  In the proof of Proposition~\ref{prop:isharp-5_2-surgeries} we identified $5_2$ as the two-bridge knot $K(2,4)$ and thus used \eqref{eq:surgery-2-bridge-even} to assert that
\[ S^3_{-1}(5_2) = S^3_{-1}(K(2,4)) \cong S^3_{-1/2}(K(2,2)) = S^3_{-1/2}(3_1), \]
which is not hyperbolic since it is a surgery on the left handed trefoil.  Since \texttt{5\char`_2(-1,1)} is hyperbolic, and in fact isometric to $\texttt{m004(1,2)} \cong S^3_{1/2}(4_1)$:
\begin{verbatim}
In[6]: Manifold('5_2(-1,1)').identify()
Out[6]: [m004(1,2)]
\end{verbatim}
we must have $\texttt{5\char`_2} = \mirror{5_2}$.

Finally, we compute the dimension of $I^\#$ for each manifold in Table~\ref{table:hw-table1}  using Theorem \ref{thm:rational-surgeries}, the fact that 
\begin{align*}
(\cinvt(4_1), r_0(4_1)) &= (0, 2), &
(\cinvt(\mirror{5_2}), r_0(\mirror{5_2})) &= (1,3), 
\end{align*}per Table \ref{table:isharp-values},
and the fact that the remaining surgeries are on knots $K \in \{P,k5_1\}$ with a positive instanton L-space surgery, so that $\cinvt(K) = r_0(K) = 2g(K)-1$ by Proposition~\ref{prop:nu-l-space}.  We have $g(P) = 5$ and $g(k5_1) = 11$, so all of the relevant slopes are greater than $2g(K)-1$ and hence these manifolds are instanton L-spaces.
\end{proof}

\begin{proposition} \label{prop:hw-branched}
Census manifolds 0, 3, 8, 14, and 15 are instanton L-spaces.
\end{proposition}

\begin{proof}
We use SnapPy to identify these manifolds as branched double covers of knots $K$ in $S^3$, as shown in Table~\ref{table:hw-table2}.
\begin{table}
\bgroup
\def\arraystretch{1.25}
\begin{tabular}{c|ccc|cc}
\# & $M$ & $K$ & QA? & $|H_1(M)|$ & $\dim I^\#(M)$ \\ \hline
0 & \texttt{m003(-3,1)} & \texttt{9\char`_49} &\checkmark& $25$ & $25$ \\
3 & \texttt{m003(-4,3)} & \texttt{10\char`_155} &\checkmark& $25$ & $25$ \\
8 & \texttt{m003(-4,1)} & \texttt{10\char`_163} &\checkmark& $35$ & $35$ \\
12 & \texttt{m003(-5,3)} & \texttt{10\char`_156} &\checkmark& $35$ & $35$ \\
13 & \texttt{m007(1,2)} & \texttt{10\char`_160} &\checkmark& $21$ & $21$ \\
14 & \texttt{m007(4,1)} & \texttt{K11n118} &?& $21$ & $21$ \\
15 & \texttt{m007(3,2)} & \texttt{9\char`_47} &\checkmark& $27$ & $27$ \\
18 & \texttt{m006(-3,2)} & \texttt{K11n92} &$\times$& $15$ & 15 \\
\end{tabular}
\egroup
\caption{Some manifolds $M$ in the Hodgson--Weeks census which are branched double covers of knots $K$ in $S^3$ for which $\Khoddr(K;\C)$ is thin. ``QA'' means quasi-alternating, and in the remaining cases we used Shumakovitch's \texttt{KhoHo} program \cite{khoho} to prove thinness.}
\label{table:hw-table2}
\end{table}
We can verify each of these examples as follows:
\begin{verbatim}
In[1]: M = Manifold("m003(-4,3)")
In[2]: S = Manifold("10_155(2,0)").covers(2)[0]
In[3]: M.is_isometric_to(S)
Out[3]: True
\end{verbatim}
In each of these cases, the reduced odd Khovanov homology of $K$ over $\C$ is thin, and this means by Proposition~\ref{prop:ss-thin} that \[\dim I^\#(M) = \det(K) = |H_1(M)|,\] so $M$ is an instanton L-space.
\end{proof}

\begin{proposition} \label{prop:census-2-7-16}
Census manifolds 2 and 16 are instanton L-spaces, and census manifold 7 has framed instanton homology of dimension either 10 or 12.
\end{proposition}

\begin{proof}
\begin{table}
\bgroup
\def\arraystretch{1.25}
\begin{tabular}{c|cc|lcc}
\# & $M$ & $|H_1(M)|$ & $M_i$ & $|H_1(M_i)|$ & $\dim I^\#(M_i)$ \\ \hline
2 & \texttt{m007(3,1)} & 18 & $\texttt{m007(1,0)} \cong L(3,1)$ & 3 & 3 \\
&&& $\texttt{m007(2,1)} \cong \dcover(8_{21})$ & 15 & 15 \\ \hline
7 & \texttt{m003(-3,4)} & 10 & $\texttt{m003(-1,1)} \cong L(5,1)$ & 5 & 5 \\
&&& $\texttt{m003(-2,3)}$ & 5 & 7 \\ \hline
16 & \texttt{m006(3,1)} & 30 & $\texttt{m006(1,0)} \cong L(5,2)$ & 5 & 5 \\
&&& $\texttt{m006(2,1)} \cong \texttt{m003(-3,1)}$ & 25 & 25 \\
\end{tabular}
\egroup
\caption{Some surgery triads $(M,M_0,M_1)$ which can be used to compute $I^\#(M)$ for various manifolds $M$ in the Hodgson--Weeks census.}
\label{table:hw-table3}
\end{table}
In each case, given the census manifold $M$, we find two other manifolds $M_0$ and $M_1$ such that $(M,M_0,M_1)$ is a ``surgery triad'': the three manifolds are all Dehn fillings of the same 3-manifold along slopes of pairwise distance one.  Then $I^\#(M)$, $I^\#(M_0)$, and $I^\#(M_1)$ fit into a surgery exact triangle.  These are presented in Table~\ref{table:hw-table3}.

Each $M_i$ which is a lens space can be identified as such in Regina, and then it is an instanton L-space.  We do the same for $\texttt{m007(2,1)} \cong \dcover(8_{21})$ (Regina  identifies both of these as \texttt{SFS [S2: (2,1) (3,2) (3,-1)]}, which is Seifert fibered with finite fundamental group), and this is an instanton L-space of order $\det(8_{21})=15$ because $8_{21}$ is quasi-alternating and hence has thin reduced odd Khovanov homology.  We then use SnapPy  for the remaining identification:
\begin{verbatim}
In[1]: Manifold('m006(2,1)').is_isometric_to(Manifold('m003(-3,1)'))
Out[1]: True
\end{verbatim}
and we already know both $I^\#(\texttt{m003(-2,3)})$ and $I^\#(\texttt{m003(-3,1)})$ because these are census manifolds 1 and 0 respectively.

Having identified $M_0$ and $M_1$ as well as their framed instanton homologies, exactness of the surgery triangle plus equation~\eqref{eq:framed-chi} tells us that
\[ |H_1(M)| \leq \dim I^\#(M) \leq \dim I^\#(M_0) + \dim I^\#(M_1) \]
and this forces $\dim I^\#(\texttt{m007(3,1)}) = 18$ and $\dim I^\#(\texttt{m006(3,1)}) = 30$.  It also implies that $\dim I^\#(\texttt{m003(-3,4)})$ is between 10 and 12, and it cannot be 11 since the Euler characteristic is $10$.
\end{proof}

\begin{proof}[Proof of Theorem \ref{thm:hyperbolic}]
The theorem follows  from Propositions \ref{prop:hw-surgeries}, \ref{prop:hw-branched}, and \ref{prop:census-2-7-16}.
\end{proof}

\section{A comparison with Heegaard Floer homology}
\label{sec:comparison}

As was explained to us by Jen Hom, the Heegaard Floer $\tau$ invariant can also be expressed as the homogenization of a concordance invariant coming from surgeries. This is  central to the proofs of Propositions \ref{prop:tau-tau} and \ref{prop:mutation}. We will explain this in detail below, and then prove Proposition \ref{prop:tau-tau} and Theorem \ref{thm:main-comparison} (we already proved Proposition \ref{prop:mutation} in the introduction, assuming these results).

In \cite[\S9]{osz-rational}, Ozsv{\'a}th--Szab{\'o} defined an integer-valued concordance invariant $\nu(K)$ from the knot Floer filtration associated to $K$. They proved that  \begin{equation}\label{eq:nu-tau}\nu(K)\in\{\tau(K),\tau(K)+1\},\end{equation}
and showed that $\nu(K)$ is related to the Heegaard Floer homology of nonzero rational surgeries on $K$ \cite[Proposition 9.6]{osz-rational}.
 Hanselman later used the immersed curves formulation of bordered Heegaard Floer homology to prove the following (essentially equivalent result) in \cite[Proposition 13]{hanselman-cosmetic}, a direct analogue of our Theorem \ref{thm:main-surgery}:

\begin{proposition}\label{prop:hanselman-surgery}
For every knot $K\subset S^3$,  there are integers $\hat\nu(K)$ and  $\hat r_0(K)$ such that
\[ \dim_\F \hfhat(S^3_{p/q}(K);\F) = q\cdot \hat r_0(K) + |p - q\hat\nu(K)| \]
for all nonzero rational  $p/q$ with $p$ and $q$ relatively prime and $q \geq 1$. \end{proposition}

We  show that $\hat\nu$ and $\nu$ are related as in Lemma \ref{lem:nuhat}, which then implies the relationship between $\hat\nu$ and $\tau$ in Corollary \ref{cor:nuhat-tau}. First, we describe some basic properties of $\hat\nu$.

\begin{remark}Since $S^3_n(K)\cong -S^3_{-n}(\mirror{K})$ for any knot $K$ and any integer $n$, Proposition \ref{prop:hanselman-surgery} implies that $\hat\nu(\mirror{K}) = -\hat\nu(K)$ and $\hat r_0(\mirror{K}) = \hat r_0(K)$ for all $K$.
\end{remark}

\begin{corollary}\label{cor:nu-adjunction}For any knot $K\subset S^3$, we have $|\hat\nu(K)|\leq \max(2g_s(K)-1,1).$
\end{corollary}

\begin{proof} The  map $\hat F_n$ in the surgery exact triangle
\[\dots\to\hfhat(S^3;\F)\xrightarrow{\hat F_n}\hfhat(S^3_{n}(K);\F)\to\hfhat(S^3_{n+1}(K);\F)\to\dots\] induced by the trace of $n$-surgery $X_n(K)$ must vanish for integers \[n \geq \max(2g_s(K)-1,1)\] by the adjunction inequality in Heegaard Floer homology, since $X_n(K)$ contains a surface of genus $\max(g_s(K),1)$ with self-intersection $n$. This implies that \[\dim_\F\hfhat(S^3_{n+1}(K);\F) = \dim_\F\hfhat(S^3_{n}(K);\F)+1\] for all $n$ in this range. Comparing this to Proposition \ref{prop:hanselman-surgery} shows that $\hat\nu(K)\leq \max(2g_s(K)-1,1).$ The same argument applied to $\mirror{K}$ completes the proof of the corollary.\end{proof}

\begin{lemma} 
\label{lem:nuhat}The integer $\hat\nu(K)$ can be expressed in terms of $\nu(K)$ by \[\hat\nu(K) = \begin{cases}
\max(2\nu(K)-1,0),& \nu(K)\geq \nu(\mirror{K})\\
-\max(2\nu(\mirror{K})-1,0),& \nu(K)\leq \nu(\mirror{K}).
\end{cases}\]
\end{lemma}

\begin{proof} 
Since $\hat\nu(\mirror{K})=-\hat\nu(K)$, we may assume without loss of generality that $\nu(K)\geq \nu(\mirror{K})$. Then it follows from \cite[Lemma 9.2]{osz-rational} that $\nu(K)\geq0$. Assume first that $\nu(K)\geq 1$. Then \cite[Proposition 9.6]{osz-rational} implies   that the sequence
\[ \dim_\F \hfhat(S^3_n(K);\F), \qquad n\in\Z \]
is uniquely minimized at $n=2\nu(K)-1$. From Proposition \ref{prop:hanselman-surgery}, we must therefore have $\hat\nu(K) = 2\nu(K)-1$ as desired. If $\nu(K)=0$, then  \cite[Proposition 9.6]{osz-rational} says \[\dim_\F \hfhat(S^3_1(K);\F) =\dim_\F \hfhat(S^3_{-1}(K);\F),\] which forces $\hat\nu=0$ as desired, by Proposition \ref{prop:hanselman-surgery}.
\end{proof}

This immediately implies the corollary below since $\nu$ is a concordance invariant; alternatively, one can prove Corollary \ref{cor:concordance} directly from Proposition \ref{prop:hanselman-surgery} in the same way we proved that $\cinvt$ is a concordance invariant in Theorem \ref{thm:conc-invt}.

\begin{corollary} 
\label{cor:concordance} The integer $\hat\nu(K)$ depends only on the concordance class of $K$.
\end{corollary}

\begin{remark}Since $\hat\nu(U) = 0$ for the unknot $U$, and $\hat\nu(\mirror{K}) = -\hat\nu(K)$ for any $K$, we have that $\hat\nu(K)=0$ for $K$ slice or amphichiral. 
\end{remark}

Moreover, $\hat\nu$ is related to $\tau$ as follows:

\begin{corollary}\label{cor:nuhat-tau}We have $\hat\nu(K)\in\{2\tau(K)-1,2\tau(K),2\tau(K)+1\}$.\end{corollary} 

\begin{proof}
Since $\hat\nu(\mirror{K})=-\hat\nu(K)$ and $\tau(\mirror{K})=-\tau(K)$, we may assume without loss of generality that $\nu(K)\geq \nu(\mirror{K})$. Then $\nu(K)\geq 0$ by \cite[Lemma 9.2]{osz-rational}. 

Assume first that $\nu(K)\geq 1$. Then $\hat\nu(K) = 2\nu(K)-1$ by Lemma \ref{lem:nuhat}, and the corollary follows from \eqref{eq:nu-tau}. 

Now assume $\nu(K) = 0$. Then $\tau(K) = 0$ or $-1$ by \eqref{eq:nu-tau}, and $\hat\nu(K)=0$ by Lemma \ref{lem:nuhat}. If $\tau(K) = 0$, we are done. If $\tau(K) = -1$ then $\tau(\mirror{K})=1$, which implies by \eqref{eq:nu-tau} that $\nu(\mirror{K})\geq 1>\nu(K)$, a contradiction.
\end{proof}

It follows immediately from Corollary \ref{cor:nuhat-tau} that $\hat\nu$ is a quasimorphism from the smooth concordance group to $\Z$, and $\tau$ is one-half of its homogenization, \[\tau(K) = \frac{1}{2}\lim_{n\to \infty} \frac{\hat\nu(\#^nK)}{n}.\] The following is an analogue of Corollary \ref{cor:tau-maximal}, which will help us compute $\hat\nu$.

\begin{corollary}
If $|\tau(K)| = g_s(K)>0,$ then $|\hat\nu(K)| = 2g_s(K)-1$ and $\hat\nu(K)$ has the same sign as $\tau(K)$.
\end{corollary}

\begin{proof}Suppose $\tau(K) = g_s(K)>0.$ Then  we must have $\hat\nu(K) = 2g_s(K)-1$ by Corollaries \ref{cor:nu-adjunction} and \ref{cor:nuhat-tau}. The same argument applied to $\mirror{K}$ completes the proof.
\end{proof}

\begin{remark}
\label{rmk:HFlspaceknot}We further note that $\hat\nu(K) = \hat r_0(K) = 2g(K)-1$ for Heegaard Floer L-space knots, by \cite[Corollary 1.4]{osz-rational}, so the analogue of Theorem \ref{thm:nu-l-space} (and, of course Lemma \ref{lem:t2q-surgery}) holds in Heegaard Floer homology.
\end{remark}

With these results in place, we may now prove Proposition \ref{prop:tau-tau} and Theorem \ref{thm:main-comparison}.

\begin{proof}[Proof of Proposition \ref{prop:tau-tau}]
Suppose \[\dim_\C I^\#(Y;\C) = \dim_\F \hfhat(Y;\F)\] for all $Y$ obtained via integer surgery on a knot in $S^3$. Let $K\subset S^3$ be a knot. Then \[\dim_\C I^\#(S^3_{\pm N}(\#^nK);\C) = \dim_\F \hfhat(S^3_{\pm N}(\#^nK);\F)\] for all positive integers $N$ and $n$. As in Remark \ref{rmk:nufromsurgery}, $\cinvt(\#^nK)$ is completely determined by the former dimensions for $N$ sufficiently large. But it is clear given Proposition \ref{prop:hanselman-surgery} that $\hat\nu(\#^nK)$ is determined by the latter in exactly the  same way. Thus, $\cinvt(\#^n K) = \hat\nu(\#^nK)$ for all positive integers  $n$, from which it follows that $\chominvt(K) = \tau(K)$, by definition.
\end{proof}

\begin{proof}[Proof of Theorem \ref{thm:main-comparison}]
The equality of dimensions \begin{equation}\label{eq:equality-dim}\dim_\C I^\#(Y;\C) = \dim_\F \hfhat(Y;\F)\end{equation} holds for nonzero surgeries on knots which are both Heegaard Floer and instanton L-space knots (in particular, for surgeries on any knot with a lens space surgery), by Remark \ref{rmk:HFlspaceknot}. The rest of the theorem follows from the facts that \eqref{eq:equality-dim} holds for $S^3$; that it holds for $Y=S^3_0(8_8)$; the  formal properties of $\cinvt,\chominvt,\hat\nu,\tau$; that Proposition \ref{prop:hanselman-surgery} is the direct analogue of Theorem \ref{thm:main-surgery}; that both theories admit surgery exact triangles and have the same Euler characteristics; and the fact that there are spectral sequences from (odd) Khovanov homology to the Floer homology of branched double covers in both settings. We spell this out in more detail below.

Theorem \ref{thm:main-twist} followed from Propositions~\ref{prop:twist-2n+2-surgery} and \ref{prop:twist-2k-1-surgery}. Apart from purely topological results, the former only used Theorem~\ref{thm:main-surgery}, together with computations of the dimension of framed instanton homology for surgeries on trefoils and the figure eight (Lemmas~\ref{lem:t2q-surgery} and \ref{lem:figure-eight}), which followed from Corollary~\ref{cor:lens-spaces}, whose proof only required the single computation $\dim I^\#(S^3)=1$.  The latter used these in addition to the fact that $\cinvt=1$ for positive knots of genus 1 (which follows from the fact that $\chominvt = 1=g_s$ for such knots, and Corollary \ref{cor:tau-maximal}). These all have identical analogues in the Heegaard Floer setting, which proves \eqref{eq:equality-dim} for surgeries on these twist knots.

Theorem \ref{thm:main-pretzels} ultimately followed from topological results, Theorem \ref{thm:main-surgery}, the fact that $\cinvt = 0$ for slice knots, and the framed instanton homology of $S^3_{\pm 1}(K_4=6_1)$, which followed from Proposition \ref{prop:twist-2n+2-surgery}. The equality \eqref{eq:equality-dim} thus holds for surgeries on these pretzel knots.

To prove Theorem \ref{thm:8-crossings}, we first computed $\chominvt$ for the knots in Table \ref{table:main} (actually, for all prime knots through 8 crossings) using the fact that all of these knots (or their mirrors) are either alternating or  quasipositive. We then found $\cinvt$ for these knots using the computations of $\chominvt$ together with the fact that $\cinvt = 0$ for slice and amphichiral knots, and  Corollary \ref{cor:tau-maximal} (we also needed Proposition \ref{prop:twist-2n+2-surgery} to pin down $\cinvt(8_1)$). These values of $\chominvt$ and $\cinvt$ therefore agree with their analogues $\tau$ and $\hat\nu$. The computations of $r_0$ for these knots, and thus of the dimension of framed instanton homology of nonzero surgeries on these knots, used only Theorem \ref{thm:main-surgery} together with the computations for surgeries on torus knots, on twist knots (for $6_2, 7_3,7_4, 8_2,8_3,8_4$; see Propositions \ref{prop:isharp-5_2-surgeries} and \ref{prop:dim-8_234}), on pretzel knots (for $8_5$ and $8_{20}$), and the fact that \[\dim_\C I^\#(S^3_0(8_8);\C) = 14\] (for $6_3,8_6,8_8$; see Proposition \ref{prop:6_3-surgeries}).  The same is true of Theorem~\ref{thm:main-pretzels-2}, which relied on the computation for $6_2$ and the fact that $2\chominvt$ is a slice-torus invariant.  All of these things have identical analogues in Heegaard Floer homology, as discussed, except possibly for the calculation of $\dim_\C I^\#(S^3_0(8_8);\C)$. But it is easy to see from the formula in \cite[Theorem 1.4]{osz-alternating} that we also have \[\dim_\F \hfhat(S^3_0(8_8);\F) = 14,\] so the equality \eqref{eq:equality-dim} also holds for nonzero surgeries on the knots in Table \ref{table:main}.

Finally, Theorem \ref{thm:hyperbolic}, which justifies the computations of framed instanton homology for  the census manifolds in Table \ref{table:hw-census}, followed from Propositions \ref{prop:hw-surgeries}, \ref{prop:hw-branched}, and \ref{prop:census-2-7-16}. The first of these propositions, which justifies the computations in Table \ref{table:hw-table1}, follows from the fact that the manifolds in this table are surgeries on $4_1$, $\mirror{5_2}$, or knots which admit lens space surgeries. We therefore have \eqref{eq:equality-dim} for these manifolds. The second, which justifies the computations in Table \ref{table:hw-table2}, follows from the fact that these manifolds are branched double covers of knots with thin odd Khovanov homology. They also have thin mod 2 reduced Khovanov homology, which proves \eqref{eq:equality-dim} for these manifolds, by the spectral sequence \[\Khr(L;\F)\Rightarrow \hfhat(-\Sigma(L);\F)\] \cite{osz-branched}. Finally, Proposition \ref{prop:census-2-7-16}, which justifies the computations in Table \ref{table:hw-table3}, follows for the two manifolds for which we actually can compute $I^\#$ from the fact that these manifolds fit into surgery triads with other instanton L-spaces (which are either lens spaces or branched double covers of quasi-alternating knots). The analogous argument in Heegaard Floer homology proves \eqref{eq:equality-dim} for the same two manifolds.
\end{proof}

\bibliographystyle{alpha}
\bibliography{References}

\end{document}